\documentclass[11pt]{article}
\usepackage[T1]{fontenc}
\usepackage[utf8]{inputenc}

\usepackage{amsmath} 
\usepackage{graphicx} 
\usepackage{geometry}          
\usepackage{subfig} 
\usepackage{amssymb}
\usepackage{amsthm}
\usepackage{epstopdf}
\usepackage{enumerate}
\usepackage{makeidx}
\usepackage{hyperref}
\usepackage{pinlabel}

\hypersetup{
backref=true, 
pagebackref=true,
hyperindex=true, 
breaklinks=true, 
urlcolor= blue, 
linkcolor= blue, 
bookmarks=true, 
bookmarksopen=true, 
pdftitle={Combinatorial Modulus on Boundary of Right-Angled Hyperbolic Buildings}, 
pdfauthor={Antoine Clais}, 
pdfsubject={Mac OS X} 
}

\newtheorem{theo}{Theorem}
\numberwithin{theo}{section}
\newtheorem{prop}[theo]{Proposition} 
\newtheorem{coro}[theo]{Corollary}
\newtheorem{lem}[theo]{Lemma} 
\newtheorem{déf}[theo]{Definition}
\newtheorem{conj}[theo]{Conjecture} 
\newtheorem{question}[theo]{Question} 
\newtheorem{fact}[theo]{Fact} 
\newtheorem{défnot}[theo]{Definition/Notation} 
\newtheorem{exmp}[theo]{Example} 
\newtheorem*{nota}{Notation}

 \newtheorem{rem}[theo]{Remark} 

\newcommand{\R} {\ensuremath{\mathbb{R}}}
\newcommand{\Z} {\ensuremath{\mathbb{Z}}}
\newcommand{\h} {\ensuremath{\mathbb{H}}}
\newcommand{\F} {\ensuremath{\mathcal{F}}}
\newcommand{\U} {\ensuremath{\mathcal{U}}}

\newcommand{\borg} {\ensuremath{\partial \Gamma}}
\newcommand{\bori} {\ensuremath{\partial \Sigma}}
\newcommand{\N} {\ensuremath{\mathbb{N}}}

\newcommand{\s} {\ensuremath{\mathbb{S}}}
\newcommand{\bw} {\ensuremath{\mathcal{M}(\Sigma)}}

\newcommand{\ap} {\ensuremath{\mathcal{A}p(\Sigma)}}
\newcommand{\apo} {\ensuremath{\mathcal{A}p_0(\Sigma)}}
\newcommand{\dialgp} {\ensuremath{\mathcal{D}(\Sigma)}}
\newcommand{\demiesp}[1] {\ensuremath{\mathcal{H}(#1)}}

\newcommand{\proj}[2] {\ensuremath{\mathrm{proj}_{#1} (#2)}}
\newcommand{\apD} {\ensuremath{\mathcal{A}p(\Delta)}}
\newcommand{\dual} {\ensuremath{\mathcal{DG}(\Sigma)}}

\newcommand{\cox} {\ensuremath{(W,S)}}

\newcommand{\modcont}[2] {\ensuremath{\mathrm{Mod}_#1} (#2)}
\newcommand{\modcomb}[2] {\ensuremath{\mathrm{Mod}_#1} (#2)}

\newcommand{\modcombw}[2] {\ensuremath{\mathrm{Mod}^A_#1} (#2)}
\newcommand{\modcombg}[1] {\ensuremath{\mathrm{Mod}_p} (#1,G_k)}
\newcommand{\modcombapg}[1] {\ensuremath{\mathrm{mod}^A_p} (#1,G^A_k)}
\newcommand{\modcombwg}[1] {\ensuremath{\mathrm{Mod}^A_p} (#1,G^A_k)}
\newcommand{\modcombfo} {\ensuremath{\mathrm{Mod}_p} (\mathcal{F}_0,G_k)}
\newcommand{\modcombapfo} {\ensuremath{\mathrm{mod}^A_p} (\mathcal{F}^A_0,G^A_k)}
\newcommand{\modcombwfo} {\ensuremath{\mathrm{Mod}^A_p} (\mathcal{F}^A_0,G^A_k)}
\newcommand{\Cay}[1] {\mathrm{Cay}(\ensuremath{#1})}
\newcommand{\cay} {\ensuremath{\Cay{\Gamma}}}
\newcommand{\Int}[1] {\mathrm{Int}(\ensuremath{#1})}
\newcommand{\di}[1] {\mathrm{dist}(\ensuremath{#1})}
\newcommand{\confdim}[1] {\mathrm{Confdim}(\ensuremath{#1})}
\newcommand{\dc}[1] {\ensuremath{d_c}(\ensuremath{#1})}
\newcommand{\dia}[1] {\mathrm{diam} \ensuremath{\:#1}}
\newcommand{\cv}[1] {\mathrm{Conv} (\ensuremath{#1})}
\newcommand{\Ch}[1] {\mathrm{Ch} (\ensuremath{#1})}
\newcommand{\ch} {\ensuremath{\Ch{\Sigma}}}

   \newcommand{\numerote} [1] {
  \begin {enumerate}
  #1
  \end{enumerate}}
  
   \newcommand{\numeroti} [1] {
  \begin {enumerate} [i)]
  #1
  \end{enumerate}}
  
     \newcommand{\liste} [1] {
  \begin {itemize}#1
  \end{itemize}}

  \makeindex       

\title{Combinatorial Modulus on Boundary of Right-Angled Hyperbolic Buildings}
\author{Antoine Clais\\Laboratoire Paul Painlevé \\ Université Lille 1 \\ 59655 Villeneuve d'Ascq, France\\ 
\texttt{antoine.clais@math.univ-lille1.fr}}
\date{\today}

\begin{document}
   \maketitle

\begin{abstract} In this article, we discuss the quasiconformal structure of boundaries of right-angled hyperbolic buildings using combinatorial tools. In particular we  exhibit  some examples of buildings of dimension 3 and 4 whose boundaries  satisfy the combinatorial Loewner property.  This  property   is a weak version of the   Loewner property. This  is motivated by the fact that the quasiconformal structure of the boundary led to many results of rigidity in hyperbolic spaces since G.D. Mostow. In the case of buildings of dimension 2, many work have been done by M. Bourdon and H. Pajot. In particular, the Loewner property on the boundary  permitted them to prove the quasi-isometry rigidity of right-angled Fuchsian buildings. 
\end{abstract}

\section{Introduction}

\subsection{Starting point}

The origin  of the theory of  modulus of curves in compact metric spaces must be found in the classical theory of quasiconformal maps in Euclidean spaces (see \cite{VaisalaLecture} or \cite{VurionenLecture}). Quasiconformal maps are maps between homeomorphisms and bi-Lipschitz maps. The aim of the classical theory is to describe the regularity of quasiconformal maps in $\R^d$ and to exhibit invariants under these maps. The notion of abstract Loewner space, introduced by J. Heinonen and P. Koskela (see \cite{HeinKoskQConf} or \cite{HeinonenLect}), intends to describe   metric measured spaces whose quasiconformal maps have a behavior of Euclidean flavor.

Moreover, since  G.D. Mostow it is known that the quasiconformal structure of the boundary of a hyperbolic space controls the geometry of the space. It turns out that this idea   extends to the case of Gromov hyperbolic spaces and groups.  Finding a Loewner space as the visual boundary of a Gromov hyperbolic group  has been useful to establish rigidity results about the group (see   \cite{HaissinskyGeomQConf} for a survey on those results). The idea that one wants to use to prove rigidity   is  that  the quasi-isometries of an hyperbolic space are given by the quasisymmetric homeomorphisms of the boundary. The Loewner property makes it possible because  the classes of quasi-Moebius, quasisymmetric and quasiconformal maps are equal in a Loewner space.

It is difficulty to prove that the boundary of a hyperbolic space is a Loewner space. To do so,  one needs to find a measure on the boundary that is optimal for the \emph{conformal dimension}. This quasiconformal invariant has been introduced by P. Pansu in \cite{PansuDimconf}. Finding a measure  that realizes the conformal dimension and even computing the conformal dimension, are very difficult questions   that we can solve in few examples for the moment.

An interesting example of this kind is the work done by M. Bourdon and H. Pajot in  Fuchsian buildings. They proved that the boundary of these buildings are Loewner spaces and then used this structure to prove the quasi-isometry rigidity of these buildings (see \cite{BourdonPajotRigi}).

Buildings are singular spaces introduced by J. Tits to study exceptional Lie groups. Currently  buildings became a topic of interest by themselves.   Among them, right-angled buildings  have been classified by F. Haglund and F. Paulin in \cite{HaglundPaulinImmeubles}. They are equipped with a wall structure and with a   simply transitive group action on the chambers that make them very regular objects.   Fuchsian buildings are right-angled hyperbolic buildings of dimension $2$. In light of the results by M. Bourdon and H. Pajot in dimension $2$, we  have the questions:

\begin{question}
Are higher dimensional right-angled hyperbolic buildings rigid? What are the quasiconformal properties of their boundaries?
\end{question} 

The geometry of  higher dimensional right-angled buildings is very close to the geometry of Fuchsian buildings. This gives hope that these questions may have interesting answers.
However, the methods used in Fuchsian buildings are very specific to the dimension 2. Thus these question are not easy.  In this article, we  will use \emph{combinatorial modulus} for a first approach of the conformal structure  of higher dimensional right-angled buildings.

A major rigidity question related to the    quasiconformal structure on the boundary is the following conjecture due to J.W. Cannon.  

\begin{conj}[{\cite[Conjecture 5.1.]{CanSwenCurv}}]\label{conj Cannon}
If $\Gamma$ is a hyperbolic group and $\borg$ is homeomorphic to $\s^2$, then $\Gamma$ acts geometrically on $\h^3$.
\end{conj} 
 
In particular, this conjecture implies Thurston's hyperbolization conjecture of 3-manifolds. Although Thurston's conjecture is now a theorem by G. Perelman, Cannon's conjecture remains very interesting as it is logically independent of Thurston's conjecture.
 
  The combinatorial modulus have been   introduced by J.W. Cannon in \cite{CannonCombiRiemMapTheo} and  by M. Bonk and B. Kleiner  in  \cite{BonkKleinerQuasiSymParamofSpheres} during the investigation of the quasiconformal structure of the 2-spheres  to approach the conjecture and by P. Pansu in a more general context in \cite{PansuDimconf}. Combinatorial modulus gave birth to a weak version of the Loewner property:   the \emph{Combinatorial Loewner Property} (CLP). One of the feature of  these modulus is that they can be used to characterize the conformal dimension as a critical exponent on the boundary.

Recently M. Bourdon and B. Kleiner (see \cite{BourdonKleinerCLP}) gave examples of boundaries of Coxeter groups that satisfy the CLP but that are not known for satisfying the Loewner property. They used this property to give a new proof of  Cannon's conjecture for Coxeter groups. Some of the methods they used for Coxeter groups can be adapted to the case of right-angled buildings. This was a motivation to investigate higher dimensional right-angled buildings using combinatorial modulus. 
 
\subsection{Main result}
  In this article, we use the combinatorial modulus to investigate the quasiconformal boundary of right-angled hyperbolic buildings. Thanks to  methods in \cite{BourdonKleinerCLP}, we obtain a control of the combinatorial modulus on the  boundary in terms of the curves contained in \emph{parabolic limit sets} (see Section \ref{seccurves inPLS}). Then we introduce a \emph{weighted modulus} on the boundary of the apartments. This allows us to control the modulus in the building by a modulus in the apartment (see Section \ref{secmoduleapartmoduleimmeuble}). For well chosen examples, the  boundary of the apartment  has a lot of symmetries that provide a strong control of the modulus. In particular, we exhibit some examples of hyperbolic buildings in dimension 3 and 4 whose boundaries satisfy the CLP.

\begin{theo}[{Corollary \ref{coroprincip}}]
Let  $D$ be the right-angled dodecahedron in $\h^3$ or the right-angled 120-cell in $\h^4$. Let  $W_D$ be the hyperbolic reflection group generated by  reflections about the faces of  $D$. Let $\Delta$ be the right-angled building of constant thickness $q\geq3$ and of  Coxeter group $W_D$.  Then $\partial \Delta$   satisfies the CLP.  
\end{theo}

Along with this result we also give in Theorem \ref{theomajomodulreciproque}, a characterization of the conformal dimension of the building using a critical exponent computed in an apartment.

 \subsection{Organization of the article} 
 In   Section \ref{secmodudef}, we introduce the combinatorial modulus of curves  in the general setting of compact metric spaces. Then  in Section \ref{secmoduhyper}, we restrict to the case of boundaries of hyperbolic spaces.
 
  After these reminders, we give the main steps and ideas of the proof of Theorem \ref{theoprincip} in Section \ref{sec stepsofproof}. This section is essentially a summary of the article.
  
  Then, in Section \ref{sec LFRAHB}, we describe the geometry  of locally finite right-angled hyperbolic buildings.  

The key notion of parabolic limit sets is introduced in Section \ref{seccurves inPLS} where we study the modulus of curves in parabolic limit sets. This section is based on ideas used in Coxeter groups in \cite[Sections 5 and 6]{BourdonKleinerCLP}. In particular, Theorem \ref{theocourbedanspara} is the first major step towards the proof of Theorem \ref{theoprincip}. As a consequence of this theorem, we obtain a first application to the CLP (Theorem \ref{theoapplic au imm de dim2}).
 
In Section \ref{sec bord topo met}, we describe the combinatorial metric on the boundary of the group in terms of the geometry of the building. This metric is useful in computing the combinatorial modulus. Then, in Section \ref{secmoduleapartmoduleimmeuble} we discuss how the modulus in the boundary of an apartment may be related to a modulus in the boundary of the building. In particular, Theorem \ref{theocontrolimmappart} is the second major step necessary to prove Theorem \ref{theoprincip}. We use this theorem to prove Theorem \ref{theomajomodulreciproque}  which relates the conformal dimension of the boundary of the building to a critical exponent computed in the boundary of an apartment. 

 In Section \ref{sec applictethickeness}, we add the constant thickness assumption for the building under which the results of the preceding section can be made more precise. In particular, we find that the conformal dimension of the boundary of the building is equal to a critical exponent computed in the boundary of an apartment (see Theorem \ref{theomajomodulreciproque}). Finally in Section \ref{sec result}, we gather these tools to obtain examples of right-angled-buildings of dimension 3 and 4 whose boundary satisfies the CLP (see Corollary \ref{coroprincip}).

\subsection{Terminology and notation}
\label{subsecterminota}
 Throughout this paper, we will use the following conventions. The identity element in a group will always be designated by $e$. For a set $E$,  the \emph{cardinality} of $E$ is designated by  $\# E$. A\index{Proper subset} \emph{proper} subset $F$ of $E$ is a subset $ F \varsubsetneq E$.
 
  If $\mathcal{G}$ is a graph then $\mathcal{G}^{(0)}$ is the \emph{set of vertices} of $\mathcal{G}$ and $\mathcal{G}^{(1)}$ is the \emph{set of edges} of $\mathcal{G}$. For $v,w \in \mathcal{G}^{(0)}$, we write $v\sim w$ if there exists an edge in $\mathcal{G}$ whose extremities are $v$ and $w$.  If $V \subset \mathcal{G}^{(0)}$, the \index{Full subgraph} \emph{full subgraph} generated by $V$ is the graph $\mathcal{G}_V$ such that $\mathcal{G}_V^{(0)}=V$ and an edge lies between two vertices $v$, $w$ if and only if there exists an edge between $v$ and $w$ in $\mathcal{G}$. A full subgraph is called a \index{Circuit} \emph{circuit} if it is a cyclic graph $C_n$ for $n\geq3$. A graph is called a \emph{complete graph} \index{Complete graph} if for any pair of distinct vertices $v$,$w$ there exists an edge between $v$ and $w$. 

A \emph{curve} \index{Curve} in a compact metric space $(Z,d)$ is a continuous map $\eta : [0,1] \longrightarrow Z$. Usually, we identify a curve with its image. If $\eta$ is a curve in $Z$, then  $\mathcal{U}_\epsilon(\eta)$ \index{$\mathcal{U}_\epsilon(\eta)$} denotes  the $\epsilon$-neighborhood of $\eta$ for the $C^0$-topology. This means that a curve $\eta' \in \mathcal{U}_\epsilon(\eta)$ if and only if there exists $s:t \in [0,1] \longrightarrow [0,1]$ a parametrization  of $\eta$ such that  for any $t\in  [0,1] $ one has $d(\eta(s(t)),\eta'(t))<\epsilon$.  

 In a metric space $Z$, if $A\subset Z$ then $N_r(A)$ is  the \emph{$r$-neighborhood} of $A$. The \emph{closure} of $A$ is designated by $\overline{A}$ and the \emph{interior} of $A$ by $\Int{A}$. 
 If $B=B(x,R)$ is an open ball  and $\lambda \in \R$ then $\lambda B$ is the ball of radius $\lambda R$ and of   center $x$. A ball of radius $R$ is called an \emph{$R$-ball}. The closed ball of center $x$ and radius $R$ is designated by $\overline{B(x,R)}$.
 
  A \emph{geodesic line} (resp. \emph{ray}) in a metric space $(Z,d)$ is an isometry from $(\R, \vert\cdot - \cdot \vert )$ (resp. $([0,+\infty),\vert\cdot - \cdot \vert ) $ into $(Z,d)$.   The real hyperbolic space (resp. Euclidean space) of dimension $d$ is denoted $\h^d$ (resp. $\mathbb{E}^d$).
 
 \subsection{Acknowledgement}
 This article is part of my Ph.D Thesis realized at Université Lille 1 under the direction of Marc Bourdon. I am very thankful to Marc for his support during these years. I thank Fréderic Haglund for the interest he demonstrated to this work and his fruitful comments. I am also grateful to the reviewers of my Ph.D thesis Mario Bonk and Pierre-Emmanuel Caprace for their attentive reading. Eventually, I would like to thank  María-Dolores Parrilla Ayuso who has the ability to turn mathematics into 3d pictures. 

\tableofcontents 

\section{Combinatorial modulus and the CLP}

\label{secmodudef}

The combinatorial modulus are tools that have been developed to compute modulus of curves  in a metric space without a natural measure. The idea   is to approximate the metric space  with a sequence of finer and finer approximations. Then with these approximations we can construct discrete measures and compute combinatorial modulus. Finally, for well chosen examples  we can check that this sequence of  modulus has a good asymptotic behavior.

In this first section, we present the general theory of combinatorial modulus in compact metric spaces.  We also recall basic definitions and facts about abstract Loewner spaces as they inspired the theory of combinatorial modulus. Most of this section can be found in \cite[Section 2]{BourdonKleinerCLP} to which we refer for details.

In  this section $(Z,d)$ denotes a compact metric space.

\subsection{General properties of combinatorial modulus of curves}
\label{subsecdefapproxcasgen} For $k\geq 0$ and $\kappa > 1$, a \emph{$\kappa$-approximation of $Z$ on scale $k$} \index{Approximation  on scale $k$} is a finite covering $G_k$ \index{$G_k$} by open subsets  such that for any $ v\in G_k$ there exists $z_v\in v$ satisfying the following properties:

\liste{ \item  $B(z_v,\kappa^{-1} 2^{-k}) \subset
v \subset B(z_v,\kappa 2^{-k}) $,
\item $\forall v , w \in G_k$ with $v\neq w$ one has $B(z_v,\kappa^{-1} 2^{-k})\cap B(z_w,\kappa^{-1} 2^{-k})= \emptyset$.}
A sequence $\{G_k \}_{k\geq 0}$ is called a \emph{$\kappa$-approximation of $Z$}.

\begin{exmp} For $k\geq 0$, a  \emph{$2^{-k}$-separated subset} of  $Z$ is a subset $E$ such that $d(z,z')\geq 2^{-k}$ for any $z\neq z'\in E_k$. Since $Z$ is compact any $2^{-k}$-separated subset of  $Z$ is finite. Let $E_k$ be a  $2^{-k}$-separated subset of  $Z$ of maximal cardinality. Then $E_k$ satisfies the following property:
\begin{center}for any $x \in Z$, there exists $z \in E_k$ such that $d(x,z)\leq 2^{-k}$.
\end{center}
The set $ \{B(z,2^{-k})\}_{z\in E_k}$ defines a $2$-approximation at scale $k$ of $Z$.
\end{exmp}

Now we fix the approximation $\{G_k \}_{k\geq 0}$. We construct a discrete measure based on each $G_k$ for $k\geq 0$. Let $\rho: G_k \longrightarrow [0,+\infty)$ be a positive function and $\gamma$ be a curve in $Z$. The \emph{$\rho$-length} of   $\gamma$   is \[L_\rho (\gamma) = \sum\limits_{\gamma\cap v \neq \emptyset } \rho (v).\] For $p\geq 1$, the \emph{$p$-mass} of $\rho$ is  \[M_p(\rho) = \sum\limits_{v\in G_k} \rho (v)^p.\]
Until the end of this subsection $p\geq 1$ is fixed. Let $\F$ be a non-empty set of curves in $Z$. We say that the function $\rho$ is \emph{$\F$-admissible} if $L_\rho(\gamma)\geq1$ 
for any curve $\gamma\in \F$.
 
\begin{déf} The \emph{$G_k$-combinatorial $p$-modulus of $\F$} \index{Combinatorial modulus} is \[ \modcombg{\F}= \inf \{ M_p(\rho)\}\] 
where the infimum is taken over the set of $\F$-admissible functions and with the convention $\modcombg{ \emptyset }=0$. \index{$\modcombg{\cdot}$}

\end{déf}

The following equality is an alternative definition of the modulus:
\[\modcombg{\F} = \inf_\rho \frac{M_p (\rho)}{L_\rho(\F)^p},\] where the infimum is taken over the set of  positive functions  on $G_k$ and  with $L_\rho (\F)=\inf_{\gamma \in \F} L_\rho (\gamma)$.

The next proposition allows us to see the $G_k$-combinatorial $p$-modulus as a weak outer measure on the set of curves of $Z$. Usually, for an outer measure the subadditivity must hold  over countable sets.
 This is useful to get intuition on these  tools.

\begin{prop}[{\cite[Proposition 2.1.]{BourdonKleinerCLP}}] \label{propmodbase}\text{   } 
\numerote{\item  Let $\F$ be a set of curves and $\F' \subset \F$. Then $\modcombg{\F'}\leq \modcombg{\F}$.

\item Let $\F_1, \dots , \F_n$ be  families of curves. Then $\modcombg{\bigcup\limits_{i=1}^n \F_i}\leq \sum\limits_{i=1}^n \modcombg{\F_i}$.}

\end{prop}

A function $\rho:G_k \longrightarrow [0,+\infty)$ is called a \emph{minimal function} for a set of curves $\F$ if $\modcombg{\F}= M_p(\rho)$. Since we only compute finite sums, minimal functions always exist. Combining with a convexity argument, this also provides an elementary control of the modulus as follows. For $\F$  a non-empty set of curves in $Z$ and $k\geq 0$ 
\[ \frac{1}{(\# G_k)^{p-1}}\leq \modcombg{\F} \leq \# G_k.\]

In the rest of this article we   mainly discuss  the curves of $Z$ of diameter larger than a fixed constant. For these curves the following basic property is useful.

\begin{prop} \label{prop minofonctionmini}
Let $\F$ be a non-empty set of curves in $Z$. Assume that there exists $d>0$ such that $\dia{\gamma}\geq d$ for any $\gamma \in \F$. Then for any $\epsilon>0$, there  exists $k_0\geq 0$ such that for any $k\geq k_0$,  there exists an admissible function $\rho:G_k \longrightarrow [0,+\infty)$  such that $\rho (v)  \leq\epsilon$ for any $v\in G_k$.
\end{prop}

\begin{proof}
Let $\gamma \in \F$. We recall that $\kappa$ denotes the multiplicative constant of the approximation $\{G_k \}_{k\geq 0}$. For $k\geq \frac{\log{(\kappa/d)}}{\log{2}}$, as $\dia{\gamma}>d$ the following inequality holds 
\[\# \{v\in G_k : v\cap \gamma \neq \emptyset \}\geq  \frac{d}{\kappa 2^{-k}}.\]
Hence the constant function $\rho : v\in G_k \longrightarrow \frac{\kappa}{d} 2^{-k} \in [0,+\infty)$ is $\F$-admissible.  This finishes the proof.
\end{proof}

A metric space $Z$ is called \index{Doubling metric space} \emph{doubling} if there exists a uniform constant $N$ such that each ball $B$ of radius $r$ is covered by $N$ balls of radius $r/2$. In doubling spaces, the $G_k$-combinatorial $p$-modulus does not depend, up to a multiplicative constant, on the choice of the approximation.

\begin{prop}[{\cite[Proposition 2.2.]{BourdonKleinerCLP}}] \label{propdouble}
Let $(Z,d)$ be a compact doubling metric space. For each $p\geq1$, if $G_k$ and $G'_{k}$ are respectively $\kappa$ and $\kappa'$-approximations, there exists $D=D(\kappa,\kappa')$ such that for any $k\geq 0$\[D^{-1}\cdot \modcombg{\F} \leq  \modcomb{p}{\F, G'_k}\leq D \cdot       \modcombg{\F}.\]
\end{prop}
Usually, we  work with $p\geq 1 $ fixed  and with approximately self-similar spaces (see Section \ref{secmoduhyper}). As these spaces  are doubling,   now  we refer to the \emph{combinatorial  modulus on scale $k$}, omitting $p$ and the approximation.

\subsection{Combinatorial Loewner property}
\label{subsec CLPdef}

In this subsection, we assume that  $(Z,d)$ is a compact arcwise connected doubling metric space. Let $\kappa >1$ and let $\{G_k\}_{k\geq 0}$ denote a $\kappa$-approximation of $Z$. Moreover we fix $p\geq 1 $.

 A compact and connected subset $A\subset Z$ is called a \index{Continuum} \emph{continuum}. Moreover, if  $A$ contains more than one point, $A$ is called a \emph{non-degenerate} continuum. The relative distance between two disjoint non-degenerate continua $A, B \subset Z$ is 
 
 \[ \Delta (A,B) = \frac{\di{A,B}}{\min \{ \dia{A}, \dia{B} \}}. \] 
 
If $A$ and $B$ are two such continua, $\F (A,B)$ denotes the set of curves in $Z$ joining $A$ and $B$ and we write $\modcombg{A,B}:=\modcombg{\F (A,B)}$.  

\begin{déf}

Let $p>1$. We say that $Z$ satisfies the \emph{Combinatorial $p$-Loewner Property} (CLP) \index{Combinatorial  Loewner Property (CLP)} if there exist two increasing functions $\phi$ and $\psi$ on $(0, + \infty)$ with $\lim_{t\rightarrow 0} \psi (t) =0$, such that

\numeroti{ \item for any pair  of disjoint non-degenerate continua $A$ and $B$ in $Z$ and for all $k\geq 0$ with $2^{-k}\leq \min\{\dia{A},\dia{B} \}$ one has:  
\[ \phi (\Delta(A,B)^{-1}) \leq \modcombg{ A,B}, \]
\item for any  pair of open balls $B_1$, $B_2$ in $Z$, with  same center and $B_1\subset B_2$, and for   all $k\geq 0$ with $2^{-k}\leq   \dia{B_1}$ one has:  
\[  \modcombg{\overline{B_1},Z\backslash B_2} \leq \psi (\Delta(\overline{B_1},Z\backslash B_2)^{-1}). \]
}

\end{déf}

As we assume that $Z$ is doubling, thanks to  Proposition \ref{propdouble}, the CLP is independent of the choice of the approximation. As we noticed, the modulus on scale $k$ is an outer measure (in a weak sense) over the set of curves in $Z$. With the previous remarks we can interpret intuitively the two inequalities of the definition as follows:

 \begin{center}
\begin{minipage}[c]{12cm}  
 
\textit{i) there are plenty of curves joining two continua,}

\textit{ii) the amount of curves joining two continua is a decreasing function of the relative distance.}
 
\end{minipage}
\end{center}
We present  examples and properties about the CLP in Subsection \ref{subsec propbaseclp}. 

\subsection{Loewner spaces}
\label{subsec loewnerdef}
Now we define the notion of Loewner space. This notion introduced in \cite{HeinKoskQConf} has inspired the definition of the CLP. Moreover, the proof of many basic properties of   combinatorial modulus are directly inspired by the classical theory of modulus (see \cite{BourdonKleinerCLP}). 

 Now we consider $(X,d,\mu)$ a metric measured space. For  simplicity, we assume that $X$ is compact and that $(X,d,\mu)$ is a \emph{$Q$-Ahlfors-regular space}  \index{Ahlfors-regular space} ($Q$-AR or AR)  for $Q>1$. This means that there exists a constant $C>1$ such that  for any $0<R\leq  \dia{X}$ and any $R$-ball $B\subset X$ one has  \[C^{-1}\cdot R^Q \leq \mu  (B) \leq C\cdot R^Q .\]
Note that under this assumption the measure $\mu$ is comparable to the Hausdorff measure $\mathcal{H}_d$.
 
 Let $\F$ be a set of curves in $X$. A measurable function $f:X\longrightarrow [0,+\infty ($ is said to be $\F$-admissible if for any rectifiable curve $\gamma \in \F$ \[ \int_{\gamma(t)} f(\gamma(t))dt\geq 1.\]
Note that the notion of admissibility  does not use the measure on $X$ but only the metric space structure.
 
\begin{déf} The \emph{$Q$-modulus of $\F$} \index{$Q$-modulus} is \[ \modcont{Q}{\F}= \inf \Big\{ \int_X f^{Q} d\mu\Big\}\] 
where the infimum is taken over the set of $\F$-admissible functions and with the convention that $\modcont{Q}{ \F }=0$ if $\F$ does not contain rectifiable curves. \index{$\modcont{Q}{\cdot}$}
\end{déf}

As before, if $A$ and $B$ are two disjoint non-degenerate continua, $\F (A,B)$ denotes the set of curves in $X$ joining $A$ and $B$. Moreover, we write $\modcont{Q}{A,B}:=\modcont{Q}{\F (A,B)}$. In the literature on quasiconformal maps the pair $(A,B)$ is called a \emph{condenser} and the modulus (with respect to the Lebesgue measure)   $\modcont{Q}{A,B }$ the \emph{capacity} of $(A,B)$ (see \cite{VurionenLecture}). 

In $X$, the classical modulus are comparable to the combinatorial modulus in the following sense.
 
\begin{prop}[{\cite[Prop B.2]{HaissinskyEmpilCercles}}]
Assume that   $X$ is  equipped with an approximation $\{G_k\}_{k\geq 0}$. For $d_0>0$, let $\F_{0}$ be the set of curves in $X$ of diameter larger than ${d_0}$. Then   for $k$ large enough one has 
\[\modcomb{Q}{\F_{0},G_k}\asymp \modcont{Q}{\F_{0}}, \]
if $\modcont{Q}{\F_{0}}>0$ and $\lim\modcomb{Q}{\F_{0},G_k} = 0$  if $\modcont{Q}{\F_{0}}=0$.

 In addition for any pair $A,B$ of non-degenerate disjoint continua and for $k$ large enough one has 
\[ \modcomb{Q}{A,B,G_k}\asymp \modcont{Q}{A,B}  \]
if $\modcont{Q}{A,B}>0$ and $\lim \modcomb{Q}{A,B,G_k} = 0$ if $\modcont{Q}{A,B}=0$.

\end{prop}

Note that this connection between combinatorial and classical modulus is only valid for the dimension $Q$.

Now we can define Loewner spaces.

\begin{déf}
  We say that $(X,d,\mu)$ is a \emph{$Q$-Loewner space}\index{Loewner space} (or satisfies the \emph{$Q$-Loewner property}) if there exists an increasing function $\phi: (0, + \infty) \longrightarrow (0, + \infty) $   such that for any pair  of non-degenerate disjoint continua $A$ and $B$ in $X$   one has: 
\[ \phi (\Delta(A,B)^{-1}) \leq \modcont{Q}{ \F(A,B)}. \]
\end{déf}
We also say that $X$ satisfies the \emph{Loewner property} or the \emph{classical Loewner property} to avoid the confusion with the CLP.

The control of the modulus from above is not required  in this definition because it is automatically  provided by the structure of $Q$-AR space.
\begin{theo}[{\cite[Lemma 3.14.]{HeinKoskQConf}}]
There exists a constant $C>0$ such that the following property holds. Let $A$ and $B$ be two non-degenerate disjoint continua. Let $0<2r<R$ and $x\in X$ be such that $A\subset\overline{B(x,r)}$ and $B\subset X\backslash B(x,R)$. Then 
\[\modcont{Q}{A,B}\leq C\Big(\log\frac{R}{r}\Big)^{1-Q} .\]
\end{theo}
 
As a consequence there exists  an increasing function  $\psi$ on $(0, + \infty)$ with $\lim_{t\rightarrow 0} \psi (t) =0$, such that for any pair  of disjoint non-degenerate continua $A$ and $B$ \[\modcont{Q}{A,B}\leq\psi(\Delta^{-1}(A,B)).\]
More precisely, there exist some constants $K, C>0$ such that $\psi(t)=K \Big(\log(\frac{1}{t}+C)\Big)^{1-Q}$ for any $t>0$. 

When $X$ is a Loewner space,  the asymptotic behavior of $\phi$ is described in \cite[Theorem 3.6.]{HeinKoskQConf}. For $t$ small enough $\phi(t)\approx \log\frac{1}{t}$, for $t$ large enough $\phi(t)\approx (\log t)^{1-Q}$. 

As we will see in the sequel an essential difference between the combinatorial and the classical modulus property is the importance  of the dimension $Q$ in the discussions about classical modulus.

\subsection{First properties and examples} 
\label{subsec propbaseclp}
In this section $Z$    is a compact arcwise connected doubling metric space  and $X$ is a compact $Q$-AR metric   space  ($Q> 1$).
First we recall a theorem and a conjecture that compare the CLP and the classical Loewner property.

\begin{theo}[{\cite[Theorem 2.6.]{BourdonKleinerCLP}}]\label{theo loewnervers clp}
 If $X$ is a compact  $Q$-AR and  Loewner metric space, then $X$ satisfies the combinatorial $Q$-Loewner property.
\end{theo}
 
The next conjecture is a main motivation for studying group boundaries that  satisfy the CLP.
\begin{conj} [{\cite[Conjecture 7.5.]{KleinerAsymptoticGeom}}]\label{conj CLPloewner}
Assume that $Z$ is quasi-Moebius homeomorphic to the boundary of a hyperbolic group. If $Z$ satisfies the CLP then it is quasi-Moebius homeomorphic to a Loewner space.
\end{conj}

As announced we want to find and use the CLP on boundaries of hyperbolic groups. Quasi-isometries between hyperbolic spaces extend to quasi-Moebius homeomorphisms between the boundaries, so it is fundamental to know how the Loewner property and the CLP behave under quasi-Moebius maps. These maps have been introduced by J. Vaïsälä in \cite{VaisalaQMmaps}. We recall that in a metric space $(Z,d)$ the \emph{cross-ratio} of four distinct points  $a,b,c,d \in Z$ is \[[a:b:c:d] = \frac{d(a,b)}{d(a,c)}\cdot\frac{d(c,d)}{d(b,d)}.\]
For  $Z,Z'$  two metric spaces, an homeomorphism $f : Z \longrightarrow Z'$ is \emph{quasi-Moebius} if there exists an homeomorphism $\phi :[0,+\infty) \longrightarrow [0,+\infty)$ such that for any quadruple of distinct points  $a,b,c,d \in Z$  \[[f(a):f(b):f(c):f(d)] \leq \phi( [a:b:c:d ]).\]
If $f$ is quasi-Moebius, as   $[a:c:b:d]=[a:b:c:d]^{-1}$,    $f^{-1}:Z'\longrightarrow Z$ is also quasi-Moebius.

\begin{theo}[{\cite[Theorem 2.6.]{BourdonKleinerCLP}}]\label{theo CLP QM}
 If $Z'$ is quasi-Moebius homeomorphic to a compact space $Z$ satisfying the  CLP, then $Z'$ also satisfies the  CLP (with the same exponent).
\end{theo}

The Loewner property does not behave so well under quasi-Moebius maps. In particular, 
it is perturbed by a change of dimension whereas the CLP is not.

\begin{theo}[\cite{TysonQCandQS}]
Let $X$ and $X'$ be respectively $Q$-Loewner and $Q'$-AR compact metric spaces. Assume that $X'$ is quasi-Moebius homeomorphic to $X$. Then $X'$ is a Loewner space if and only if $Q=Q'$.
\end{theo}

If we apply to $X$ a snowflake transformation  $f_\epsilon:(X,d)\longrightarrow (X,d^\epsilon)$, $0<\epsilon<1$ then $\dim_\mathcal{H} (X,d^\epsilon) =  \frac{1}{\epsilon} \dim_\mathcal{H} (Z,d)$. Such a transformation is a quasi-Moebius homeomorphism and Combining with the previous theorem we get the following fact.

\begin{fact}
The Loewner property is not invariant under quasi-Moebius homeomorphisms.
\end{fact}

  However quasi-Moebius maps are the appropriate  homeomorphisms to discuss between Loewner spaces.
  
\begin{theo}[\cite{HeinKoskQConf}]
Let $X$ and $X'$ be compact $Q$-regular Loewner spaces
 and let $f:X\longrightarrow X'$ be a homeomorphism. The following are equivalent
\numerote{
\item $f$ is quasi-Moebius,

\item there exists $C>1$ such that 

\[C^{-1}\cdot\modcont{Q}{\F} \leq \modcont{Q}{f(\F)} \leq C\cdot \modcont{Q}{\F} \]

for any set of curves $\F$ in $X$.

} 

\end{theo}

The next proposition gives examples of spaces that do not satisfy the CLP.
 
\begin{prop}[{\cite{HeinKoskQConf} or \cite[Theorem 2.6.]{BourdonKleinerCLP}}]
Assume that $Z$ satisfies the CLP   then it has no local cut point, \emph{i.e} no connected open subset is disconnected by removing a point. \end{prop} 

Combining with the theorem of  Bowditch (see \cite{BowditchCutPoints}) this proposition  says that the boundary of a one-ended hyperbolic group that splits along a two-ended subgroup does not satisfy the CLP.

The first examples of spaces that satisfy the CLP are provided by   Theorem \ref{theo loewnervers clp} and by known examples of Loewner spaces. The next examples are provided by \cite{BourdonKleinerCLP}.

\begin{exmp} \label{ex loewner}

\text{ }

\numerote{

\item The following spaces are Loewner spaces \numeroti{
\item the Euclidean space  $\R ^d$ for $d\geq 2$, this result is due to C. Loewner for $d\geq3 $ (see \cite{LoewnerConfCapa}), 
\item any compact Riemannian manifold modeled by  $\R ^d$ for $d\geq 2$ (see \cite{HeinKoskQConf}),  
\item visual boundaries of right-angled Fuchsian buildings (see \cite{BourdonPajotPoinc}).

}

\item The following spaces satisfy the CLP (see \cite{BourdonKleinerCLP})
\numeroti{ 

\item the Sierpiński carpet and the $n$-dimensional Menger  sponge embedded in the Euclidean space,
\item  boundaries of Coxeter groups of various type: simplex groups, some prism groups, some highly symmetric groups and some groups with planar boundary.

}

 }

\end{exmp}

\begin{figure}[h]
 
\centering
\includegraphics[scale=0.17]{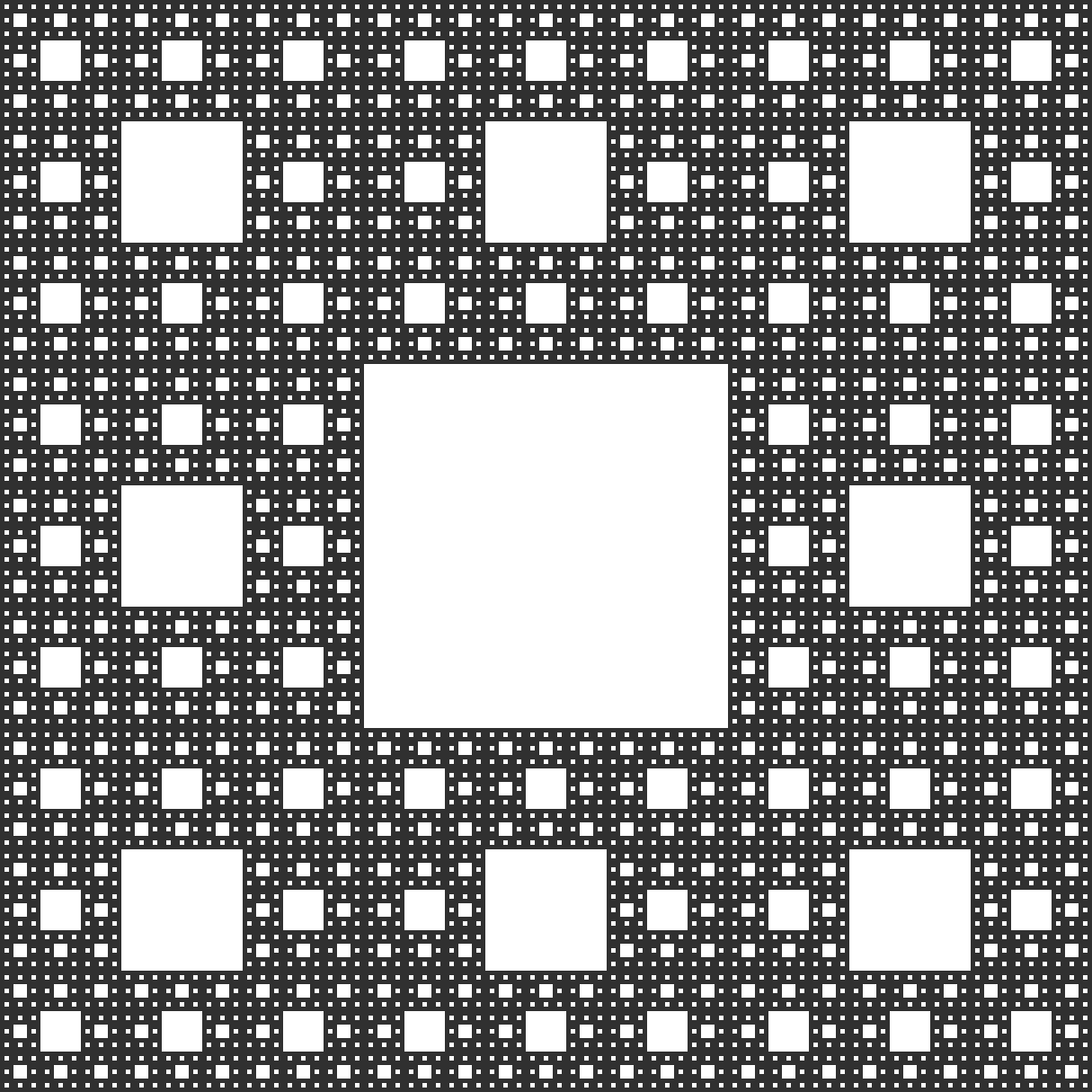}
\caption{The Sierpiński carpet in $\mathbb{E}^2$ satisfies the CLP}
\label{fig:D18}
\end{figure}

For   Examples \ref{ex loewner}.2, we do not know   if they are   Loewner spaces. This provides a first kind of interesting questions.  

\begin{question}
 Is any of the examples in \ref{ex loewner}.2. quasi-Moebius homeomorphic to a Loewner space?
\end{question}

Among these examples the Sierpiński carpet is the one that should be studied first as it should be the easiest one.

Note that Example \ref{ex loewner}.$1.ii)$ provides many examples of hyperbolic groups whose boundaries are Loewner spaces. Indeed  consider a group $\Gamma$ that is acting geometrically (see Subsection \ref{subsechypergroups}) on the standard hyperbolic  space $\h^d$ for $d\geq 3$. Then  $\borg$ is quasi-Moebius equivalent to $\s^{d-1}$ equipped with the standard spherical metric. Hence with Example \ref{ex loewner}$.1.ii)$, $\borg$ is a Loewner space.

Now we consider the following general question: how the geometry of a hyperbolic space is determined by its boundary? The Cannon's conjecture (see Conjecture \ref{conj Cannon}) is a question of this type. The notions of Loewner property and CLP   have been fruitfully used by M. Bonk and B. Kleiner  to approach this conjecture.   By a theorem of D. Sullivan in \cite{SullivanDiscConfGrpsandDyna}, Cannon's conjecture is equivalent to the following in which  the quasiconformal structure of the boundary is the main point.
 
\begin{conj}[{\cite[Conjecture 1.3.]{BonkKleinerQuasiSymParamofSpheres}}]
If $\Gamma$ is a hyperbolic group and $\borg$ is homeomorphic to $\s^2$, then it is quasi-Moebius homeomorphic to the standard 2-sphere.
\end{conj}
 
 If Conjecture \ref{conj CLPloewner} is true, the CLP would provide many interesting examples of Loewner spaces. This motivates our second question.  

\begin{question}
 Can we find new examples of compact metric spaces satisfying the CLP?
\end{question}
 
In this article, we exhibit examples of boundaries of right-angled buildings of dimension 3 and 4 that satisfy the CLP. These buildings  have been studied by  J. Dymara and D. Osajda who described the topology of the boundary. 
 \begin{theo}[\cite{DymOsaj}]
Let $\Delta$ be a right-angled thick building whose associated Coxeter group is a cocompact reflection group in $\h^d$. Then $\partial \Delta$ is homeomorphic to the universal $(d-1)$-dimensional Menger space $\mu^{d-1}$.
\end{theo} 
Our interest in these buildings is inspired by  Examples \ref{ex loewner}$.1.iii)$ and \ref{ex loewner}$.2.ii)$. We will use some ideas from  \cite{BourdonPajotRigi} and \cite{BourdonKleinerCLP} where these examples are studied. 
 
\section{Combinatorial modulus of curve families on boundaries of hyperbolic groups}
\label{secmoduhyper}
Boundaries of hyperbolic groups are naturally endowed with   metric structures that satisfy a property of self-similarity. 
This property implies Proposition \ref{propscale} that will be used several times in proving the main theorem. Intuitively, we can say that this proposition is a tool to enlarge sets of small curves while controlling the modulus.

In this section,  we explain how the  boundary of a hyperbolic group can be seen  as approximately self-similar spaces. Then, we recall the connection between the combinatorial modulus and the conformal dimension of the boundary. Finally, we give a sufficient condition for the boundary   to satisfy the CLP. 

 Most of this section is a review of \cite[Section 3 and 4]{BourdonKleinerCLP} to which we refer for details.

\subsection{Boundaries of hyperbolic groups and  approximately self-similar spaces}
\label{subsechypergroups}
For  details concerning hyperbolic groups and spaces, we refer to \cite{CoorDelPapa}, \cite{GhysHarpe} or \cite{KapoBenak}.
Let $(X,d)$ be a proper geodesic metric space. We say that a finitely generated  group  $\Gamma$ \emph{acts geometrically} on $X$, if:\liste{
\item $\Gamma < \text{Isom}(X)$,
\item $\Gamma$ acts cocompactly,
\item $\Gamma$ acts properly discontinuously.
}

 We say that $X$ is \emph{hyperbolic} (in the sense of Gromov) if there exists a constant $\delta>0$ such that for every geodesic triangle $ [x,y ] \cup  [y,z ] \cup   [z,x ] \subset X$ and every $p\in  [x,y ]$, one has \[\di{p,  [y,z ] \cup   [z,x ]}\leq \delta.\]
A  finitely generated group that acts geometrically on a   hyperbolic space $X$   is called a \emph{hyperbolic group}. In this case, the Cayley graph of a hyperbolic group is a hyperbolic space. 
 
 From now on, let $X$ be a hyperbolic space with a fixed base point $x_0$ and let $\Gamma$ be a hyperbolic group acting geometrically on $X$. Let $\partial X$ be the set of equivalence classes of geodesic rays where two geodesic rays $\gamma, \gamma' : [0,+\infty) \longrightarrow X$ are equivalent if and only if: 
 \begin{center}
\begin{minipage}[c]{12cm}
   there exists $K>0$ such that $d(\gamma(t),\gamma'(t)) \leq K$ for any $t\in [0,+\infty) $.
\end{minipage}
\end{center}
   Thanks to the hyperbolicity condition, we can restrict to the set of geodesic rays starting from $x_0$. We can equip $\partial X$ and $X\cup \partial X$ with   topologies which make them compact spaces. In this setting  $X$, is dense in $X\cup \partial X$ and $\partial X$ is called the \emph{boundary of $X$}. Moreover we can equip  $\partial X$ with a  family of \emph{visual metric}. A metric $\delta(\cdot,\cdot)$  is visual if there exist two constants $A\geq 1$ and $\alpha >0$ such that for all  $\xi,\xi' \in \partial X$: \[A^{-1}e^{-\alpha \ell}\leq \delta (\xi,\xi') \leq A\:e^{-\alpha \ell},\] where $\ell$ is the distance from $x_0$ to a geodesic line $(\xi,\xi')$. In such a situation we also write \[ \delta (\xi,\xi') \asymp e^{-\alpha \ell}. \index{$\asymp$}\]

 The action of $\Gamma$ on $X$ extends naturally to $(\partial X, \delta)$ and elements of $\Gamma$ are uniform quasi-Moebius homeomorphisms of the boundary.
  Moreover, if $\partial \Gamma$ is also equipped with a visual metric, the homeomorphism $\partial \Gamma \longrightarrow \partial X$ induced by the orbit map $g\in \Gamma \longrightarrow g x_0 \in X $ is quasi-Moebius.

 The following definition is a generalization of the classical notion of self-similar space.

 \begin{déf}
\label{defselfsimgene} 
  A compact metric space $(Z,d)$ is called \emph{approximately self-similar} \index{Approximately self-similarity} if there exists a constant $L\geq 1$ such that for every $r$-ball $B$ with $0<r< \dia{Z}$, there exists an open subset $U\subset Z$ which is $L$-bi-Lipschitz homeomorphic to the rescaled ball $(B,\frac{1}{r}d)$. 
 \end{déf}
  
 The definition and proposition that follow  say that the boundary of a hyperbolic group is an approximately self-similar space   and that this structure is linked to the action of the group on its boundary.

\begin{déf} \label{defselfsimbord}
Let  $\Gamma$ be a hyperbolic group.  A metric $d$ on $\partial \Gamma$ is called a \emph{self-similar metric} \index{Self-similar metric} if there exists a hyperbolic space $X$ on which $\Gamma$ acts geometrically, such that $d$ is the preimage of a visual metric on $\partial X$ by the canonical quasi-Moebius homeomorphism $\partial \Gamma \longrightarrow \partial X$.
\end{déf}

\begin{prop}[{\cite[Proposition 3.3.]{BourdonKleinerCLP}}] \label{propdefselfsim}
The space $\borg$ equipped with a self-similar metric is doubling and is an approximately self-similar space, the partial bi-Lipschitz maps being restrictions of group elements. Moreover, $\Gamma$ acts on $(\partial \Gamma,d)$ by (non-uniformly) bi-Lipschitz homeomorphisms. 
\end{prop} 

\subsection{Combinatorial modulus and conformal dimension}
\label{subsecdimconf}
In this subsection we present  the connection between combinatorial modulus and the conformal dimension in approximately self-similar spaces.

Let  $Z$ be an arcwise connected approximately self-similar metric space. For us $Z$ will be the boundary of a hyperbolic group. Let $\{G_k\}_{k\geq0}$ be a $\kappa$-approximation of $Z$ and $d_0$ be a small constant compared with $\dia{Z}$ and with the constant of self-similarity.

 Let  $\F_0$ \index{$\F^0$} denote the set of all the curves in $Z$ of diameter larger than $d_0$.    In \cite{BourdonKleinerCLP}, it is proved that   the properties  of the combinatorial modulus are contained in the asymptotic behavior of $\modcombfo$. This point is explained in Subsection \ref{subsec howtoprove}. Here we write   $M_k:=\modcombfo$.
 
 The following proposition allows to define a critical exponent that is related to the conformal dimension of $Z$.

\begin{prop} \label{propdefexpocrit}
There exists $p_0\geq 1$ such that  for $p\geq p_0$ the modulus $M_k$ goes to zero as $k$ goes to infinity.
\end{prop}

\begin{proof}
Let   $\{G_k\}_{k\geq0}$ be a $\kappa$-approximation of $Z$. According to the doubling condition and the definition of an approximation, there exists an integer $N'$ such that each element $v\in G_k$ is covered by at most $N'$ elements of $G_{k+1}$. As a consequence, if   $K>0$ is the cardinality $G_0$, then 
\[\# G_k \leq K \cdot N'^k \text{ for any } k \geq 1.\] 

Moreover, as we saw in the proof of Proposition \ref{prop minofonctionmini}, there exists a constant $K'>0$ such that the constant function $\rho : v\in G_k \longrightarrow \rho(v)=K'\cdot 2^{-k}\in [0,+\infty)$ is $\F_0$-admissible.

As a consequence \[M_k \leq C \cdot (\frac{N'}{2^p})^k,\]
where $C$ is a positive constant. Thus, for $p$ large enough, $M_k$ goes to zero.
\end{proof}

According to this proposition the next definition makes sense.

\begin{déf} \label{defexpocrit}
The \emph{critical exponent $Q$} \index{Critical exponent}  \index{$Q$} associated with the curve family $\F_0$ is defined as follows  \[Q = \inf \{p\in [1, +\infty) : \lim_{k\rightarrow  + \infty} \modcombfo = 0\}.\]
\end{déf}

We notice that for $k\geq  0$ fixed  the function  $:p \longmapsto \modcombfo$ is non-increasing. Hence $ \modcombfo$ goes to zero  for $p>Q$.

 This critical exponent   is related to the conformal dimension, which provides another motivation to study combinatorial modulus. The conformal dimension has been introduced by P. Pansu in \cite{PansuDimconf}. It is an important property of the conformal structure of the boundary of a hyperbolic group. In particular, it is invariant under quasi-Moebius maps.
 
In the following, $\mathcal{H}_d(\cdot)$  and  $\dim_\mathcal{H} (Z,d)$ respectively  denote the Hausdorff measure and  the Hausdorff dimension of  $Z$ equipped with $d$. 
The \emph{Ahlfors-regular conformal gauge} of $(Z,d)$ is defined as follows
\[\mathcal{J}_c(Z,d):= \{ (Z',\delta) : (Z',\delta) \text{ is AR and quasi-moebius homeomorphic to }  (Z,d)\}. \]

\begin{déf}
Let $(Z,d)$ be a compact metric space. The \emph{Ahlfors-regular conformal dimension} \index{Conformal dimension (Ahlfors-regular)} of  $(Z,d)$ is  
\[\confdim{Z,d} := \inf\{\dim_\mathcal{H}(Z',\delta) : (Z',\delta) \in   \mathcal{J}_c(Z,d)  \}  . \]
\end{déf}
In the rest of the  paper we will simply call it the \emph{conformal dimension}. 

In  \cite{KeithKleiner}, S. Keith and B. Kleiner proved that the combinatorial modulus are related to the conformal dimension. The proof of the following theorem may be found in \cite{Carra}.

\begin{theo}[{\cite{KeithKleiner} or \cite[Theorem 4.5.]{Carra}}] \label{theocarrasco}
The critical exponent $Q$ (see Definition  \ref{defexpocrit}) is equal to $\confdim{Z,d}$.
\end{theo}
The definition of the conformal dimension, Combining with basic topology give the following inequalities:
\[\dim_T (Z)\leq \confdim{Z,d}\leq \dim_\mathcal{H}(Z,d),\]
where $\dim_T (Z)$ is the topological dimension of $Z$. 

The following theorem due to J. Tyson makes a connection between the conformal dimension and the Loewner property.

 \begin{theo}[{\cite[Corollary 4.2.2.]{MackayTysonConfDim}}] \label{theo Tyson} Let $Q>1$ and  $X$ be a $Q$-regular and $Q$-Loewner space, then $\confdim{X}=Q$. 
 \end{theo}
 
 \begin{exmp}
 It has been proved   independently by  B. Kleiner and in \cite{KeithLaaksoConfAssDim}  that the Euclidean metric on the Sierpiński carpet does not realize the conformal dimension. As a consequence the Sierpiński carpet equipped with this metric  does not satisfy the Loewner property. However it satisfies the CLP (see Example \ref{ex loewner}).
 \end{exmp}

Again, Cannon's conjecture has been an important motivation for studying the conformal dimension of the boundary of a hyperbolic group. In particular in \cite{BonkKleinerConfDimGromHypergrps} it is proved that  Conjecture \ref{conj Cannon} is equivalent to the following.

\begin{conj}
If $\Gamma$ is a hyperbolic group and $\borg$ is homeomorphic to $\s^2$, then $\confdim{\borg}$ is attained by a metric in $\mathcal{J}_c(\borg)$.

\end{conj}

\subsection{How to prove the CLP}
\label{subsec howtoprove}

Now we give the sufficient condition that will be used in   this  article to exhibit some examples of groups with a boundary that satisfies the CLP. 
 
Let  $Z$ be an arcwise connected approximately self-similar metric space and let $\{G_k\}_{k\geq0}$ be a $\kappa$-approximation of $Z$.

 The following proposition says that the combinatorial modulus are preserved by the bi-Lipschitz homeomorphism coming from the approximately self-similar structure. This proposition will be used several times in   Sections \ref{seccurves inPLS} to \ref{sec result} to compare the modulus of a set of small curves with the modulus of a set of   curves of diameter larger than a fixed constant.  
\begin{prop} \label{propscale}
 Let $B$ be a ball in $\borg$ such that $\dia{B}<1$. Let $g\in \Gamma$ be the local $L$-bi-Lipschitz homeomorphism that rescales $B$ (given by   Definition \ref{propdefselfsim}). Let $\F$ be a set of curves contained in $\lambda B$ for $\lambda<1$. Then there exist  $\ell \in \N$, and $D>1$  depending only on $L$ and on the doubling constant of $\borg$ such that the following property holds. 
 \begin{center}
\begin{minipage}[c]{12cm}
 If $k\geq 0$ is large enough so that \[\{v \in G_k : \gamma \cap v \neq \emptyset \text{ for some } \gamma\in \F \} \subset \{v \in G_k :   v \subset B \},\] 
 then  
 \[D^{-1} \cdot \modcomb{p}{g\F,G_{k}}\leq \modcomb{p}{\F,G_{k+\ell}} \leq D \cdot \modcomb{p}{g\F,G_{k}},\] where $g\F=\{g\gamma : \gamma \in \F\}$.
\end{minipage}
\end{center}

\end{prop}
\begin{proof}

Let $k\geq 0$ be large enough so that, if $\gamma \cap v \neq \emptyset$ with $\gamma\in \F$ and $v\in G_k$, then $v\subset B$. Let $d=\dia{B}$ and let $\ell\in \N$ denote the  integer satisfying  $2^{-(\ell+1)}<d\leq 2^{-\ell}$. Let $v\in G_{k+\ell}$ such that $v \subset B$ and assume
\[B(\xi,\kappa^{-1} 2^{-(k+\ell)}) \subset v \subset B(\xi,\kappa 2^{-(k+\ell)}). \]
Then \[B(g \xi,(L\kappa)^{-1} 2^{-k}) \subset g v \subset B(g \xi,4L\kappa 2^{-k}). \]
We write $G'_k\cap gB =\{gv : v\in G_{k+\ell}, v \subset B \}$. Then  $G'_k\cap gB$  is a $\kappa'$-approximation of $gB$  on scale $k$, with $\kappa'$ depending only on $\kappa$ and $L$. As the curves of $\F$ are strictly contained in $B$ we obtain the following equality 

\[\modcomb{p}{\F, G_{k+\ell} }= \modcomb{p}{g\F,G'_k\cap gB}. \]    Thanks to   Proposition \ref{propdouble}, there exists $D>1$ such that 
  \[D^{-1} \cdot \modcomb{p}{g\F,G_{k}}\leq \modcomb{p}{g\F,G'_k\cap gB} \leq D \cdot \modcomb{p}{g\F,G_{k}},\]
  and the proposition follows.
\end{proof}

Now we fix  $d_0>0$ a small constant compared with $\dia{Z}$ and with the constant of self-similarity. More precisely it must be small enough so that any non-constant curve in $Z$ may be rescaled to a curve of diameter larger than $d_0$ by self-similarity.  For us, $Z$ will be the boundary of a hyperbolic group and $d_0$ will depend on the hyperbolicity constant. In the following,  $\F_0$ \index{$\F^0$} is the set  of all the curves in $Z$ of diameter larger than $d_0$.   

 Again, we use the letter $Q$ to designate the critical exponent of Definition \ref{defexpocrit}. We recall that if $\eta$ is a curve in $Z$, then  $\mathcal{U}_\epsilon(\eta)$ \index{$\mathcal{U}_\epsilon(\eta)$} denotes  the $\epsilon$-neighborhood of $\eta$ for the $C^0$-topology. This means that a curve $\eta' \in \mathcal{U}_\epsilon(\eta)$ if and only if there exists $s:t \in [0,1] \longrightarrow [0,1]$ a parametrization  of $\eta$ such that  for any $t\in  [0,1] $ one has $d(\eta(s(t)),\eta'(t))<\epsilon$.  
 
The following proposition gives the sufficient conditions for $Z$ to satisfy the CLP that will be used in proving  the main theorem. 
 
\begin{prop}[{\cite[Proposition 4.5.]{BourdonKleinerCLP}}] \label{propcritCLP}
Let  $Z$ be an arcwise connected approximately self-similar metric space.   Let $\{G_k\}_{k\geq0}$ be a $\kappa$-approximation of $Z$ and $d_0$ be a small constant compared with $\dia{Z}$ and with the constant of self-similarity. Let  $\F_0$ \index{$\F^0$} denote the set of all the curves in $Z$ of diameter larger than $d_0$.

For $p=1$, we assume that $\modcombfo$ is unbounded. For $p\geq 1$, we assume that for every non-constant curve $\eta \subset Z$ and every $\epsilon>0$, there exists $C=C(p,\eta,\epsilon)$ such that for every $k \in \N$: \[ \modcombfo \leq C\cdot \modcomb{p}{\mathcal{U}_\epsilon(\eta),G_k}. \] 
Suppose furthermore when $p$ belongs to a compact subset of $[ 1, + \infty)$ the constant  $C$ may be chosen independent of $p$. Then $Z$ satisfies the combinatorial $Q$-Loewner property.
\end{prop}
 
\section{Steps in the proof of Theorem \ref{theoprincip}} 
\label{sec stepsofproof}
 
Before going into details about boundaries of right-angled buildings, we give a sketch of proof of the main theorem. In this  section, $D$ is the right-angled dodecahedron in $\h^3$ or the right-angled 120-cell in $\h^4$. We write  $W_D$ for the hyperbolic reflection group generated by the faces of  $D$. The main theorem of this  article may be stated as follows.

\begin{theo}[{Corollary \ref{coroprincip}}]\label{theointrotheoprincip}
 Let  $\Gamma$ be the graph product of constant thickness   $q\geq3$  and of  Coxeter group $W_D$.  Then $\partial \Gamma$ equipped with a visual metric satisfies the CLP.  
\end{theo}

As announced,  we will  verify that $\borg$ satisfies the hypothesis of Proposition \ref{propcritCLP}.   To prove that  $\modcomb{1}{\F_0,G_k}$  is unbounded, it is enough to see that for every $N \in \N$ there exist $N$ disjoint curves in $\borg$. Indeed, this implies that for $k\geq 0$ large enough $\modcomb{1}{\F_0,G_k}>N$. 

\paragraph{To follow curves to control the modulus.}

 For $p>1$, we want to prove that the curves of $\borg$ satisfy the following property.
 
\begin{center}\begin{minipage}[c]{12cm}
$(P)  $ : For  $\epsilon>0$,  there exists a finite set $F$ of bi-Lipschitz maps $f :\borg \longrightarrow \borg$  such that for   any curve $\gamma\in \F_0$ and any   curve $\eta$ in $\borg$, the subset $\bigcup_{f\in F} f(\gamma)$ of $\borg$ contains a curve that belongs to $\mathcal{U}_\epsilon (\eta)$.
\end{minipage}
\end{center}
Where $\mathcal{U}_\epsilon (\eta)$ denotes the $\epsilon$-neighborhood of $\eta$ for the $C^0$ distance (see Subsection \ref{subsecterminota} for details).  Intuitively, $(P)$ holds if from  any curve $\gamma$ we can \index{Follow curve} \emph{follow} any other curve  $\eta$ using bi-Lipschitz maps. The following computation shows that  property $(P)$ implies the desired inequality.  
\begin{prop}
If $\modcomb{1}{\F_0,G_k}$  is unbounded, then property $(P)$ implies the CLP.
\end{prop}
\begin{proof}
Let $\eta$ be a curve in $\borg$ and $\epsilon>0$. Fix $\rho$ a  $\mathcal{U}_\epsilon (\eta)$-admissible function. The inequality required by the hypothesis of Proposition \ref{propcritCLP} is obtained if we find a constant $K>0$ independent of the scale $k$ and   a $\F_0$-admissible function $\rho'$ such that $M_p(\rho')\leq K \cdot M_p(\rho)$.

 Let $F$ be the set of bi-Lipschitz maps given by the property $(P)$. We  assume, without loss of generality that $F$ contains $F^{-1}$. We  define   the function $\rho' : G_k \longrightarrow [0,+\infty)$ by: 

\[(*) \ \ \ \ \ \rho' ( v) = \sum\limits_{f\in F}\sum\limits_{fw\cap  v \neq \emptyset} \rho(w).   \] Let $\gamma \in \F_0$ and $\theta$ be a curve contained in $\bigcup\limits_{f\in F} f \gamma \subset \borg$ such that $\theta \in \U_\epsilon(\eta)$. Then 
\[L_{\rho'}(\gamma) = \sum\limits_{f\in F}\sum\limits_{v\cap\gamma\neq \emptyset}\sum\limits_{fw\cap  v \neq \emptyset} \rho(w)\geq  \sum\limits_{f\in F} \sum\limits_{w\cap f\gamma\neq \emptyset} \rho(w).\]On the other hand \[L_{\rho}(\theta) \leq \sum\limits_{f\in F} L_\rho (f\gamma) = \sum\limits_{f\in F} \sum\limits_{w\cap f\gamma\neq \emptyset} \rho(w).\] Hence $L_{\rho'}(\gamma)\geq L_\rho (\theta)$ and $\rho'$ is $\F_0$-admissible.

Then the number of terms in the right-hand side of the definition $(*)$ is bounded by a constant $N$ depending only on $\# F$, the bi-Lipschitz constants of the elements of $F$, and  the doubling constant of $\borg$. Therefore by convexity \[M_p(\rho')= \sum\limits_{v\in G_k} \Big(\sum\limits_{f\in F}\sum\limits_{fw\cap   v \neq \emptyset} \rho(w)\Big)^p \leq N^{p-1} \cdot \sum\limits_{v\in G_k} \sum\limits_{f\in F}\sum\limits_{w\cap f v \neq \emptyset} \rho(w)^p  \leq N^p \cdot \# F \cdot M_p(\rho) . \] 
 \end{proof}
 
Note that this idea of following curves may be used to obtain an inequality between the modulus of any pair of sets of curves.

\paragraph{The issue of parabolic limit sets.} Since $\Gamma$ acts on $\borg$ by bi-Lipschitz homeomorphisms, it is natural to try to follow curves  by using the elements of $\Gamma$. However, in right-angled buildings some curves may be contained in parabolic limit sets (boundaries of residues). As we will see in Example \ref{ex translation PLS}, these curves are  obstacles to proving the property $(P)$ by using the elements of $\Gamma$.

To solve this problem we start by showing that    $\modcombfo$ is determined by the combinatorial modulus of the sets of all the  curves contained in a parabolic limit set.   This is done at the beginning of the proof of Theorem \ref{theomajomodul}.

\paragraph{Following curves inside parabolic limit sets.}
 Then inside the parabolic limit set $\partial P$ it is possible to follow curves. An analogue of property $(P)$ inside the parabolic limit sets is proved in Proposition \ref{propdupli}. From this property we can obtain Theorem \ref{theocourbedanspara}. This theorem is the first major step toward the proof. Essentially it says that the combinatorial modulus of $\F_{\delta,r}(\partial P)$  is controlled by any curve contained in $\partial P$.

\paragraph{Controlling the modulus in $\borg$ by the modulus in the boundary of an apartment.}

The second major step in the proof is to use the building structure to reduce the problem in  $\borg$ to a problem in the boundary of an apartment $i.e$ in $\partial W_D$.  
This is done by Theorem \ref{theocontrolimmappart}. Essentially, this theorem says that the modulus of a curve family in $\borg$ is controlled by a weighted modulus defined in the boundary of an apartment. The idea used to prove this is that, from the point of view of the modulus, $\borg$ can be seen as the direct product of the boundary of an apartment by a finite set whose cardinality depends on the scale.

\paragraph{Conclusion of the proof thanks to the symmetries of $D$.}
Now, thanks to Theorems \ref{theocourbedanspara} and \ref{theocontrolimmappart},   we arrive at a point where the modulus of $\F_{\delta,r}(\partial P)$, and thus of $\F_0$, is controlled by some modulus of the parabolic limit sets of $W_D$. The idea we use to conclude is that the symmetries of $D$ extends to the boundary of $W_D$. Combining with the elements of the groups, these symmetries permit us to follow curves in $\partial W_D$. As a consequence we obtain   a strong control of the modulus of the parabolic limit sets in $\partial W_D$ and we can  complete the proof. 
 
\section{Locally finite right-angled hyperbolic buildings}
\label{sec LFRAHB}
The aim of this  article is to study combinatorial modulus on boundaries of hyperbolic buildings. Below we  set up the context about hyperbolic buildings that will be used until the end of this  article. In particular, we  discuss the geometry of locally finite right-angled hyperbolic buildings.   

 For details concerning the notions recalled in this section, we refer  to  \cite{TitsBuildingsLectureNotes}, \cite{RonanBuildings}, or \cite{AbramBrown}. Concerning the Davis realization, we  can   refer to \cite[Chapter 8]{DavisBook} or to \cite{MeierWhen} for an example of the Davis construction along with suggestive pictures.
Bellow $S=\{s_1,\dots, s_n\}$ is a fixed finite set.  

\subsection{Chamber systems}
 \label{subsec sys of chambers}
 
Following the definition of J. Tits, a \emph{chamber system} $X$ \index{Chamber system} over  $S$ is a set endowed with a family of partitions indexed by $S$. The elements of $X$ are called  \emph{chambers}.
 
 Hereafter $X$ is a chamber system over $S$.  For $s\in S$, two chambers $c, c'\in X$ are said to be \emph{$s$-adjacent} if they belong to the same subset of $X$ in the partition associated with $s$. Then we write $c \sim_s c'$. Usually, omitting the type of adjacency  we refer to \emph{adjacent} chambers and we write $c \sim c'$. Note that any chamber is adjacent to itself.
 
 A \emph{morphism   $f : X \longrightarrow X'$} \index{Morphism of chamber system} between two chamber systems $X,X'$ over $S$ is a map that preserves the adjacency relations. A bijection of $X$ that preserves the adjacency relations is called an \emph{automorphism} and we designate by $\mathrm{Aut}(X)$ the \emph{group of automorphisms} of $X$.
A \emph{subsystem of chamber} $Y$ of $X$ is a subset $Y\subset X$ such that the inclusion map is a morphism of chamber systems.
 
   We call   \emph{gallery},\index{Gallery} a finite sequence  $\{c_k \}_{k=1, \dots , \ell}$  of chambers  such that  $c_k \sim c_{k+1}$  for  $k=1, \dots , \ell-1$. The galleries induce a \emph{metric on $X$}.
   
   \begin{déf} The \emph{distance between two chambers} $x$ and $y$ is the length of the shortest gallery connecting $x$ to $y$.
\end{déf}
  
  We   use the notation \index{$\dc{\cdot, \cdot }$} $\dc{\cdot, \cdot }$ for this metric over $X$.  A shortest gallery   between two chambers is called \emph{minimal}.  
  
   Let $I \subset S$. A subset $C$ of  $X$ is said to be \emph{$I$-connected} if for any pair of chambers $c,c' \in C$ there exists a gallery   $c=c_1 \sim \dots \sim c_\ell=c'$ such that for any $k=1, \dots , \ell-1$, the chambers  $c_k$ and $c_{k+1}$ are $i_k$-adjacent for some  $i_k \in I$. 

   \begin{déf} \label{def residus}
    The $I$-connected components are called the \emph{$I$-residues} or the \emph{residues of type $I$}\index{Residues}. The cardinality of $I$  is called the \emph{rank} of the residues of type $I$. The residues of rank $1$ are called \emph{panels}.\index{Panels}
    \end{déf}  
   The following notion of convexity is used in chamber systems.
   
\begin{déf} \label{def convexe}
A subset $C$ of $X$ is called \emph{convex} if every minimal gallery whose extremities belong to $C$ is entirely contained in $C$.
\end{déf}    
The convexity is stable by intersection and for $A\subset X$,  the \emph{convex hull} \index{Convex hull}of  $A$ is the smallest convex subset  containing $A$. In particular, convex subsets of $X$ are subsystems. The following example is crucial because it will be used to equip Coxeter groups and graph products with structures of chamber systems (see Definition \ref{def CoxSyst} and Theorem \ref{theo defbuildinggp}).  

\begin{exmp} 

\label{ex systeme chambre}
Let  $G$ be a group, $B$ a subgroup and $\{H_i\}_{i\in I}$ a family of subgroups of $G$ containing~$B$. The set of left cosets of $H_i/B$ defines a partition of $G/B$. We denote by $C(G,B,\{H_i\}_{i\in I})$ \index{$C(G,B,\{H_i\}_{i\in I})$}this chamber system over  $I$. This chamber system comes with a natural action of $G$. The group $G$ acts by automorphisms and transitively on the set of chambers. 
\end{exmp}

\subsection{Coxeter systems} 
 \label{subsect Coxeter}

 A  \emph{Coxeter matrix over $S$} is a symmetric matrix $M=\{m_{r,s}\}_{r,s\in S}$ whose entries are elements of $\N \cup \{\infty\}$ such that  $m_{s,s}=1$ for any $s\in S$ and  $\{m_{r,s}\}\geq 2$ for any $r, s \in S$ distinct. Let $M$ be a Coxeter matrix. The \emph{Coxeter group} \index{Coxeter group} of type $M$ is the group given by the following presentation
\[W=\left\langle s \in S \vert (rs)^{m_{r,s}}=1 \text{ for any } r,s\in S  \right\rangle .\]We call \emph{special subgroup} a subgroup of $W$  of the form \[W_I= \left\langle s \in I \vert (rs)^{m_{r,s}}=1 \text{ for any } r,s\in I  \right\rangle \text{ with } I\subset S. \]
 \begin{déf} \label{def ssgrpeparacox} We call \emph{parabolic subgroup}  \index{Parabolic subgroup} a subgroup of $W$  of the form $w W_I w^{-1}$ where $w\in W$ and $I\subset S$. An involution of the form $w s w^{-1}$ for $w\in W$ and $s\in S$ is called a \emph{reflection}.\index{Reflection}
 \end{déf}

\begin{exmp} \label{ex cox poly} 
Let $\mathbb{X}^d= \s^d, \mathbb{E}^d$ or $\h^d$. A \emph{Coxeter polytope} \index{Coxeter polytope}is a convex polytope of $\mathbb{X}^d$ such that any dihedral angle  is of the form $\frac{\pi}{k}$ with $k$ not necessarily constant. Let $D$ be a Coxeter polytope and  let  $\sigma_1, \dots, \sigma_n$ be the codimension 1 faces of $D$.  We set $M= \{m_{i,j}\}_{i,j=1,\dots, n}$ the matrix defined by $m_{i,i}=1$, $m_{i,j}= \infty$ if $\sigma_i$ and $\sigma_j$ do not meet in a codimension 2 face, and $m_{i,j}=k$ if $\sigma_i$ and $\sigma_j$  meet in a codimension 2 face  and $\frac{\pi}{k}$ is the dihedral angle between $\sigma_i$ and $\sigma_j$. 

Then a theorem of H. Poincaré (see \cite[Theorem 1.2.]{GabPauImmeubles})  says that the reflection group  of $\mathbb{X}^d$  generated by the  codimension 1 faces of $D$ is  a discrete subgroup of $\mathrm{Isom}(\mathbb{X}^d)$ and is isomorphic to the Coxeter group of type $M$.
\end{exmp}

\begin{déf} \label{def CoxSyst}
With the notation introduced in Example \ref{ex systeme chambre}, the\index{Coxeter system} \emph{Coxeter system} associated with $W$ is the chamber system over $S$ given by $C(W,\{e\},\{W_{\{s\}}\}_{s\in S})$.
We use the notation $\cox$ \index{$\cox$} to denote this chamber system.
\end{déf}

 The chambers of $\cox$ are the elements of $W$ and two distinct chambers $w,w' \in W$ are $s$-adjacent if and only if  $w=w's$. 
For $I\subset S$,  notice that for  any $I$-residue $R$  in $\cox$ there exists $w \in W$ such that, as a set $R=wW_I$.  
 Again  $W$ is a group of automorphisms of $\cox$ that acts transitively on the set of chambers.

Hereafter $\cox$ is a fixed Coxeter system.

\begin{exmp} \label{ex cox poly pavage} 
In the case of Example  \ref{ex cox poly}, the chamber system associated with $W$ is realized geometrically by the tilling of $\mathbb{X}^d$ by copies of the polytope $D$. Two chambers  are adjacent in $\cox$ if and only if the corresponding copies of $D$ in $\mathbb{X}^d$ share  a codimension 1 face.
\end{exmp}

\subsection{The Davis chamber of $\cox$}
 \label{subsec davis chcox}
  
 To a Coxeter group $W$, M.W. Davis associates a cellular complex $D$ called the  \index{Davis chamber}\emph{Davis chamber}.  In the particular case  of reflection groups (see Examples \ref{ex cox poly} and \ref{ex cox poly pavage}) the Davis Chamber is the Coxeter polytope.

 We recall that  $S=\{s_1, \dots , s_n\}$ is a set of  generators of $W$. Let $\mathcal{S}_{\neq S}$ be the collection of proper subsets of $S$. We denote by 
\begin{center}
 $\mathcal{S}_f$ the set of proper subsets $F\varsubsetneq S$ such that $W_F$ is  finite. 
 \end{center}

 By \cite[Appendix A]{DavisBook}, a poset admits a geometric realization which is a simplicial complex. This complex is such that the inclusion relations between cells represent  the partial order. We denote by $D$  the \emph{Davis chamber} which is the geometric realization of the poset  $\mathcal{S}_f$. In the following we give details of this construction.

Let  $\Delta^{n-1}$ be the standard $(n-1)$-simplex and label the codimension 1 faces of $\Delta^{n-1}$ with distinct elements of $S$.  Then  a codimension $k$ face $\sigma$   of $\Delta^{n-1}$ is associated with a \emph{type} \emph{i.e} a subset $I\subset S$ of cardinality $k$. In this case, we write \index{$\sigma_I$} $\sigma_I$ for the \emph{face of    type $I$}\index{Type}. Equivalently, we can say that each vertex of the barycentric subdivision of $\Delta^{n-1}$ is associated with a subset of $S$.  Combining with the fact that the empty set is associated with the barycenter of the whole simplex,  we get a bijection between the vertices of the barycentric subdivision and $\mathcal{S}_{\neq S}$. Hence a vertex in the barycentric subdivision is designated by  $(s_i)_{i\in K}$ for $K \subset \{1, \dots , n\}$.  Using this identification, let $\mathcal{T}$ be the subgraph of the $1$-skeleton of the barycentric subdivision of $\Delta^{n-1}$ defined as follows:
\begin{itemize}
\item $\mathcal{T}^{(0)}=\mathcal{S}_{\neq S}$,
\item the vertices $(s_i)_{i\in I}$ and $(s_j)_{j\in J}$, with $\# J  \geq \# I$, are adjacent if and only if  $I\subset J$ and $\# I   =\# J  -1$.
\end{itemize}

In the following definition, for $k\geq 1$ we call a $k$-cube, a CW-complex that is isomorphic, as a cellular complex, to the Euclidean $k$-cube $[0,1 ]^k$. In particular, it is not necessary to equip these cubes with a metric  for the  purpose of this chapter.

\begin{déf}

The $1$-skeleton $D^{(1)}$ of the Davis chamber  is the full subgraph of $\mathcal{T}$ generated by the elements of   $\mathcal{S}_f$. The \emph{Davis chamber} is obtained from $D^{(1)}$ by attaching  a   $k$-cube to every  subgraph that is isomorphic to the $1$-skeleton of a $k$-cube.  

\end{déf}
  
By construction, $D\subset \Delta^{n-1}$. We call \emph{maximal faces} of $D$ the subsets of the form $\sigma \cap D$ where $\sigma$ is a codimension 1 face of $\Delta^{n-1}$. Likewise, for $I\subset S$, the \emph{  face of $D$ of type $I$} is $D\cap \sigma_I$. 
\begin{exmp}
 In the case of Example \ref{ex cox poly}, the Davis chamber is combinatorially identified with the Coxeter polytope. So if we equip $D$ with the appropriate metric (Euclidean, spherical, or hyperbolic) we recover  the Coxeter polytope.
\end{exmp}
 
\subsection{Buildings} 
\label{sec buildings}

Buildings are singular spaces defined by J. Tits. We may view them as  higher dimensional analogues of trees.   Hereafter $(W,S)$ is a fixed Coxeter system.

\begin{déf} [{\cite[Definition 3.1.]{TitsBuildingsLectureNotes}}]

A chamber system $\Delta$ over $S$ is a \emph{building of type $(W, S)$} \index{Building of type $(W, S)$} if it admits a maximal family $\apD$ of subsystems isomorphic to $(W, S)$, called apartments, such that 
\begin{itemize}
\item any two chambers lie in a common apartment,
\item for any pair of apartments $A$ and $B$, there exists an isomorphism from $A$ to $B$  fixing $A\cap B$.
\end{itemize}
\end{déf}

An immediate consequence of this definition is the existence of retraction maps of the building onto the apartments.\begin{déf}
Let $x\in \Delta$ and $A\in \apD$. Assume that $x$ is contained in $A$. We call \index{Retraction} \emph{retraction onto $A$ centered at $x$} the map $\pi_{A,x} : \Delta \longrightarrow A$ defined by the following property. \begin{center}\begin{minipage}[c]{12cm}For $c\in \Delta$, there exists a chamber $\pi_{A,x}(c) \in A$ such that for any apartment $A'$  containing $x$ and $c$,  for any  isomorphism  $f:A'\longrightarrow A$  that fixes $A \cap A'$, we have $f(c)=\pi_{A,x}(c)$\end{minipage}\end{center}
\end{déf}
 
Hereafter, $\Delta$ is a fixed building of type $(W,S)$. The building $\Delta$ is called a \emph{thin} (resp. \emph{thick})  \index{Thin building}   \index{Thick building} building if any panel contains exactly two (resp. at least three) chambers. Note that thin buildings are Coxeter systems.

\subsection{Graph products and right-angled buildings}
\label{subsec def gp rab}

Let $\mathcal{G}$ denote a \emph{finite simplicial graph} \emph{i.e} $\mathcal{G}^{(0)}$ is finite, each edge has two different vertices, and $\mathcal{G}$ contains no double edge. We denote  by  $\mathcal{G}^{(0)}=\{v_1, \dots, v_n\}$  the vertices of $\mathcal{G}$.  If for $i\neq j$, the corresponding vertices $v_i, v_j$ are connected by an edge, we write $v_i \sim v_j$.  A finite cyclic group $ G_i=\left\langle s_i\right\rangle$ of order $q_i\geq 2$ is associated with each $v_i \in \mathcal{G}^{(0)}$ and we set $S=\{s_1, \dots , s_n\}$.  
Throughout this article, we assume that $n\geq 2$ and that $\mathcal{G}$ has at least one edge.  

\begin{déf} \label{def graphprod}
The \emph{graph product} \index{Graph product}  given by $(\mathcal{G},\{G_i\}_ {i=1,\dots , n} )$ is the group defined by the following presentation\[ \Gamma = \left\langle s_i\in S \vert  s_i^{q_i}=1,  s_is_j=s_js_i \text{ if } v_i \sim v_j  \right\rangle  . \] 
\end{déf}

\begin{exmp}

If the graph $\mathcal{G}$ is fixed with $q_i \in \{2,3,\dots ,\}$, graph products are groups between \emph{right-angled Coxeter groups} (see \cite{DavisBook}) and \emph{right-angled Artin groups} (see \cite{CharneyRAAG}). If the integers $\{q_i\}_ {i=1,\dots , n} $ are fixed and we add edges to the graph starting from a graph with no edge, those groups are groups between free products  and direct products of  cyclic groups.
\end{exmp}
From now on, we fix a graph product $\Gamma$ associated with the pair $(\mathcal{G},\{G_i\}_ {i=1,\dots , n} )$. By analogy with Definition \ref{def ssgrpeparacox},  we define parabolic subgroups in $\Gamma$.
 \begin{déf}
 The subgroup of $\Gamma$ generated by a subset $I\subset S$ is denoted by \index{$\Gamma_I$} $\Gamma_I$ and a subgroup of the form $g \Gamma_I g^{-1}$, with $g \in \Gamma$, is called a \emph{parabolic subgroup}.  \index{Parabolic subgroup}
    \end{déf}

Let $W$ be  the graph product  defined by the pair $(\mathcal{G},\{\Z/2\Z \}_ {i=1,\dots , n})$. This graph product is isomorphic to the right-angled Coxeter group  of type $M=\{m_{i,j}\}_{i,j=1,\dots, n}$ defined by : $m_{i,j}=2$ if $v_i\sim v_j$ and  $m_{i,j}=\infty$ if $v_i\nsim v_j$.

 Throughout  this article, $W$  denotes this Coxeter group canonically associated with $\Gamma$ and $(W,S)$ is the Coxeter system associated with $W$.

\begin{theo}[{\cite[Theorem 5.1.]{DavisCAT}}] \label{theo defbuildinggp}
Let $\Delta$ be the chamber system $C(\Gamma,\{e\}, \{\Gamma_{\{s\}}\}_{s\in S})$  (see Example  \ref{ex systeme chambre}). Then $\Delta$ is a building of type $(W,S)$.

\end{theo}

Hereafter, $\Delta$ denotes the right-angled building associated with $\Gamma$ by the preceding theorem. In Subsection \ref{subsec davis complx rab}, we describe the Davis complex associated with this  building. 

We notice that $\Gamma$ is infinite if and only if the graph $\mathcal{G}$ contains  two distinct vertices $v_i,v_j$ such that $v_i\nsim v_j$. A criterion of M. Gromov allows  J. Meier to  prove  that an infinite graph product  $\Gamma$ is hyperbolic  if and only if in $\mathcal{G}$ any circuit of length $4$ contains a chord (see \cite{MeierWhen}). For an infinite hyperbolic graph product, a necessary  and sufficient condition is given in   \cite{DavisMeierTopo} for $\borg$ to be arcwise connected (see Theorem \ref{theo davismeier} bellow). This condition involves only the graph $\mathcal{G}$. In the rest of this paper, we will assume that $\Gamma$ is infinite  hyperbolic with arcwise connected boundary $\borg$.

\subsection{The Davis complex associated with $\Gamma$}

\label{subsec davis complx rab}

 To a graph product $\Gamma$,  M.W. Davis associates a cellular complex $\Sigma$ called the  \index{Davis complex} \emph{Davis complex}. This complex is a metric space on which $\Gamma$ acts geometrically.   In the particular case  of reflection groups (see Examples \ref{ex cox poly} and \ref{ex cox poly pavage}) the Davis Complex is   $\mathbb{X}^d$ tilled by the  Coxeter polytopes. This complex is also called the geometric realization of the building $\Delta$.
Below we introduce the Davis complex associated with $\Gamma$.  Again we refer   to \cite{MeierWhen} for an example  along with suggestive pictures.

 Let $D$ be the Davis chamber associated with $W$ as in Subsection \ref{subsec davis chcox}. Again  a face of $D$ is associated with a type $I\subset S$.  For $x \in D$, if $I$ is the type of the face containing $x$ in its interior, we set $\Gamma_x:=\Gamma_I$. To the interior points of $D$  we associate the trivial group $\Gamma_\emptyset$. 

Now we can define the \emph{Davis complex} \index{Davis complex}:  $ \Sigma = D \times \Gamma / \sim $   with 
 \[(x,g)\sim (y,g') \text{ if and only if }x=y \text{ and }g^{-1}g' \in \Gamma_x.\]

 We study  the building $\Delta$ through it geometric realization $\Sigma$ and we briefly recall what this means. 

A chamber of $\Sigma$ is a subset of  the form $[D\times\{g\} ]$ with $g\in \Gamma$.  Two chambers  are adjacent if and only if they share a maximal face. For a subset  $E\subset \Sigma $ we  designate by $\Ch{E}$ \index{$\Ch{E}$} the set of chambers contained in $E$. Equipped with this chamber system structure, $\Sigma $ is isomorphic to  $\Delta$.  In particular, the set of apartments in $\Sigma$ is designated by $\ap$\index{$\ap$}.   Then the left action of $\Gamma$ on itself induces an action on $\Sigma$. For  $\gamma \in \Gamma$ and $ [(x,g)]\in \Sigma$ we set  $\gamma [(x,g)] :=[(x,\gamma g)]$. Moreover this action induces a simply transitive action of $\Gamma$ on $\ch$.   Naturally this action is also isometric for  $\dc{\cdot , \cdot}$.  
 
\begin{exmp} 
In the case of Example \ref{ex cox poly}, if we equip  $D$ with the appropriate metric we see that the Davis complex is realized by the tilling of $\mathbb{X}^d$ by $D$.  
\end{exmp}

 \subsection{Building-walls and residues in the Davis complex}
\label{subsec bw and dials}

We call \emph{base chamber} of $\Sigma$, denoted by $x_0$, the chamber  $[D \times \{ e\}]$. For $g\in \Gamma$, as $[D \times \{ g\}]$ is the image of $x_0$ under $g$,  we  designate this chamber by $g x_0$.   Below we present some basic tools used to describe the structure of $\Sigma$. In particular, we extend to $\Sigma$ some definitions and properties that have been  used in Coxeter systems.

The notion of walls in a Coxeter system  extends to right-angled buildings.

\begin{déf}
\text{ }
\numerote{\item We call \emph{building-wall} \index{Building-wall} in  $\Sigma$ the subcomplex $M$ stabilized by a non-trivial isometry  $r=g s^\alpha g^{-1}$ with  $g \in \Gamma$, $s \in S$, $\alpha \in \mathbb{Z}$ and $s^{\alpha}\neq e$. The isometry $r$  is called a \index{Rotation around a building-wall} \emph{rotation around $M$}. We denote by $\bw$ the set of all the building-walls of $\Sigma$.  
 
\item  Let $M$ be a building-wall associated with a rotation $r\in \Gamma$. For   $x \in \ch$ we say that $M$ is \emph{along $x$} if $r(x)$ is adjacent to $x$.
 } 
\end{déf}

With the   notation of  the definition , the term $s\in S$ is called the \emph{type of the building-wall $M$ and of the rotation $r$}. We remark that in the graph product $\Gamma$ two elements of $S$ that are conjugate are equal so the type is uniquely defined. Indeed if $s=g s'g^{-1}$  with  $g \in \Gamma$, $s,s' \in S$, let $M$ be the building-wall associated with the rotation $s$. Then we observe that from the base chamber $x_0$ we can reach the chamber  $gx_0$ by successive   rotations about faces that are all orthogonal to $M$. In other word $g$ is a product of elements of $S$ that all commute with $s$. Thus $s=s'$.
 Clearly with the notation of the definition,  the building-wall $M$ is fixed by any rotation $g s^{\alpha'} g^{-1}$ with $s^{\alpha'}\neq e$.  


We say that the building-wall $M$  is non-trivial if it contains more than one point. A non-trivial   building-wall $M$  may be equipped with a building structure. Indeed, if  $s_i$ is the type of $M$,  associated with $v_i  \in \mathcal{G}^{(0)}$, we write $I  =\{ j : v_j\sim v_i, v_j\neq v_i \}$ and $V=\{v_j\in \mathcal{G}^{(0)} : j\in I\}$. Then if  $\mathcal{G}_V$ is the full subgraph generated by $V$, we can check that, $M$ is isomorphic to the geometric realization of the graph product $(\mathcal{G}_V,\{\Z/q_i\Z\}_{i\in I})$. The Davis chamber of this geometric realization is the maximal face of type $s_i$ of $D$.  Moreover building-walls also divide $\Sigma$ in isomorphic connected components.
In the case of infinite dimension 2 buildings, the building-walls are trees and thus they have been called \emph{trees-walls} by M. Bourdon and H. Pajot in \cite{BourdonPajotRigi}. These explain  our terminology.
 \begin{déf}
Let  $M$ be a building-wall of type $s$ and let  $r\in \Gamma$ be a rotation around $M$. 
 A \index{Dial of building} \emph{dial of building bounded by $M$} is the closure in $\Sigma$ of a connected component of $\Sigma\backslash M$. We denote by   $\dialgp$ \index{$\dialgp$} the set of all the dials of building of $\Sigma$. 

  \end{déf}

  This definition implies the following fact. 
  
\begin{fact} \label{prop defcadrangp}
Let  $M$ be a building-wall of type $s$. Assume that   $s$ is of finite  order $q$.  Then $\Sigma\backslash M$ consists of $q$   connected components.   We  designate by $D_0(M), D_1(M), \dots, D_{q-1}(M)$ these dials of building, with  the convention that $x_0 \subset D_0(M)$. 
In this setting, for any $i=0,\dots ,q-1$, if $y\in \Ch{D_i(M)}$ then 
  \[\Ch{D_i(M)}=\{x \in \Ch{\Sigma} : \dc{y,x}<\dc{y,rx} \}.\]
 Finally $r$ permutes $D_0(M), D_1(M), \dots, D_{q-1}(M)$.

For a building-wall associated with a type $s\in S$ of infinite order, the analogous property holds. 
\end{fact}

In thin right-angled buildings, in particular in apartments,   building-walls are called \emph{walls}  and   dials  of building   are called \emph{half-spaces}. For $A\in \apD$ we write  $\mathcal{M}(A)$ for the set of all the walls and  $\demiesp{A}$ for the set of all the half-spaces of $A$. \index{$\demiesp{A}$} \index{$\mathcal{M}(A)$} 
 
The building-walls  in $\Sigma$ and the walls  in the apartments are closely related.
 \begin{fact}  
 \label{fact bwstructure}
\numeroti{\item For $A\in \apD$  and $x\in \Ch{A}$ we have  
\begin{align*}
\mathcal{M}(A)&=\{M\cap A : M \in \bw \text{ and } M\cap A \neq \emptyset\}, \\
 &=\{\pi_{A,x}(M): M \in \bw \text{ and } M\cap A \neq \emptyset\},\\
  &=\{\pi_{A,x}(M): M \in \bw\}. 
\end{align*}
\item   For $A\in \apD$ and $x\in \Ch{A}$ we have \[\bw=\bigcup_{m \in \mathcal{M}(A)} C({\pi^{-1}_{A,x}(m)}),\] where $C({\pi^{-1}_{A,x}(m)})$ denotes the set of all the connected components of $\pi^{-1}_{A,x}(m)$.
} 
\end{fact}

Building-walls (resp. walls) bound dials of buildings (resp. half-spaces) so a similar fact holds for dials of building and half-spaces.

We use the following terminology  to describe a building-wall relatively to some chambers.

\begin{déf} \label{defcross}
Let $M\in \bw$ and $E,F\subset \Sigma$.
 \begin{enumerate}[i)]
   \item We say that   \emph{$M$ crosses $E$} if $E\backslash M$  has at least two connected components.
  
 \item We say that \emph{the building-wall $M$ separates $E$ and $F$} if the interior of $E$ and $F$ are entirely contained in two distinct connected components of $\Sigma\backslash M$. 
 \end{enumerate}
 \end{déf}  

The metric over the chambers is determined by the building-wall structure.

\begin{prop}[{\cite[Proposition 5.8]{ClaParaResidinBuildings}} ] \label{prop distcombigrap}
Let $x_1$ and $x_2$ be two chambers. If we denote  $\dc{\cdot,\cdot}$ the metric on the chamber system, then \begin{center}$\dc{x_1,x_2}=\#\{M \in \bw : M \text{ seperates } x_1 \text{ and } x_2\}.$
\end{center}
\end{prop}

In a right-angled building it appears that two distinct building-walls are either orthogonal or do not intersect. This explains the following terminology and notation.

\begin{nota}Let $M$ and $M'$ be two distinct building-walls. \numeroti{\item if $M\cap M' \neq \emptyset$ we write $M \perp M'$ and we say that \emph{$M$ is orthogonal to $M'$},\item if $M\cap M' = \emptyset$ we write $M \parallel M'$ and we say that \emph{$M$ is parallel to $M'$}.}\end{nota} 
 
Clearly,  if $D,D' \in \dialgp$ are bounded by    $M , M'\in \bw$ with if  $M \perp M'$, then $D\cap D'$ contains a chamber. On the other hand, if $M\parallel M'$ then there exists $D$ bounded by $M$ and $D'$ bounded $M'$ such that $M'\subset D$ and $M \subset D'$.

   \subsection{Geometric characterization of parabolic subgroups} 
   \label{subsec geom carac cox}

Now we discuss residues in $\Sigma$, parallel to the discussion at the end of Subsection  \ref{subsect Coxeter}.

\begin{nota}
For $I\subset S$ and $g \in \Gamma$, let \index{$g \Sigma_I$}  $g \Sigma_I$ denote the union of the chambers of the $I$-residue containing $g x_0$. 
\end{nota}

 Notice  that $g \Sigma_I = g \Gamma_I x_0$ and $\Ch{g\Sigma_I} =g\Gamma_I$. For  simplicity, in the following  we  also call a subset $g\Sigma_I\subset \Sigma$ a \emph{residue}.\index{Residue}   Notice that a  rotation  around a building-wall  that crosses   $g\Sigma_I$ is of   the form $g\gamma  s^\alpha \gamma^{-1}g^{-1}$ with $s\in I$, $s^\alpha\neq e$  and  $\gamma \in \Gamma_I$. By the definitions of the action and   the residues we obtain the following fact.

\begin{fact} \label{fact residuebasegp} Let $R=g \Sigma_I$ be a residue.  Then   \liste{\item $R$ is stabilized by the rotations around the building-walls that cross it, \item  $\mathrm{Stab}_\Gamma(R)=g \Gamma_I g^{-1}$ is generated by these rotations,
\item   The type $I$ of $R$ is equal to the set of types of all the building-walls $M$ satisfying 
\[M \text{ crosses } R \text{ and } M \text{ is  along  } g x_0.\]  }
\end{fact}

The following result gives a converse to the preceding fact. We recall that a set of chambers is convex if it is convex for the combinatorial metric over the chambers (see Definition \ref{def convexe}).
 \begin{theo} \label{theo caracresidurab}
 Let $C\subset \ch$ be a convex set of chambers. Let $R= \bigcup_{x\in C} x \subset \Sigma$ and let $P_R$ denote the group generated by the rotations  around the building-walls that cross $R$. If $P_R$ stabilizes $R$, then $R$ is a residue in $\Sigma$. 
 \end{theo}  
\begin{proof} Up to a translation on $R$ and a conjugation on $P_R$, we can assume that $x_0\subset R$.
 We start by  proving that  $P_R$ acts freely  and transitively on the set of chambers $C$.  The action of $\Gamma$ is free thus the action of $P_R$ is free. For  $x \in C$, by convexity of $C$, there exists a gallery \[x_0\sim x_1\sim \dots \sim x_\ell=x\]   of distinct chambers  in  $C$. Let  $M_i$ be the  building-wall  of type $s_i\in S$   between $x_{i-1}$ and $x_i$.  Then  \[s_1^{\alpha_1} x_0=x_1,\ s_1^{\alpha_1} s_2^{\alpha_2} x_0= x_2 , \ \dots, \ s_1^{\alpha_1} \dots s_\ell^{\alpha_\ell} x_0= x,\]for some exponent $\alpha_i \in \Z$.

We notice that  $s_1^{\alpha_1} \dots s_{i-1}^{\alpha_{i-1}} s_i (s_{1}^{\alpha_1} \dots s_{i-1}^{\alpha_{i-1}} )^{-1}$ is a rotation around $M_i$. Therefore $x$ may be obtained from $x_0$ by  successive rotations around the building-walls $M_i$. These building-walls cross  $R$, thus the action is transitive. This proves that $R=P_Rx_0$ and $\mathrm{Stab}_\Gamma(R)=P_R$. It remains to prove that $P_R$ is of the form $\Gamma_I$ for a certain $I\subset S$.
 
We set $I\subset S$   the set of all the types of the building-walls that cross $R$ along  $x_0$ and  we identify  $P_R$ with $\Gamma_I$. The inclusion $\Gamma_I < P_R$ comes from the definitions of $P_R$ and $I$.  We proceed by induction on  $\dc{x_0,gx_0}=\ell$ to check that every element  $g$ of $P_R$ is a product of elements of $\Gamma_I$. If $\ell=0$, there is nothing to say.  If $\ell>0$ we choose  $g=s_1^{\alpha_{ 1}} \dots s_\ell^{\alpha_{\ell}} $ such that $\dc{x_0,gx_0 }=\ell$. By convexity,  $s_1 ^{\alpha_{ 1}} x_0 \in C$ so $s_1\in \Gamma_I$. Indeed  $s_1$ is a rotation around a building-wall that crosses  $R$ along $x_0$. Then $\dc{x_0, s_1 ^{-\alpha_{ 1}}gx_0}=\ell-1$ and $s_1 ^{-\alpha_{ 1}} g \in P_R$. The induction assumption allows us to conclude. 
\end{proof}

In particular, this last theorem is used in Subsection \ref{subsecpara} when we discuss the boundaries of the residues in the hyperbolic case.  Finally we recall that the intersections of parabolic subgroups  in $\Gamma$ (resp. in $W$) is again a parabolic subgroup.

\subsection{$\Sigma$ as a metric space}

  A natural geodesic metric on $\Sigma$ is obtained as follows. We designate by $D$ the Davis chamber of $\Gamma$.  
   We recall that $D$  is obtained from $D^{(1)}$ by attaching  a  $k$-cube to every subgraph   that is isomorphic to the $1$-skeleton of a $k$-cube. Now, for any $k$,  we equip each $k$-cube of $D$ with the Euclidean metric of the $[0,1]^k$.

The polyhedral metric $d(\cdot,\cdot)$ induced on $\Sigma$ by this construction is geodesic and complete. Moreover, any automorphism of  $\Delta$ is an isometry of $(\Sigma,d)$. In particular, $\Gamma$ acts geometrically on $(\Sigma,d)$ and, as $\Gamma$ is assumed to be hyperbolic,  $(\Sigma,d)$ is a hyperbolic metric space.

In $(\Sigma,d)$ the building-walls are convex and connected subsets and we can precise this descriptions using geodesic rays. 
Let $M\in \bw$ be of type $s$, let $x\in \ch$ such that $M$ is along $x$ and let $\sigma_s$ be  the maximal face of type $s$ of $x$. We denote by $\mathrm{Ext}(\sigma_s)$ the set of all the geodesic rays such that there exists $\epsilon>0$ with $r([0,\epsilon)) \subset \sigma_s$. Then  $M=\{r(t) : r\in \mathrm{Ext}(\sigma_s) \text{ and } t\in [0,+\infty)\}$. 
  
However in the case when $W$ is a reflection group of the hyperbolic space $\h^d$ it seems more natural to equip $D$ with the hyperbolic metric. Then $D$ is isometric to the Coxeter polytope provided by $W$. We designate by $d'(\cdot,\cdot)$ the piecewise hyperbolic metric on $\Sigma$ induced by this construction. This metric satisfies the same properties stated above (geodesic, complete, hyperbolic and admits a geometric action of $\Gamma$).
The  two metrics $(\Sigma,d)$ and $(\Sigma,d')$ are quasi-isometric. Since  our goal is to study $\borg$,  it makes no difference to consider $(\Sigma,d)$ or $(\Sigma,d')$.  However, the arguments presented in   Sections \ref{seccurves inPLS}, \ref{sec bord topo met}, \ref{secmoduleapartmoduleimmeuble}, and \ref{sec applictethickeness}  hold in the generic case so we consider  $(\Sigma,d)$ in those sections.   For  Section \ref{sec result}, which focuses on hyperbolic reflection groups, it will be more convenient to consider $(\Sigma,d')$.

 \subsection{Boundary of the building}
  \label{subsec BofLFRAHB}

 In this subsection we describe some basic properties of $\borg$. In the sequel, we use the geometric action of $\Gamma$ on $(\Sigma,d)$ to identify $\borg$ and $\partial \Sigma$. Now consider a building-wall  $M$  of type $s$. We have  $ \partial M\simeq \partial \mathrm{Stab}_\Gamma(M)$ using the previous identification. Consider a subgroup $P<\Gamma$, as $\Gamma$ is hyperbolic, either   $\partial P= \partial \Gamma$, or  $\partial P$ is of  empty interior. Hence, here there are two possible cases:
 \liste{\item    $\partial M\simeq \partial \Gamma$,
  \item    $\Int{\partial M}$ is  empty.}

Moreover, the hyperbolicty assumption implies the following lemma.
 \begin{lem} \label{lem intersec bord des murs}Let $M$ and $M'$ be two distinct building-walls. If $M\parallel M'$ then $\partial M \cap \partial M' =\emptyset$.  
\end{lem}
\begin{proof}
Let $\gamma : [0,+\infty) \longrightarrow \Sigma$ be a geodesic ray contained in $M$. For  simplicity we denote by $\gamma$ the image of $\gamma$. Assume that there exists $K >0$ such that $\di{\gamma(t), M'}\leq K$ for every $t\geq 0 $. Since $M\cap M' =\emptyset $ and because of the chamber structure, there exists $K' >0$ such that $K'\leq \di{\gamma(t), M'}$ for every $t\geq 0 $.

Now let $\Gamma'$ be the group generated by a rotation $r$ around $M$ and a rotation $r'$ around $M'$. If the rotations are of order two, then $\Gamma'$ is an infinite dihedral group and    the subset $\Gamma'  \gamma \subset \Sigma$ is quasi-isometric to a Euclidean half space. If the rotations are of order larger than two   then $\Gamma' \gamma$ contains a proper subset quasi-isometric to a Euclidean half space.
\end{proof}
 
Now we can precise the description of the two cases above. In the first case, $s$ commutes with every   generator $r\in S$. In the Davis complex this means that all the other building-walls are orthogonal to $M$.  Then $\mathrm{Stab}_\Gamma(M)=\Gamma$.

 In the second case, there exists $r\in S$ that does not commute with $s$. In the Davis complex this means that there exists a building-wall $M'$ parallel to $M$. This implies that $\partial M\varsubsetneq \borg$. In this case, $\borg \backslash \partial M$ is the disjoint union $\Int{\partial D_0(M) }\sqcup \dots\sqcup \Int{\partial D_{q-1}(M)}$ where $D_0(M),\dots,D_{q-1}(M)$  are the dials of building bounded by $M$. Naturally a rotation  around $M$ extends to the boundary as an homeomorphism that permutes $\partial D_0(M),\dots, \partial D_{q-1}(M)$  and  fixes $\partial M$.
Moreover $\Int{\partial M}=\emptyset$, $\Int{\partial D_{i}(M)}\neq\emptyset$,  and the topological  boundary of $\partial D_{i}(M)$ in $\borg$ is $\partial M$ for any $i=0,\dots,q-1$.  
Concerning the dials of building, the following alternative holds. Let $D$ be a dial of building bounded by the building-wall $M$, then: \liste{\item either $\partial M=\partial D=\borg$,
\item or the topological boundary of $\partial D$ in $\borg$ is $\partial M$. In this case, $\Int{\partial D(M)}\neq\emptyset$ and $\Int{\partial M}=\emptyset$.}

 In \cite[Proposition 5.2.]{BourdonKleinerCLP}, it is proved that the boundaries of half-spaces in a hyperbolic Coxeter group form a basis for the visual topology on the visual boundary. In the case of right-angled building  the analogous statement holds.

 \begin{fact} \label{fact equivdestopo}
 The topology generated by $\{\partial D : D \in \dialgp   \}$ is equivalent to the topology induced by a visual metric on $\borg$.
 \end{fact}
 
 Finally, consider an apartment $A$ containing the base chamber $x_0$ and the retraction map $\pi_{A,x_0}:\Sigma \longrightarrow A$. This retraction maps any geodesic ray of $\Sigma$ starting from a based point $p_0\in x_0$ to a geodesic ray in $A$ starting from $p_0$. Hence  $\pi_{A,x_0}$ extends naturally to the boundaries and we keep the notation $\pi_{A,x_0}:\partial \Sigma \longrightarrow \partial A$ for this extension.

\begin{rem}
In \cite{DavisMeierTopo},   M.W. Davis and J. Meier described how properties of  connectedness  of $\borg$ are encoded in the combinatorial structure  of $\mathcal{G}$. We use a corollary of their result  in Subsection \ref{subsec control curveinpara}. 
\end{rem}

\begin{rem} 
 A classification of F. Haglund and F. Paulin  states that the construction presented in Subsection \ref{subsec def gp rab} describes all the right-angled buildings in the following sense.
  
\begin{theo}[{\cite[Proposition 5.1.]{HaglundPaulinImmeubles}}] Let $\Gamma$ be the graph product given by the pair $(\mathcal{G}, \{G_i\}_{i=1,\dots, n})$ as in Definition \ref{def graphprod}. Let $\Delta$ be the building of type $(W,S)$ associated with  $\Gamma$.  Assume that  $\Delta'$ is a building of type $(W,S)$ such that for any $s_i\in S$ the  $\{s_i\}$-residues of $\Delta'$ are of cardinality $\#G_i$. Then $\Delta$ and $\Delta'$ are isomorphic. 
\end{theo}  
\end{rem}

\section{Curves in connected parabolic limit sets} 
\label{seccurves inPLS}
 
As we will see with Example \ref{ex translation PLS}, parabolic limit sets (\emph{i.e} boundaries of residues of the building) play a key role in the proof of the CLP on boundaries of graph product. 
 
 In this section, we use the ideas of \cite[Section 5 and 6]{BourdonKleinerCLP} to prove Theorem \ref{theocourbedanspara}  which is the first   major step to prove the main result of this  article (Theorem \ref{theoprincip}).  The idea of this theorem is to control the modulus of the curves of a parabolic limit set by the modulus of curves in the neighborhood of a single curve. Then we  apply this theorem to  recover a  result about boundaries of right-angled Fuchsian buildings.

We use the notation and convention of Section \ref{sec LFRAHB}. In particular $\Gamma$ is a fixed graph product given by the pair $(\mathcal{G}, \{\Z/q_i\Z\}_{i=1,\dots,n})$. We identify the building $\Delta$ with its Davis complex $\Sigma$ equipped with the piecewise Euclidean metric.   The base chamber is   $x_0$. We  assume that $\Gamma$ and $\Sigma$ are hyperbolic and that  $\borg$ is arcwise connected. The metric on $\Ch{\Sigma}$ is denoted by $\dc{\cdot,\cdot}$.
Moreover, in this section we  equip $\borg$  with a  self-similar metric that comes from the action of $\Gamma$ on $\Sigma$.

\subsection{Parabolic limit sets in $\borg$}
 \label{subsecpara}

In this subsection we give some basic properties of  boundaries of parabolic subgroups. At the end of this subsection we will see that these subsets of the boundary cause difficulties in proving the CLP.

 \begin{déf} \label{defPLS}
Let $P=g\Gamma_I g^{-1}$ be a parabolic subgroup of $\Gamma$. If the limit set of $P$ in $\borg$ is non-empty, we call it   a \index{Parabolic limit set of type $I$} \emph{parabolic limit set}.   If moreover $\partial P \neq \borg $ the parabolic limit set is called a \emph{proper} parabolic limit set. 
 \end{déf}
 Equivalently we could  say that a subset $F\subset \borg$ is   a parabolic limit set if there exists a residue $g\Sigma_I$ such that $F$ is equal to $\partial (g\Sigma_I)$ under the canonical homeomorphism between $\borg$ and $\bori$. In the following we will frequently use this point of view about parabolic limit sets.

 The following notion of convex hull of a subset of the boundary will be used in this section.

\begin{déf} \label{def convexhull}Let $F$ be a subset of $\partial \Gamma$ containing more than one point and such that $\overline{F}\neq\borg$. Let \[\mathcal{D}^c(F)=\{ D\in \dialgp : F  \subset \borg   \backslash \partial D\}.\]
The  \index{Convex hull of $F \subset \borg$} \emph{convex hull of $F$} in $\Sigma$ is defined by
\[\cv{F}=  \Sigma \backslash \cup_{D \in \mathcal{D}^c(F)} D. \]
If  $\overline{F}=\borg$ then we set $\cv{F}=\Sigma$.
\end{déf} 

Clearly we can also write $\cv{F}= \bigcap_{D \in \mathcal{D}^c(F)}\Sigma \backslash   D$. Hence $\cv{F}$ is convex for both the geodesic metric on $\Sigma$ and the combinatorial metric over the chambers of $\Sigma$. Moreover we observe that  $F\subset \partial \cv{F}$.





\begin{exmp} \label{ex conv}
 Let $\partial P$ be a parabolic limit set and assume that $P=\Gamma_I$. Then we can verify that  $\cv{\partial P}=\Sigma_J$ where $J=I \cup \{s_j\in S: s_js_i=s_is_j \text{ for any } s_i\in I\}$.
 
In particular, if $M\in \bw$ then $\cv{\partial M}$ is the the union of all the chambers along $M$.

\end{exmp}
  
 In the following definition $\overline{\Sigma}= \Sigma \cup  \partial \Sigma$  and if $M$ is a building-wall  $\overline{M} = M \cup \partial M$.
 
\begin{déf}
\text{ }
\numeroti{
 \item Let $F$ be a subset of $\bori$. We say that a building-wall \emph{$M$ cuts $F$} if   there exist two distinct indices $i$ and $j$ such that $F$ meets both $\partial D_i (M)$ and  $\partial D_j (M)$.

\item If  $E_1 \subset \bori$ and  $E_2 \subset   \Sigma $ (resp. $E_2 \subset \bori$) 
 we say that \emph{a building-wall $M$ separates $E_1$ and $E_2$} if $E_1$ and $E_2$ are entirely contained in two distinct connected components of $\overline{\Sigma} \backslash \overline{M}$.}
\end{déf}

The proof of the following fact is identical to the proof of   \cite[Lemma 5.7]{BourdonKleinerCLP}.
 
\begin{fact}
Let $F$ be a subset of $\bori$. The  building-wall $M$ cuts $F$  if and only if  $M$ crosses $\cv{F}$ (see Definition \ref{defcross}).\end{fact}

The following corollary is an immediate consequence of the preceding fact and of  Theorem \ref{theo caracresidurab}
   
 \begin{coro}\label{corocaracpls}
Let $F$ be a subset of $\bori$ containing at least two distinct points and $P_F$ denote the group generated by the rotations around the building-walls that cut $F$. If $P_F$ stabilizes $F$, then $F$ is a parabolic limit set. 
 \end{coro}
 
 This characterization yields the following corollary concerning the connectedness of the parabolic limit sets.

\begin{coro} \label{coro CompoconnexePLS}
Let $\partial P$ be a parabolic limit set. Then every connected component $F$ of $\partial P$ containing more than one point is a parabolic limit set.
\end{coro}

\begin{proof}
 Let $M$ be a building-wall that cuts $F$. Since $M$ cuts $\partial P$ a rotation $r\in \Gamma$ around $M$ stabilize $\partial P$ so in particular it permutes the connected components of $\partial P$. With $r(\partial M\cap F)= \partial M\cap F$  we deduce that $r(F)=F$ and so $F$ is a parabolic limit set.

\end{proof}

 Finally the following example illustrates the difficulty caused by parabolic limit sets in proving the CLP.
 
\begin{exmp} \label{ex translation PLS}
Let $M\in \bw$ be a building-wall of type $s$ along the base chamber $x_0$. Let $P =\mathrm{Stab}_\Gamma (M)$. The group  $P$ is the parabolic subgroup which  is   generated by the generators $r\in S\backslash\{s\}$ such that  $rs=sr$. Moreover, as we recalled in Subsection \ref{subsec BofLFRAHB}, $\partial P\simeq\partial M$.
 Now we assume that $\partial P$ is a proper parabolic limit set and we pick $g\in \Gamma$. Now  we verify that \liste{\item either $g\partial P = \partial P$, 
\item or $ \partial P\cap g\partial P=\emptyset$.}
Indeed if  two building-walls $M_1$ and $M_2$ are distinct with $M_1\perp M_2$ then $M_1$ and $M_2$ are of distinct types. As $M$ and $gM$ are of the same type it follows that if $M\cap gM\neq \emptyset$ then $M=gM$ and  $g\partial P = \partial P$.  On the other hand, if  $M\cap gM = \emptyset$, by Lemma \ref{lem intersec bord des murs} the hyperbolicty implies that $ \partial P\cap g\partial P=\emptyset$. 

  Finally the set $\cup_{g\in\Gamma} g   \partial M$ is the union of countably many disjoint copies of $\partial M$. In the introduction    we recalled  that an efficient way to prove the CLP is to follow curves using bi-Lipschitz maps (see Section \ref{sec stepsofproof}). As $\Gamma$ acts by bi-Lipschitz homeomorphisms on its boundary, the first idea is to use $\Gamma$ to follow curves. However if a non-constant curve $\eta$ is contained in $\partial M$ then we cannot hope to follow   the curves of $\borg$ using $\eta$ and $\Gamma$. 
\end{exmp}

\begin{figure}[h!]
 
\centering
\includegraphics[height=5cm]{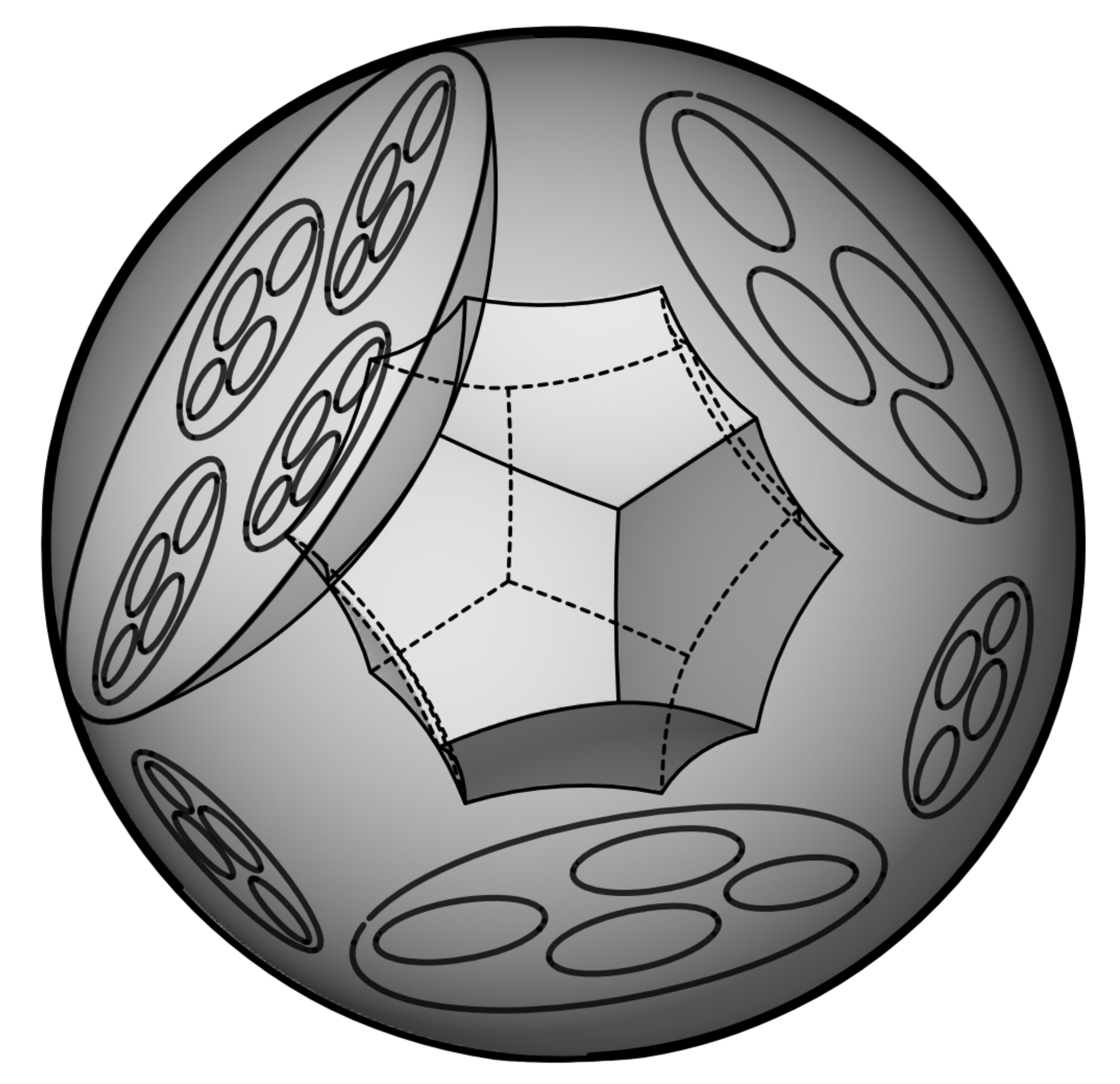}
\caption{Example \ref{ex translation PLS} on the boundary of a thin building}
\label{fig:D21}
\end{figure}

\subsection{Modulus of curves in connected parabolic limit set} 
\label{subsec control curveinpara}
In this subsection we apply the ideas of \cite[Section 5 and 6]{BourdonKleinerCLP} to $\Gamma$.
As in Subsection \ref{subsecdimconf},  $d_0$ denotes a small constant compared with $\dia{\borg}$ and with the constant of approximate self-similarity. Then $\F_0$ is the set of curves of diameter larger than $d_0$.
Here we prove that, from the point of view of the modulus, curves in a parabolic limit set are all the same (see consequences of   Theorem \ref{theocourbedanspara}).

Until the end of this  article, we  use the following notation:
\begin{nota}
Let $\partial P$ be a connected parabolic limit set in $\borg$. For $\delta,r >0$, let $\F_{\delta,r}(\partial P)$ denote \index{$\F_{\delta,r}(\partial P)$} the set of curves $\gamma$ in $\borg$ such that:  \liste{
\item $\dia{\gamma} \geq d_0$,
\item $\gamma \subset N_\delta(\partial P)$,
\item $\gamma \not\subset N_r(\partial Q)$ for any connected parabolic limit set $\partial Q \varsubsetneq \partial P$.
}
\end{nota}
 
 As we saw in Example \ref{ex translation PLS}, it is impossible to use the curves in the parabolic limit set to follow other curves.  Nevertheless, in this section we prove that these curves can be used to follow the curves in $\partial P$ (Proposition \ref{propdupli}). Then we deduce from this proposition a control of the modulus of the curves in parabolic limit sets (Theorem \ref{theocourbedanspara}). To this end we   use the following notion.

\begin{déf}
Let $L\geq0$ and $I$ a non-empty subset of $S$. A curve $\gamma$ in $\borg$ is called  a  \index{$(L,I)$-curve} \emph{$(L,I)$-curve} if \liste{\item $x_0 \subset \cv{\gamma}$,

\item for all $s\in I$, there exists a panel $\sigma_s$ of type $s$ inside $\cv{\gamma}$ with $\di{x_0,\sigma_s}\leq L$.}
\end{déf}

The next proposition says that curves in parabolic limit sets are $(L,I)$-curves.

\begin{prop} \label{propLI}
Let $I\subset S$ and $P=h \Gamma_I h^{-1}$. Then for all $r>0$, there exist $L>0$ and $ \delta >0$ such that if $x_0 \subset \cv{\gamma}$ and $\gamma \in \F_{\delta,r}(\partial P)$, then $\gamma$ is a $(L,I)$-curve.
\end{prop}
\begin{proof}
We fix $r>0$ and we assume that for every integer $n\geq1$, there exists a curve $\gamma_n$ such that:\liste{\item $x_0 \subset \cv{\gamma_n}$, \item  $\gamma_n \in \F_{1/n,r}(\partial P)$, \item $\gamma_n$ is not a $(n,I)$-curve.} For  $n\geq1$, we designate the ball of center $x_0$ and of radius $n$ for the metric over the chambers by \[B_c(x_0,n)=\{x\in \ch : \dc{x_0,x}\leq n \}.\]
For simplicity we also designate by $B_c(x_0,n)$ the union of its chambers.
 Up to a subsequence we can suppose  that for a fixed $s\in I$, there is no panel of type $s$ in $B_c(x_0,n)\cap \cv{\gamma_n}$ for $n\geq1$.

We want to reveal a contradiction using the sequences $\{\gamma_n\}_{n\geq 1}$ and $\{\cv{\gamma_n}\}_{n\geq 1}$. According to \cite[p. 281]{MunkresTopo}
 the set of non-degenerate  continua in a compact metric space is a compact metric space with respect to the Hausdorff distance. Therefore, up to a subsequence, we can suppose that $\gamma_n$ tends to a non-degenerate continuum $\mathcal{L}\subset \partial P$. 

Since $x_0 \subset \cv{\gamma_n}$, using a diagonal argument we can also assume that, up to a subsequence, $\cv{\gamma_{n+k}}\cap B_c(x_0,n)$ is non-empty and constant for $k\geq 0$. We denote $C:= \bigcup_{n\geq 1} \cv{\gamma_{n }}\cap B_c(x_0,n)$. As we remarked before, if $F\subset \partial \Gamma$ contains more than one point  $F\subset \partial \cv{F}$. In particular   $\mathcal{L}\subset \partial C$. 

Now, by the assumptions on the sequence $\{\gamma_n\}_{n\geq 1}$ we observe that $C$ does not contain any panel of type $s$. Hence $\mathcal{L}$ is contained in the limit set of a parabolic subgroup of the form $g\Gamma_J g^{-1}$ with $s \notin J$. Moreover, we know that  intersections of parabolic subgroups are parabolic subgroups (see  \cite[Theorem 1.2]{ClaParaResidinBuildings}). Therefore $\mathcal{L}$ is contained in the limit set of the parabolic subgroup $Q'=P\cap g\Gamma_J g^{-1}$.  Let $\partial Q$ be the connected component of $\partial Q'$ that contains  $\mathcal{L}$. Thanks to Corollary \ref{coro CompoconnexePLS}, $\partial Q$ is a parabolic limit set. Now, as $\gamma_n$ tends to $\mathcal{L} \subset \partial Q$ with respect to the Hausdorff distance, we have that for $n\geq 0$ large enough $\gamma_n \subset N_r(\partial Q)$. 

Finally, if $\partial Q \varsubsetneq \partial P$, this reveals a  contradiction with $\gamma_n \in \F_{1/n,r}(\partial P)$ for every $n \geq 1$. If $\partial Q = \partial P$, then we apply the same reasoning with $s'\in I\backslash \{s\}$.

\end{proof}

An interesting feature of  $(L,I)$-curves is that these curves are  crossed by building-walls of type in $I$. Which means that from a $(L,I)$-curve,  we can follow curves using rotations around building-walls of type in $I$.

\begin{prop} \label{propdupli}
Let $\epsilon>0$, $L>0$ and $I$ be a non-empty subset of $S$. For $P=h\Gamma_I h^{-1}$, let  $\eta$ denote a curve contained in $\partial P$. Then there exists a finite subset $F\subset \Gamma$ such that for any $(L,I)$-curve $\gamma$ the set $\bigcup\limits_{g\in F} g \gamma$ contains a curve that belongs to $\U_\epsilon(\eta)$.
\end{prop}

\begin{proof}
We divide the proof in four steps. In this proof $M_s$ denotes the building-wall of type $s\in S$ along $x_0$.

i) First, we can suppose without loss of generality that $P= \Gamma_I$. Indeed, as $h\in \Gamma$ is a bi-Lipschitz homeomorphism of $(\borg,d)$, then if the property holds for $\Gamma_I$ it holds for $h\Gamma_I h^{-1}$.

ii) Now we prove that the following property holds. \begin{center}
\begin{minipage}[c]{12cm}\emph{The set $\bigcup_{g\in F_L} g \gamma$ contains a curve passing through $\partial M_s$ for every $s \in I$,}
\end{minipage}
\end{center}
where $F_L =\{g\in \Gamma : \vert g \vert \leq L  \}$ and $\vert g \vert= \dc{x_0,g x_0}$.

As $\gamma$ is a $(L,I)$-curve, for any $s\in I$ there exist $\alpha \in \Z$ and $g \in F_L$ such that $g x_0$ and $gs^\alpha x_0$ belong to $\cv{\gamma}$. We fix $s\in I$, $\alpha \in \Z$ and $g \in F_L$ as before and let $x_0\sim g_1 x_0 \sim  \dots \sim  g_\ell x_0$ be a gallery contained in $\cv{\gamma}$ with $g_{\ell-1}=gs^\alpha $ and $g_{\ell }=g $. For any $i= 0, \dots , \ell-1$, let  $M_i$ denote the building-wall separating $g_i x_0$ and $g_{i+1} x_0$. In particular,   $\partial M_i$  cuts $\gamma$ for any $i= 0, \dots , \ell-1$. This means that if  $M_i$ is of type $s_i$, then $\gamma \cap g_i s_i^{\alpha_i}  g_i^{-1} \gamma \neq \emptyset$ for any $\alpha_i \in \Z$. In particular, if $\alpha_i$ is such that $g_{i+1}= g_i s_i^{-\alpha_i}$ then  $\gamma \cap g_i g_{i+1}^{-1} \gamma \neq \emptyset$ for any $i= 0, \dots , \ell-1$.
  Hence the set $\gamma\cup g_{1}^{-1} \gamma \cup \dots \cup g_{\ell}^{-1} \gamma$ is arcwise connected and $ g_{\ell}^{-1} \gamma$ intersects $\partial M_s$. Thus the property is satisfied.

iii) Let  $\Sigma_I\subset \Sigma$ be the residue associated with $\Gamma_I$. We recall that this means $\Sigma_I=\Gamma_I x_0$. For each $1\leq i \leq k$ let $D_i$  be a dial of building bounded by the building-wall $M_i$. We assume that each $D_i$ intersects   $\Sigma_I$ properly (\emph{i.e.} $\Sigma_I \cap D_i \neq \emptyset$ and  $\Sigma_I \cap D_i \neq \Sigma_I$).  In particular, this means that the building-walls $M_1, \dots, M_k$ have their types contained in $I$.

 Now we prove that the following property holds. \begin{center}
\begin{minipage}[c]{12cm} \emph{There exists a finite subset $F_0\subset \Gamma$ such that for every $(L,I)$-curve $\gamma$ the set $\bigcup_{g \in F_0} g \gamma$ contains a curve passing through $\partial D_1, \dots, \partial D_k$.}
\end{minipage}
\end{center}
 For   $i=1,\dots , k$ pick $h_i \in \Gamma_I$ such that $M_i$ is along   $h_i x_0 \in \Sigma_I$. In particular, for any $i$, we can write  $M_i=h_i( M_s)$ where  $s\in I$ is the type of $M_i$. Let $g_1x_0 = h_1 x_0 \sim g_2 x_0 \sim  \dots \sim  g_\ell x_0 = h_k x_0$ be a gallery in $\Sigma_I$ passing through $h_1 x_0, \dots, h_k x_0$ in this order. 
 
Applying the second step of the proof, there exists  a curve $\theta$ in $\bigcup_{g \in F_L} g \gamma$ such that $\theta$ crosses $\partial M_s$ for  every $s\in I$. Therefore the set $\bigcup_{i=1,\dots , \ell} g_i\theta$ meets $g_i(M_s)$ for any $i=1,\dots,k$ and any $s\in I$. In particular,   it meets every $h_i (M_s)$ and  intersects every $\partial D_1, \dots, \partial D_k$.
 
We set $F_0 = \{g_ig\in \Gamma : \vert g \vert \leq L, 1\leq  i\leq \ell  \}$, and it is now enough to check that $\bigcup_{g \in F_0} g \gamma$  is arcwise connected.  For any $i=1,\dots , \ell-1 $ let $s_i\in I$ and $\alpha_i \in \Z$ be such that $g_{i+1}=g_i s_i^{\alpha_i}$. Then   $g_{i+1} \theta = (g_i s_i^{\alpha_i} g_i^{-1}) g_i \theta$. Since $\theta$ intersects any $\partial M_{s_i}$  then $g_{i} \theta \cap g_i(\partial M_{s_i}) \neq \emptyset$  and this intersection is fixed by $g_i s_i^{\alpha_i} g_i^{-1}$. Thus $g_{i} \theta \cap g_{i+1}\theta \neq \emptyset$.

iv) With Fact \ref{fact equivdestopo}, we can choose,  $D_1', \dots, D_{k+1}'$   a collection of dials of building such that the union of their boundaries is a neighborhood of $\eta$ contained in the $\epsilon/ 2 $ neighborhood of $\eta$.  We also assume that  $\eta$ enters in the boundaries of the  $D_1', \dots, D_{k+1}'$ in this order. For any $i=1,\dots,k+1$, let $r_i$ denote the rotation around the building-wall associated with $D_i'$. Let $D_1, \dots, D_{k}$ be a collection of dials of building such that  $\partial D_i \subset \partial D_i'\cap \partial D_{i+1}'$.  Applying the previous step of the proof, there exists a finite set $F_0\subset \Gamma$ such that for every $(L,I)$-curve $\gamma$ the set $\bigcup_{g\in F_0} g\gamma$ contains a curve passing through each $\partial D_1, \dots, \partial D_{k}$.

If for some $i=1,\dots,k+1$ the curve $\eta$ leaves $\partial D_i'$ then   $\theta \bigcup_{\alpha\in \Z} r_i^\alpha \theta$ contains a curve that does not leave $\partial D_i'$.
 Finally we set $F=\{r_i^\alpha g : \alpha\in \Z \text{ and } g \in F_0  \}$ and $F$ satisfies the desired property.
\end{proof}

  We   use the two preceding propositions to obtain a control  of $\modcomb{p}{\F_{\delta,r}(\partial P)}$. 
\begin{theo} \label{theocourbedanspara} There exists an increasing  function $\delta_0 : (0,+\infty) \longrightarrow (0,+\infty)$ satisfying the following property. Let $p\geq 1$, let $\eta \in \F_0$,  and let $\partial P$ be the smallest parabolic limit set containing $\eta$. Let $r>0$ be  small enough so that $\eta \not\subset \overline{N_r}(\partial Q ) $ for any connected parabolic limit set $\partial Q \varsubsetneq \partial P$ and  let $\delta<\delta_0 (r)$. Then for  $\epsilon>0$ small enough  there exists a constant $C=C(d_0,p,\eta,r,\epsilon)$ such that for every $k\geq 1$
\[C^{-1}\cdot \modcomb{p}{\U_\epsilon(\eta),G_k} \leq \modcomb{p}{\F_{\delta,r}(\partial P), G_k}\leq C \cdot \modcomb{p}{\U_\epsilon(\eta), G_k}.\]
Furthermore when $p$ belongs to a compact subset of $[ 1, + \infty)$ the constant  $C$ may be chosen independent of $p$.
\end{theo}
 \begin{proof} 
 
i) We can assume,  without loss of generality,   that $x_0 \subset \cv{\gamma}$ for every $\gamma \in \F_0$. Indeed, there exists an upper bound $N$ depending on $d_0$ such that $\di{x_0,\cv{\gamma}}\leq N$  for every $\gamma \in \F_0$. So there exists only a finite set $E$ of elements of $\Gamma$ such that for $g\in E$, there exists $\gamma \in \F_0$ with $\di{x_0,\cv{\gamma}}= \dc{x_0,gx_0}$. 
Now, for $\F\subset \F_0$,  we write \[\F^{x_0} = \{\gamma \in \F : x_0 \subset \cv{\gamma} \}.\] Applying  Proposition \ref{propscale} $N$ times and Proposition \ref{propmodbase},  we obtain that for every $k\geq 1$ large enough \[\modcombg{\F^{x_0}}\leq  \modcombg{\F  }\leq C \cdot \modcombg{\F^{x_0}}, \]
where $C$ depends only on $d_0$.
 
ii) Similarly, for $\epsilon>0$ small enough, we can assume that  $\mathcal{U}_\epsilon(\eta)\subset \F_0$.  Indeed, for $\epsilon>0$ small enough there exists an upper bound $N$ depending on $d_0$ such that if $\gamma \in \mathcal{U}_\epsilon(\eta)$ then $\di{x_0,\cv{\gamma}}\leq N$. Again, the multiplicative constant induced by this assumption depends  only on $d_0$. 

iii) As a consequence of the preceding part, for $\epsilon>0 $ small enough, one has $\mathcal{U}_\epsilon(\eta)\subset \F_{\delta,r}(\partial P)$ and   the left-hand side inequality is established by  Proposition \ref{propmodbase} (1).  

 iv) Now we prove the right-hand side inequality. Let $P=h\Gamma_Ih^{-1}$, let $\eta$ be a curve in  $\partial P$, $r>0$ as in the hypothesis of the  theorem and $\epsilon>0$ small enough so that the preceding part hold.

With the assumption of the first part, we can apply   Proposition \ref{propLI} and set $L>0$ and $\delta>0$ such that the curves of $\F_{\delta,r}(\partial P)$ are $(L,I)$-curves. Let $F\subset \Gamma$ be the finite set given by   Proposition \ref{propdupli} and let $\rho : G_k \longrightarrow [0,+\infty)$ be a $\mathcal{U}_\epsilon(\eta)$-admissible function. We define $\rho' : G_k \longrightarrow [0,+\infty)$ by: 

\[(*) \ \ \ \ \ \rho' ( v) = \sum\limits_{g\in F}\sum\limits_{w\cap g v \neq \emptyset} \rho(w).   \] Let $\gamma \in \F_{\delta,r}(\partial P)$ and $\theta \subset \bigcup\limits_{g\in F} g \gamma$ such that $\theta \in \U_\epsilon(\eta)$. Then 
\[L_{\rho'}(\gamma) = \sum\limits_{g\in F}\sum\limits_{v\cap\gamma\neq \emptyset}\sum\limits_{w\cap g v \neq \emptyset} \rho(w)\geq  \sum\limits_{g\in F} \sum\limits_{w\cap g\gamma\neq \emptyset} \rho(w).\]However \[L_{\rho}(\theta) \leq \sum\limits_{g\in F} L_\rho (g\gamma) = \sum\limits_{g\in F} \sum\limits_{w\cap g\gamma\neq \emptyset} \rho(w).\] Thus $L_{\rho'}(\gamma)\geq L_\rho (\theta)$ and $\rho'$ is $\F_{\delta,r}(\partial P)$-admissible.

Then the number or terms in the right-hand side of the definition $(*)$ is bounded by a constant $N$ depending on $\# F$, the bi-Lipschitz constants of the elements of $F$, and  the doubling constant of $\borg$. Therefore by convexity \[M_p(\rho')= \sum\limits_{v\in G_k} \Big(\sum\limits_{g\in F}\sum\limits_{w\cap g v \neq \emptyset} \rho(w)\Big)^p \leq N^{p-1} \cdot \sum\limits_{v\in G_k} \sum\limits_{g\in F}\sum\limits_{w\cap g v \neq \emptyset} \rho(w)^p  \leq N^p \cdot \# F \cdot \sum\limits_{w\in G_k}   \rho(w)^p . \] 
Which  proves the inequality. The multiplicative constant induced by this part depends on $p$, $\eta$ and $r$ as the set $F$ depends on $\eta$ and $r$. However it does not depend on $k$.

v) Here we write $\delta(r,\partial P)$ for the constant $\delta$ chosen in the previous part such that  the curves of $\F_{\delta,r}(\partial P)$ are $(L,I)$-curves for a certain $L>0$.

By Proposition \ref{proptaillepara} there exist only a finite number of parabolic limit set of diameter larger than  $d_0$. Moreover we recall that according to our notation, if $\delta\geq \delta'$ then $ \F_{\delta',r}(\partial P) \subset \F_{\delta,r}(\partial P)$. As a consequence, the function defined by \[ \delta_0 ( r ) = \min{\{\delta(r,\partial P) : \dia{\partial P}\geq d_0 \}},\] 
satisfies the desired property.

vi) As we see at the end of part iv), $ C= \lambda. N^p$ with $\lambda$ and $N$ independent of $p$. Hence if $p$ belongs to a compact subset $K \subset [ 1, + \infty)$, the constant  $C$ may be chosen independent of $p$ by taking  $ C= \lambda. N^{\max{K}}$.

\end{proof}
 
As an immediate application, we notice that under the assumptions of the theorem, the behavior of $ \modcomb{p}{\mathcal{U}_\epsilon(\eta),G_k}$ as $k$ goes to infinity does not depend,  on the choice of $\eta$ and $\epsilon$. Indeed, for $p\geq 1$, $r>0$ and $\delta<\delta_0(r)$ fixed, if $\eta, \eta' \subset \partial P$ and $\epsilon, \epsilon' >0$ are such  that the  hypothesis of the theorem are satisfied. Then there exist $C=C(\eta,\epsilon)$ and $C'=C'(\eta',\epsilon')$ such that \[ C^{-1} \cdot \modcomb{p}{\mathcal{U}_{\epsilon'}(\eta'),G_k} \leq \modcomb{p}{\mathcal{U}_\epsilon(\eta),G_k}\leq C' \cdot \modcomb{p}{\mathcal{U}_{\epsilon'}(\eta'),G_k}. \]Of course, if $\eta=\eta'$ and $\epsilon'<\epsilon$ we can choose $C=1$.

\subsection{Application to Fuchsian buildings}

Another consequence of the theorem is that if the boundary of  a graph product   does not contain connected parabolic limit sets, then it satisfies the CLP.
 
\begin{theo}\label{theoapplic au imm de dim2} Let $\Gamma$ be a thick hyperbolic graph product such that $\borg$ is connected and any proper  parabolic limit set is disconnected. Then $\partial \Gamma$  equipped with a visual metric satisfies the CLP.
\end{theo}

\begin{proof}

 We  check the hypothesis  of Proposition \ref{propcritCLP}. To prove that  $\modcomb{1}{\F_0,G_k}$  is unbounded, it is enough to check that for every $N \in \N$ there exist $N$ disjoint curves of diameter larger than $d_0$ in $\borg$. Indeed, this implies that for $k\geq 0$ large enough $\modcomb{1}{\F_0,G_k}>N$.

 To prove this we use the assumption on the thickness which implies that   for every $N \in \N$ there exist $N$ apartments with disjoint boundaries that intersects in a compact domain inside the building. To observe such apartments we  use the following notation. \liste{\item $ W_{i-1} = \{w\in \Gamma : w=s_1^i\dots s_k^i \text{ with }   s_j\in S \text{ and } \dc{x_0,wx_0}=k\}$   for $i=1,2$.
 \item $\mathrm{CB}_n= \cv{B_c(x_0,n)}\subset \ch $ the convex hull for the metric over the chambers $\dc{\cdot,\cdot}$ of $B_c(x_0,n)\subset \ch$   the ball of center $x_0$ and of radius $n$ for $\dc{\cdot,\cdot}$ for $n\geq 0$.
 \item  $\mathrm{FCB}_n$ is the frontier of $\mathrm{CB}_n$. By this we mean the  set of all the chambers $x\in   \mathrm{CB}_n$ such that there exists $y \notin   \mathrm{CB}_n$ with $x\sim y$.
 }
 
 Then the following sets of chambers define $N$ apartments all containing the base chamber $x_0$, and intersecting only inside $\mathrm{CB}_N$
 
 \liste{\item $A_1=\{w x_0 : w \in W_0  \}$,
 \item $A_2=\{w x_0 : w \in W_1  \}$,
 \item $A_3=A_2 \cap \mathrm{CB}_1 \bigcup \{g w x_0 : w \in W_0,  \  gx_0 \in A_2 \cap \mathrm{FCB}_1  \}$,
 \item $A_n= A_{n-1} \cap \mathrm{CB}_{n-2} \bigcup \{g w x_0 : w \in W_{n+1\ \mathrm{mod}(2)} \text{ and }  gx_0 \in A_{n-1} \cap \mathrm{FCB}_{n-2} \}$ for $n=2,\dots, N$. }

Now let $\eta$ be a non-constant curve in $\partial \Gamma$. Up to a change of scale, by  Proposition \ref{propscale}, we can assume $\eta\in \F_0$. Then as $\partial \Gamma$ is the only parabolic limit set containing $\eta$, it is enough to apply   Theorem \ref{theocourbedanspara} to satisfy the second hypothesis of   Proposition \ref{propcritCLP}.

\end{proof}

In particular, we can  apply this result to the case of right-angled Fuchsian buildings.  In the following, we call \emph{right-angled Fuchsian building} \index{Fuchsian building} a building associated with a graph product $(C_n, \{\Z/q_i\Z\}_{i=1,\dots, n})$ where   $C_n$ is the cyclic graph with $n\geq 5$ vertices and   $q_1, \dots, q_n$ is  a family of integers larger than or equal to $3$.

\begin{coro}
For $n\geq 5$, let $C_n$ be the cyclic graph with $n$ vertices and let $q_1, \dots, q_n$ be a family of integers larger than or equal to $3$. Let $\Gamma$ be the graph product given by the pair $(C_n, \{\Z/q_i\Z\}_{i=1,\dots, n})$. Then $\borg$ equipped with a visual metric satisfies the CLP.  
\end{coro}

This result was known since  boundaries of right-angled Fuchsian buildings are Loewner spaces (see \cite[Proposition 2.3.4.]{BourdonPajotRigi}). However, here we give a direct proof of this result.

Furthermore, we can prove that these thick graph products are the only ones that satisfy the hypothesis of Theorem \ref{theoapplic au imm de dim2}. To verify this we need to introduce the following simplicial complex.

\begin{déf}
Let $\Gamma =  (\mathcal{G}, \{\Z/q_i\Z\}_{i=1,\dots, n})$ be a graph product, the \emph{nerve of $\Gamma$} is the simplicial complex $L=L(\mathcal{G})$ such that:\liste{
\item the 1-skeleton of $L$ is $\mathcal{G}$,
\item $k$ vertices of $\mathcal{G}$ span a $(k-1)$-simplex in $L$ if and only if the corresponding parabolic subgroup in $\Gamma$ is finite. 
}
\end{déf}

For  a  simplex $\sigma \subset L$ spanned by the vertices $v_1,\dots,v_k\in \mathcal{G}^{(0)}$, we denote by $\mathcal{G}_\sigma$ the full sub-graph of $\mathcal{G}$ spanned by the vertices $\mathcal{G}^{(0)}\backslash \{v_1,\dots,v_k\}$. Now we write  $L  \backslash \sigma = L(\mathcal{G}_\sigma)$.  Note that $L  \backslash \sigma$ can be seen as a subcomplex of $L$.
 
The following theorem is a special case of  \cite[Corollary 5.14.]{DavisMeierTopo}. 
 
 \begin{theo} \label{theo davismeier}
The boundary of $\Gamma$ is connected if and only if  the subcomplex $L  \backslash \sigma$ is connected for any simplex $\sigma \subset L$.
\end{theo}

Now we can prove the following.

\begin{prop} \label{prop caracdesFuchsin}
 Let $\Gamma =  (\mathcal{G}, \{\Z/q_i\Z\}_{i=1,\dots, n})$  be a hyperbolic graph product. Assume that $\borg$ is connected and  that any proper parabolic limit set $\partial P$ is disconnected, then the building associated with $\Gamma$ is a right-angled Fuchsian building.
\end{prop}

\begin{proof} We only need to prove  that $\mathcal{G}$ contains a circuit of length $n\geq 5$. According to Corollary \ref{coro CompoconnexePLS}, if any proper parabolic limit set in $\borg$ is disconnected then any proper parabolic limit set in $\borg$ is discrete. Moreover, $\borg$ contains at least one proper parabolic limit set of the form $\partial \Gamma_I$ with $\# I = n-1$ otherwise $\borg=\emptyset$. The subgroup $\Gamma_I$ is a graph product associated with the graph  $\mathcal{G}_I$. This  graph is obtained from $\mathcal{G}$ to which we remove a vertex $p$ and all the edges adjacent to $p$. Then if $L_I$ is the nerve associated with $\Gamma_I$, we get $L_I$ from $L$ to which we remove the interior of any simplex containing $p$.

Now, thanks to Theorem \ref{theo davismeier}, we know that there exists a simplex $\sigma\subset L_I$ such that $L_I \backslash \sigma$ is disconnected. Let $C_1$ and $C_2$ be two connected components of $L_I \backslash \sigma$. Up to  a subsimplex, we can assume that any vertex of $\sigma$ is connected to $C_1$ or to $C_2$ by an edge. However, if we consider  the simplex $\sigma$ in $L$, we see that $L \backslash \sigma$ is connected because $\borg$ is connected. Therefore there exist at least one edge attaching $p$ to $C_1$ and  at least one edge attaching $p$ to $C_2$.

 We set  $V=\{v_1, \dots , v_k\}$   the  vertices of $\sigma$ that are not connected to $p$ by an edge and $V'=\{v'_1,\dots, v'_{k'}\}$ the rest of the vertices of $\sigma$. At this point, we assume by contradiction, that $\mathcal{G}$ contains no circuit of length $n\geq 4$. We can check  that under this assumption the following situations does not occur

\begin{figure}[ht] 
  \begin{minipage}[b]{0.5\linewidth}
  \labellist
\small\hair 2pt
\pinlabel $C_1$ at 100 340 
\pinlabel $C_2$ at 1100 340 
\pinlabel $p$ at 590 590 
\pinlabel $\sigma$ at 585 140
\endlabellist
\centering
    \includegraphics[width=.49\linewidth]{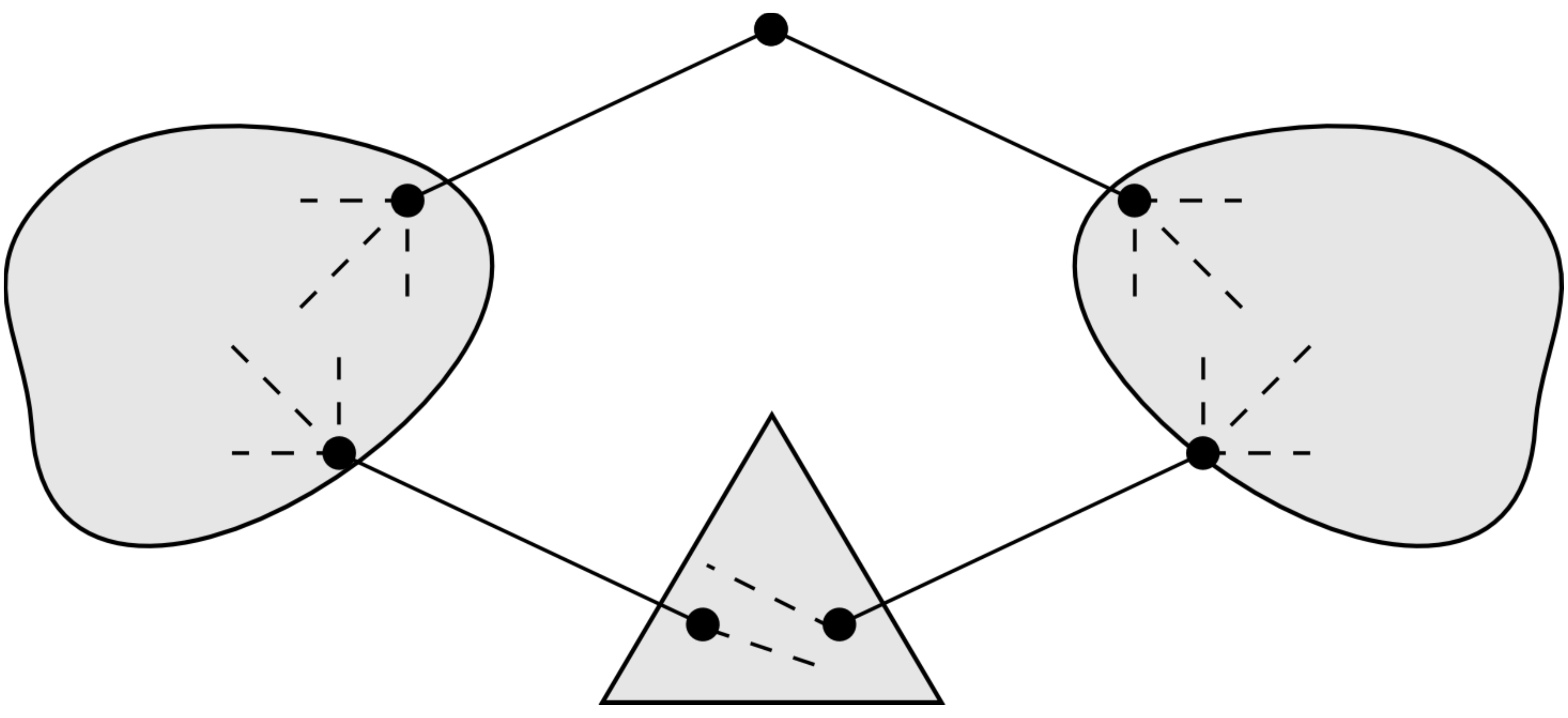}
    \caption{Forbidden situation $i)$} 
  \end{minipage} 
  \begin{minipage}[b]{0.5\linewidth}
   \labellist
\small\hair 2pt
\pinlabel $C_1$ at 100 340 
\pinlabel $C_2$ at 1100 340 
\pinlabel $p$ at 590 590 
\pinlabel $\sigma$ at 585 50
\endlabellist
\centering
    \includegraphics[width=.49\linewidth]{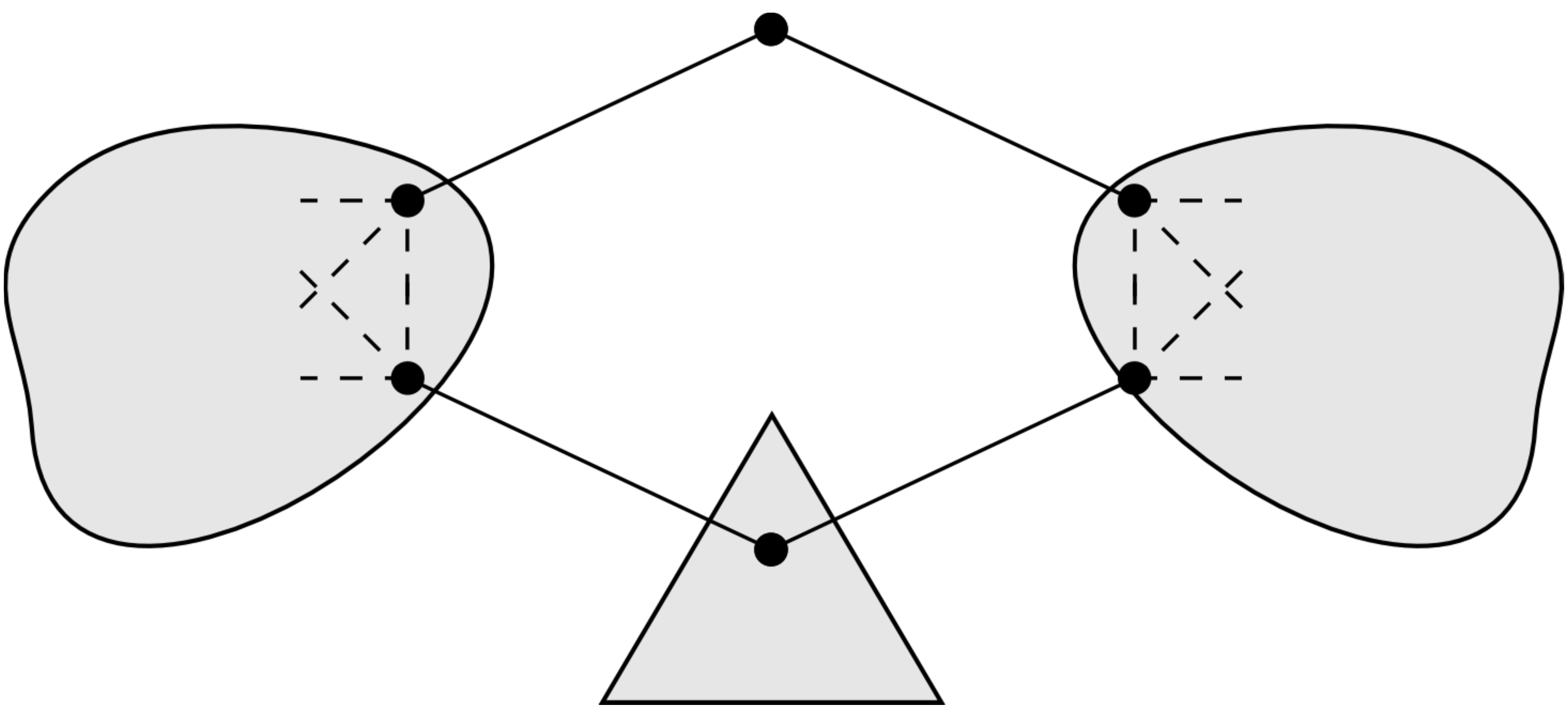}
    \caption{Forbidden situation $ii)$} 
  \end{minipage}

\vspace{0.5cm}
 
  \begin{minipage}[b]{0.5\linewidth}
   \labellist
\small\hair 2pt
\pinlabel $C_1$ at 100 340 
\pinlabel $C_2$ at 1100 340 
\pinlabel $p$ at 590 590 
\pinlabel $\sigma$ at 585 120
\endlabellist
\centering
    \includegraphics[width=.49\linewidth]{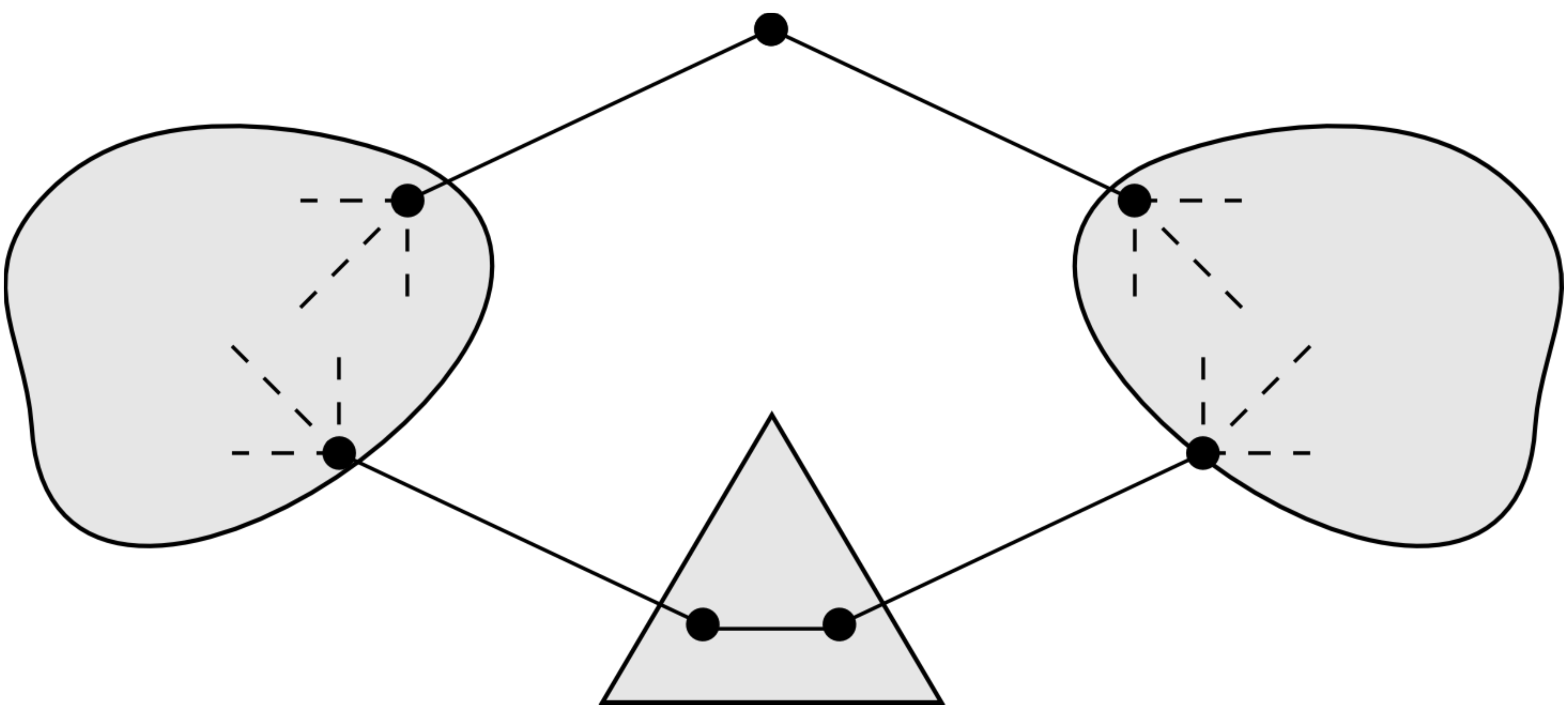}
    \caption{Forbidden situation $iii)$} 
  \end{minipage}
  \hfill
  \begin{minipage}[b]{0.5\linewidth}
   \labellist
\small\hair 2pt
\pinlabel $C_1$ at 100 340 
\pinlabel $C_2$ at 1100 340 
\pinlabel $p$ at 590 590

\pinlabel $\sigma'$ at 670 360
\pinlabel $V$ at 405 25

\endlabellist
\centering
    \includegraphics[width=.49\linewidth]{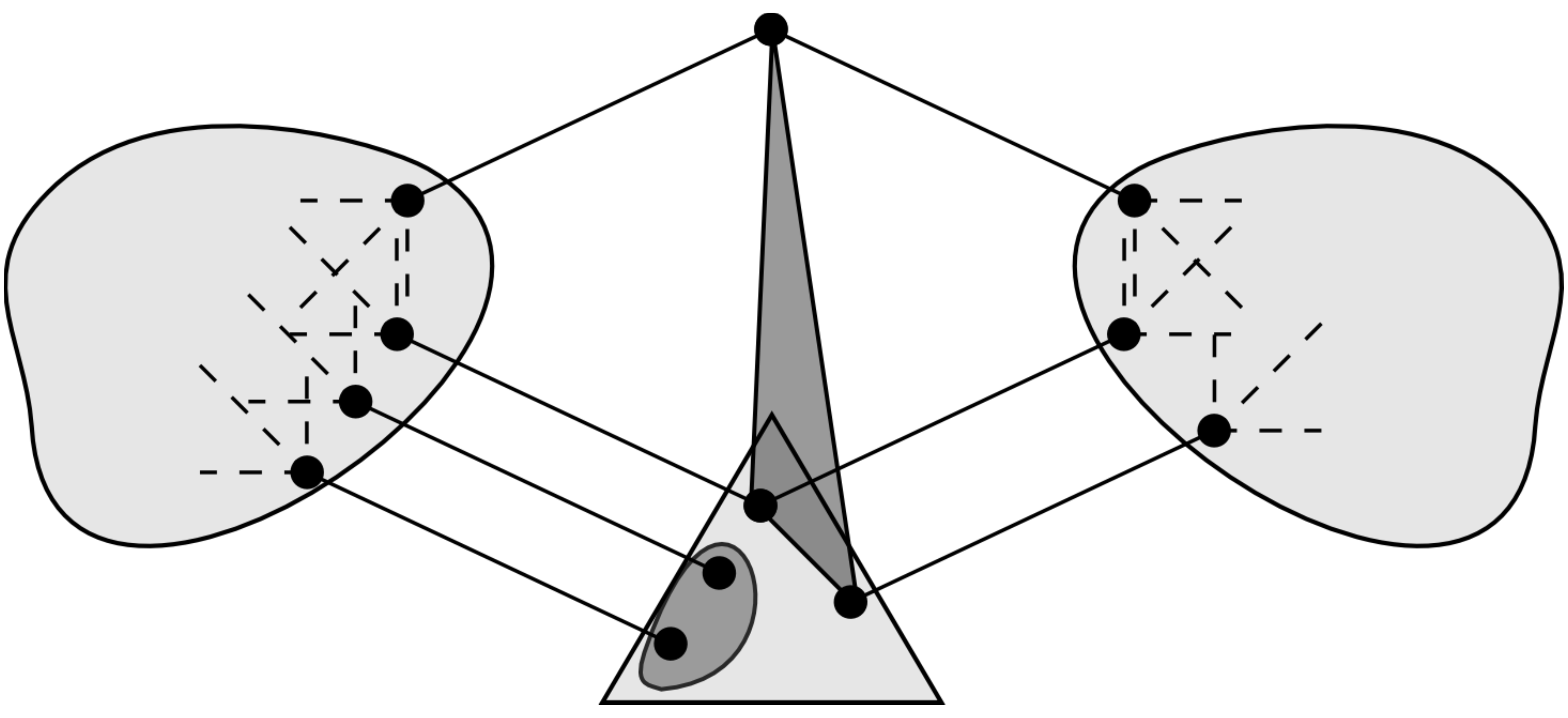}
    \caption{Resulting situation} 
  \end{minipage} 
\end{figure} 
 
  \numeroti{
\item $V'$ is empty,
\item there exists $v \in V $ such that $v $ is adjacent to both $C_1$ and $C_2$, 
\item there exist $v,w \in V$ such that $v$ is adjacent to $C_1$ and $w$ is adjacent to $C_2$.
}

Hence,   the vertices in $V$ are either  all adjacent to $C_1$ or  all adjacent to $C_2$. Assume that   the vertices in $V$ are  all adjacent to $C_1$. As a consequence, if $\sigma'$ designates the simplex in $L$ spanned by $V'\cup \{p\}$ then $L\backslash \sigma'$ is not connected. This is not possible because $\borg$ is connected. 

Therefore $\mathcal{G}$ contains a circuit of length $n\geq 4$, but as $\Gamma$ is hyperbolic it contains no circuit of length $4$. This concludes the proof.
 
\end{proof}

\section{Combinatorial metric on boundaries of right-angled hyperbolic buildings}
\label{sec bord topo met}

 In this section  we explain how   the geometry of the boundary is determined by  boundaries of   building-walls. 
We start by discussing the geometry of intersections of dials of building and the boundaries of such intersections. Then, we describe  a   combinatorial and  self-similar metric on $\borg$ in terms of dials of building. Finally,  we construct an approximation of $\borg$ that will be  convenient to use in Section \ref{secmoduleapartmoduleimmeuble}. 

Here we use the notation and assumptions of Section \ref{sec LFRAHB} and \ref{seccurves inPLS}. In particular, $\Gamma$ is a fixed graph product given by the pair $(\mathcal{G}, \{\Z/q_i\Z\}_{i=1,\dots,n})$ and acting on the building $\Sigma$. The base chamber is   $x_0$, and $W$ is the right-angled Coxeter group associated with $\Gamma$. We  assume that $\Gamma$ is hyperbolic and $\borg$ is connected.

 \subsection{Projections of chambers in $\Sigma$}

  In right-angled buildings, we can project chambers on  residues and on dials of building. This will be useful in the rest of this section to understand the metric on the boundary in the hyperbolic case.
   
\begin{prop}  \label{prop deffrojdial} Let $D$ be a residue or a dial of building  and  $C = \Ch{D}$. Then for any  $x \in \ch$ there exists a unique chamber $\proj{C}{x} \in C$ \index{$\proj{C}{\cdot}$} such that \[\dc{x,\proj{C}{x}}= \di{x, C}.\]   Moreover, for any chamber  $y \in C$ there exists a minimal gallery from  $x$ to $y$ passing through   $\proj{C}{x}$. \end{prop} 

For  simplicity, in the following we sometime use the notation $\proj{D}{\cdot}$ \index{$\proj{D}{\cdot}$} to  designate $\proj{\Ch{D}}{\cdot}$.
Before the proof of the proposition we need to establish the following lemma.
  \begin{lem}  \label{lem stabiortho}
  Let $M$ and $M'$ be two building-walls such that $M \perp M'$. Let $r\in \Gamma$ be a rotation around $M$ and $D'$ be a dial of building bounded by $M'$ then $r(D')=D'$.  
  \end{lem}

  \begin{proof}
Up to a translation on the dials and a conjugation on the rotations we can assume that $M$ and $M'$ are along $x_0$ and $x_0\subset D'$. If $s$ is a rotation around $M'$, we  can write 
\[\Ch{D'}=\{x \in \ch : \dc{x_0,x}<\dc{x_0,sx} \}.\]  
Hence \begin{align*}
  r(\Ch{D'})=& \{rx \in \ch : \dc{x_0,x}<\dc{x_0,sx}   \}    \\
      =&  \{x \in \ch : \dc{x_0,r^{-1}x}<\dc{x_0,sr^{-1}x}\}.
\end{align*}

By assumption $rs=sr$, thus $r(\Ch{D'})= \{x \in \ch : \dc{r x_0,x}<\dc{r x_0,sx}   \}$. Moreover, $\dc{x_0,r x_0}=1$ and $\dc{x_0,sr x_0}=2$ so $r x_0\in \Ch{D'}$. With Fact \ref{prop defcadrangp} we obtain  $\Ch{D'}= r(\Ch{D'})$.
  \end{proof}

\begin{proof}[Proof of Proposition  \ref{prop deffrojdial}]
If $D$ is a residue, then we refer to \cite[Proposition 3.19.3]{TitsBuildingsLectureNotes}. If $D$ is a dial of building, let $y\in C$ be such that $\dc{x,y}= \di{x, C}$. Then for   $z \in C$ we set  $x=x_1\sim x_2\sim \dots\sim y$ and $ y=x_\ell \sim \dots \sim z=x_k$ two  minimal galleries. Assume that the gallery   \[x=x_1\sim  x_2\sim \dots\sim y=x_\ell \sim \dots \sim z=x_k \] is not minimal. Then there exists a  building-wall $M$ and two indices $i,j$ with $1\leq i < \ell$ and $\ell\leq j < k$ such that \liste{
\item $M$ separates $x_i$ and $x_{i+1}$,
\item $M$ separates $x_j$ and $x_{j+1}$.
}
Now consider $r\in \Gamma$ the rotation around $M$  such that $r  x_{i+1}=x_i$. By Proposition \ref{lem stabiortho}, $r(D)=D$ hence, the gallery \[x \sim \dots \sim  x_i\sim r  x_{i+2} \dots\sim r  x_{\ell} = r y\] connects $x$ to $C$ and is of length $\di{x, C}-1$, which is a contradiction.
 
\begin{figure}[ht!]
 \labellist
\small\hair 2pt

\pinlabel $x$ at 30 480
\pinlabel $x_{i}=rx_{i+1}$ at 370 425 
\pinlabel $x_{i+1}$ at 355 255
\pinlabel $ry$ at 740 430
 \pinlabel $y$ at 740 160
 \pinlabel $C$ at 785 80
 \pinlabel $M$ at 800 370

\endlabellist
\centering
\includegraphics[height=4cm]{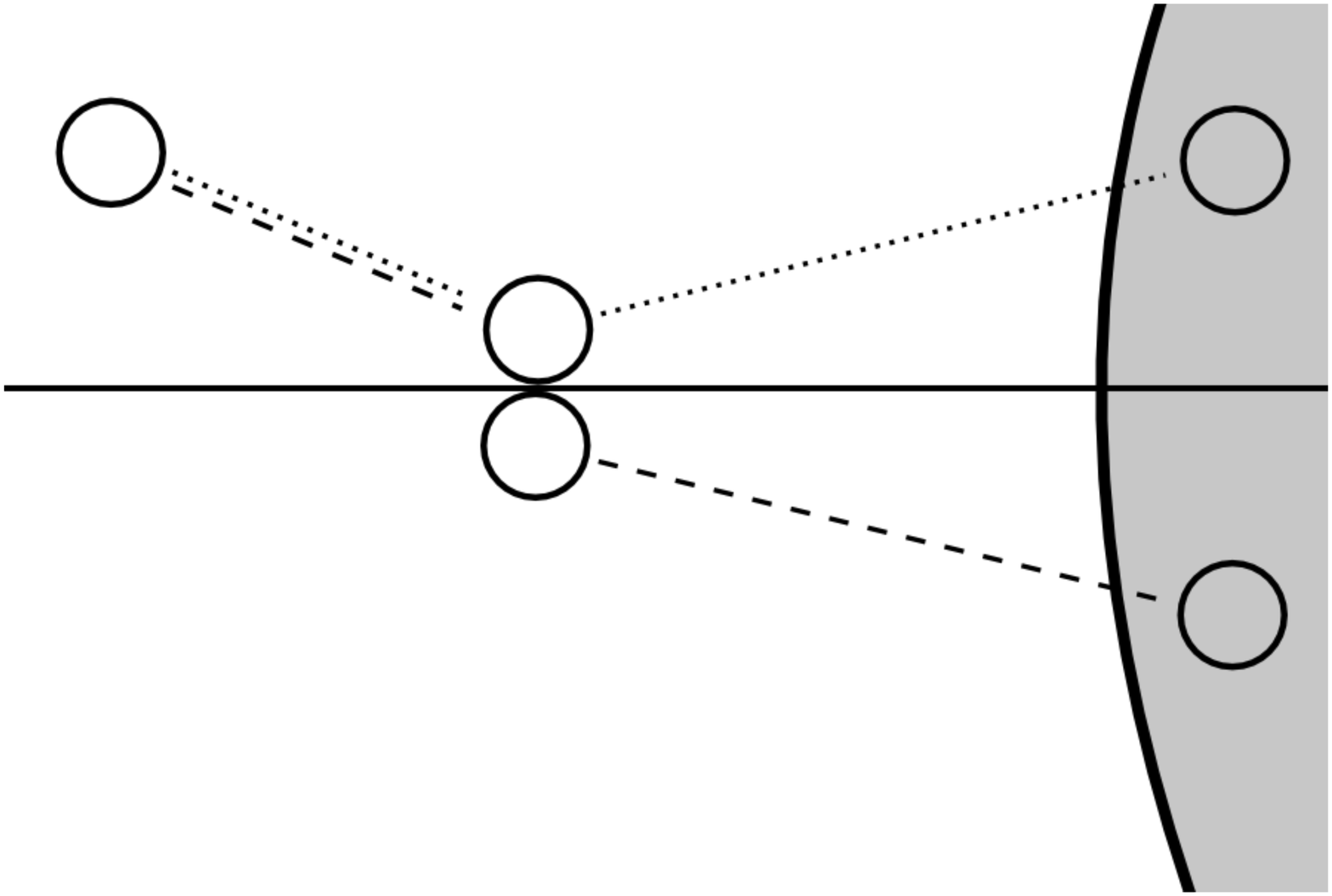}
\caption{}
 
\end{figure}
We proved that  for any $z \in C$, there exists a minimal gallery from $x$ to $z$ passing through $y$. This proves in particular that $y$ is unique and the proof is achieved.
 
\end{proof}

The following   lemma says that   the projections on the dials of building are orthogonal relatively to the building-wall structure.

\begin{lem} \label{lem stabiproj}
Let $D,D' \in \dialgp$  such that $\Ch{D \cap D'}\neq\emptyset$. If $x\in\Ch{D}$ then   $\proj{D'}{x}\in \Ch{D \cap D'}$.
\end{lem}

\begin{proof}
Clearly  $\proj{D'}{x} \subset D'$ so we   check that $\proj{D'}{x}\subset D $. Under the assumption $\Ch{D \cap D'}\neq\emptyset$  three cases are possible.  
First, if $D'\subset D$ then $\proj{D'}{x} \subset D' \subset D $. Then if  $D\subset D'$, then   $\proj{D'}{x}=x \subset D$ for any  $x\in \Ch{D}$. 

Now let $M$ and $M'$ be the building-walls that bound $D$ and $D'$. The last case is realized when $M\perp M'$. In this case consider a minimal gallery $x\sim\dots\sim  \proj{D'}{x}$.

If $\proj{D'}{x}\not\subset D $, then the preceding gallery crosses $M$. As a consequence, we can write that there exists a minimal gallery of the form 
\[x\sim \dots \sim x_i \sim x_{i+1}=r x_i \sim x_{i+2}\sim \dots \sim  \proj{D'}{x}\]
where $r\in \Gamma$ is rotation around $M$. Then with Lemma \ref{lem stabiortho} we obtain that

\[x\sim \dots \sim x_i \sim r^{-1} x_{i+2} \sim \dots \sim  r^{-1} \proj{D'}{x} \]
is a gallery between $x$ and $D'$ of length $\dc{x,\proj{D'}{x}}-1$. Which is a contradiction.

\end{proof}

Applying several times the projection maps on dials of building, we define projection maps on finite intersections of dials of building.

\begin{prop}  \label{prop defprojinter}  
 Let  $D_1,\dots, D_k\in \dialgp$  and $C=\Ch{D_1\cap \dots \cap D_k}$. Assume that  $C\neq \emptyset$. Then for any $x\in \ch$ there exists a unique chamber $\proj{C}{x}\in C$  \index{$\proj{C}{\cdot}$} such that \[\dc{x,\proj{C}{x}}= \di{x, C}.\]   Moreover, for any chamber  $y \in C$ there exists a minimal gallery from  $x$ to $y$ passing through   $\proj{C}{x}$. Finally $\proj{C}{x}=\mathrm{proj}_{D_k } \circ \dots \circ \mathrm{proj}_{D_1}(x)$.\end{prop} 
 
For  simplicity, in the following we will   use the notation   $\proj{D}{\cdot}$ \index{$\proj{D}{\cdot}$} instead of $\proj{C}{\cdot}$. Notice that it is not always possible to define a projection on a  convex set of chambers. For instance, if $\Sigma$ is a thick building there exist  pairs of adjacent chambers $x$ and $y$ with $\dc{x_0,x}=\dc{x_0,y}$.  

\begin{proof}
First, according to Lemma \ref{lem stabiproj}, 
 we can assume, up to a subfamily, that $x\not\subset D_i$ for each $i=1,\dots, k$. Now we set \liste{\item $C_1=\Ch{D_1}$ and $C_i=C_{i-1} \cap \Ch{D_i}$ for any $i=2,\dots, k$,
\item  $x_1=\proj{D_1}{x}$ and $x_i=\proj{D_i}{x_{i-1}}$ for any $i=2,\dots, k$.}

By induction on $i$ we prove the following property:
\begin{center}
\begin{minipage}[c]{12cm} 
\textit{  $x_i\in C_i$  and   is the unique chamber of $C_i$ such that $\dc{x,x_i}= \di{x, C_i}$. Moreover, for any chamber  $y \in C_i$ there exists a minimal gallery from  $x$ to $y$ passing through   $x_i$. }
\end{minipage}
\end{center}

If $i=1$ the property holds by Proposition \ref{prop deffrojdial}. Let $i>1$ and assume that the property holds at rank $i$. In particular, for $j=1,\dots, i$ one has  $x_i\in \Ch{D_j}$ and $\Ch{D_j}\cap \Ch{D_{i+1}}\neq\emptyset$. Therefore, with Lemma \ref{lem stabiproj}, $x_{i+1}\in \Ch{D_1}\cap \dots \cap \Ch{D_i}\cap \Ch{D_{i+1}}=C_{i+1}$.

By Proposition \ref{prop deffrojdial}, $x_{i+1}$ is the unique chamber in $C_{i+1}$ such that \[\dc{x_i,x_{i+1}}=\di{x_i,C_{i+1}}.\]
Moreover, by the same proposition, for $y\in  C_i$ if  the galleries   $x\sim \dots \sim x_i$  and $x_i\sim \dots \sim y$ are minimal, the gallery \[ x\sim \dots \sim x_i \sim \dots \sim y\]is minimal. As a consequence,  $x_{i+1}$ is the unique chamber in $C_{i+1}$ such that $\dc{x,x_{i+1}}=\di{x,C_{i+1}}.$

On the other hand, by Proposition \ref{prop deffrojdial}, for any $y \in C_{i+1}$ there exists a minimal gallery \[ x_i\sim \dots \sim x_{i+1} \sim \dots \sim y.\] Hence if the gallery $x\sim \dots \sim x_i$ is minimal, the gallery \[x\sim \dots \sim x_i\sim \dots \sim x_{i+1} \sim \dots \sim y\] is also minimal.
\end{proof}

\subsection{Shadows on $\borg$}
\label{subsecshadows}

The following notions are used in the rest of this  article to describe the topology and the metric on $\borg$. We recall that the boundary of $\Gamma$ is canonically identified with the boundary of $\Sigma$.

\begin{déf} \label{def cone}
 Let $x\in \Ch{\Sigma}$. We call \index{Cone of chambers} \emph{cone of chambers of base $x$} and we denote \index{$C_x$} $C_x \subset \Sigma$, the union of the set of chambers $y\in \ch$ such that  there exists a minimal gallery from $x_0$ to $y$ passing through $x$.  
\end{déf}

Cones of chambers are characterized by  projection maps and dials of building.

\begin{prop} \label{prop intersontcone}
Let $D_1, \dots, D_k \in \dialgp$ and $C=D_1\cap \dots \cap D_k$. Assume that  $C$ contains a chamber and that $x_0\not\subset D_i$ for  $i=1, \dots, k$. If we set $x=\proj{C}{x_0}$  then $C_x=C$.
\end{prop}

\begin{proof}
According to Definition \ref{def cone} and to  Proposition \ref{prop defprojinter}, $C\subset C_x$. Now let $y \in \Ch{C_x}$ and let $M_i$ be the building-wall that bounds $D_i$  for each $i=1, \dots, k$. If $x_0\sim \dots \sim x\sim \dots \sim y$ is  a minimal gallery, then the subgallery $x\sim \dots \sim y$ does not cross the building-wall $M_i$ for any $i=1, \dots, k$ and $y\subset C$.
\end{proof}
 
Reciprocally cones are intersections of dials of building.

\begin{prop} \label{prop conesontinter}
Let $x\in \ch$ and  let $D_1 ,\dots , D _k \in \dialgp$ be the  family of   dials of building such that for any $i=1, \dots, k$ \[x_0\not\subset D_i \text{ and }x \subset D_i.\]  Then $C_x =D_1 \cap \dots \cap D_k$.
  \end{prop}
\begin{proof} Let $C=D_1 \cap \dots \cap D_k$. According to Proposition \ref{prop intersontcone}, it is enough to prove that $\proj{C}{x_0}=x$. If we write $x'=\proj{C}{x_0}$, with Proposition \ref{prop intersontcone}, $C=C_{x'}$. Hence there exists a minimal gallery \[x_0\sim\dots\sim x'\sim\dots\sim x.\]

Now assume that $x'\neq x$, this means that there exists a     building-wall $M$ that separates $x$ and $x'$. Since the preceding gallery is minimal, the dial of building $D$ bounded by $M$ that contains $x$ does not contain $x'$ and $x_0$. Thus $D\in \{ D_1 ,\dots , D _k\}$ and $x'\not\subset C$ which is a contradiction. 
\end{proof}

 In particular, as cones of chambers are intersections of dials of building, it makes sense to consider projection maps on cones and  $\proj{C_x}{x_0}=x$. In the following fact we summarize what we can say about intersections of dials of  building.

\begin{fact} \label{fact interdial} Let $D_1,\dots, D_k\in \dialgp$  be such that $D_i\not\subset D_j$ for any $i\neq j$. Assume that $x_0\not\subset D_i$ for any $i=1, \dots, k$. Let $M_i$ be the  building-wall that bounds $D_i$  for any $i=1, \dots, k$,   and let   $C=D_1\cap \dots \cap D_k$. Then exactly one of these assertions holds. \liste{\item There exists $i,j$ such that $M_i\parallel M_j$ and $C=\emptyset$.\item  $M_i\cap M_j \neq \emptyset$  for any   $i, j=1, \dots, k$,  and there exists $i\neq j$ such that $M_i=M_j$. In this case $C$ is contained in $M_i$.\item $M_i \perp M_j$ for any   $i\neq j$. In this case $C$ is a cone.}
\end{fact}

This fact, up to a translation and up to a subfamily, describes how a finite family of dials intersects. The following  lemma    specifies the   case when the intersection is a cone.

\begin{lem} \label{lem projalong}
Let $D_1,\dots, D_k$ be   a family of distinct dials of building  bounded by the building-walls $M_1 ,\dots , M_k$. Let $C=D_1\cap \dots \cap D_k$.  Assume that $x_0\not\subset D_i$ and that $M_i \perp M_j$  for any $i,j=1,\dots,k$ with $i\neq j$. Then any $M_i$ is along  $\proj{C}{x_0}$. 
\end{lem}

\begin{proof} 
First if $k=1$ the property is clearly satisfied. According to Lemma \ref{lem stabiproj}, $\proj{D_1}{x_0} \notin \Ch{D_2} \cup \dots \cup \Ch{D_k}$. Applying this lemma $k-2$ times we obtain that \[\mathrm{proj}_{D_{k-1}} \circ \dots \circ \proj{D_1}{x_0}\notin \Ch{D_k}.\] Hence $\proj{C}{x_0}=\proj{D_k}{\proj{D_{k-1}}\circ \dots \circ \proj{D_1}{x_0}}$ is along $M_k$. Finally, we apply the same argument to  $M_1,\dots, M_{k-1}$ and the proof is finished. 
\end{proof}

Finally we obtain the following characterization of cones.

\begin{prop} \label{prop projalong}
 Let $x\in \ch$ and  let   $D_1 ,\dots , D_k$ be  the family of dials of building bounded by $M_1 ,\dots , M_k$ such that for any $i=1, \dots, k$ \[x_0\not\subset D_i, \ x \subset D_i, \text{ and } \ M_i\text{ is along }x.\] Then \numeroti{ \item   $C_x=D_1\cap\dots \cap D_k$,
  \item there exist  $g\in \Gamma$ and $M^0_1, \dots, M^0_k \in \bw$  with \liste{
\item for any $i\neq j$: $M^0_i\perp M^0_j$,
\item for any $i$: $M^0_i$ is along $x_0$,
} 
such that $C_x=g(C)$ where $C=D_0(M'_1) \cap  \dots \cap D_0(M'_k)$.}
  
\end{prop}
\begin{proof}
Let $D'_1,\dots, D'_\ell$ be the family of   dials of building such that  $x_0\not\subset D_i $ and $x \subset D_i$ for any $i=1, \dots, \ell$.  Then  \[\{D_1 ,\dots , D_k\} \subset \{D'_1,\dots, D'_\ell\}.\]
According to Proposition \ref{prop conesontinter}, $C_x= D_1'\cap\dots \cap D_\ell'$, thus $C_x\subset D_1\cap\dots \cap D_k$.
 Let $M'_i$ be the wall that bounds $D'_i$ for any $i=1, \dots, \ell$. Up to a subfamily, we can assume that $C_x =D_1'\cap\dots \cap D_\ell'$  with  $M'_i\perp M'_j$ for any $i\neq j$. Then, as $\proj{C_x}{x_0}=x$, by  Lemma \ref{lem projalong}, any building-wall $M'_i$ is along  $x$.   Finally, we get \[\{D'_1,\dots, D'_\ell\}\subset\{D_1 ,\dots ,D _k\}\]
and $D_1\cap\dots \cap D_k\subset  C_x $.

 The second part is immediate with  $g\in \Gamma$ such that $gx_0=x$.

\end{proof}

 The following lemma is used to prove that boundaries of cones are of non-empty interior.

\begin{lem} \label{lem existwallparallel}
Let $M_1,\dots, M_k$ be a collection of building-walls such that any $M_i$ admits a parallel building-wall. Assume that $M_i\perp M_j$ for any $i\neq j$. Then there exists $M\in \bw$ such that $M\parallel M_i$ for any $i$.
\end{lem}

\begin{proof}
We prove the proposition by induction on $k$. For $k=1$ there is nothing to prove.  For $k\geq 2$, pick  a collection  of building-walls $M_1,\dots, M_k$  satisfying the hypothesis of the lemma. Assume that there exists $M$ a building-wall such that $M\parallel M_1,\dots, M\parallel M_k$. Let $M_{k+1}\in \bw$ be such that $M_{k+1}\perp M_1,\dots, M_{k+1}\perp M_k$.

If $M$ is parallel to $M_{k+1}$ there is nothing more to say. Now we assume  $M\perp M_{k+1}$ and we pick a wall $M'$ parallel to $M_{k+1}$. If $M'$ is parallel to  $M_1,\dots, M_k$ there is nothing more to say. Now we assume that there exists $1\leq i\leq k$ such that, up to a reordering
\[M'\perp M_1,\dots, M'\perp M_i,M'\parallel M_{i+1},\dots,M'\parallel M_{k},M'\parallel M_{k+1}.\]
First we consider the case $M'\perp M$. With \liste{\item $M\perp M_{k+1}$, $M_{k+1}\perp M_1$, $M_1\perp M'$, \item and $M'\parallel M_{k+1}$, $M\parallel M_1$,} we obtain that the building-walls $M', M, M_{k+1}$, and $M_1$ form a right-angled rectangle. Which is a contradiction with the hyperbolicity of $\Sigma$.

Secondly we consider the case $M'\parallel M$. Let $r'\in \Gamma$ be a rotation around $M'$. Then $r'(M)$ is such that \[r'(M)\parallel M_1,\dots ,r'(M)\parallel M_i.\]
Indeed, as $M\parallel M_j$, it comes that $r'(M)\parallel r'(M_j)$ and, as  $M'\perp M_j$, by   Lemma \ref{lem stabiortho}, $r' (M_j)=M_j$  for $1\leq j\leq i$. Thus $r'(M)\parallel M_j$.

Moreover\[r'(M)\parallel M_{i+1},\dots ,r'(M)\parallel M_{k+1}.\]
Indeed, as $M_{i+1}\cap \dots \cap  M_{k+1}\neq \emptyset$ and  $M'\parallel   M_j$ for  $i+1\leq j\leq k+1$,  the building-walls $M_{i+1}, \dots,  M_{k+1}$ are entirely contained in  the same  connected component  of $\Sigma\backslash M'$. Let $C$ be this connected component. Since $M\parallel M'$ and $M\cap M_{k+1}\neq\emptyset$ it comes that $M$ is also contained in $C$. Thus  $r'(M)$ is not contained in $C$ and $r'(M)\parallel   M_j$ for  $i+1\leq j\leq k+1$.
 \end{proof}

 \begin{prop} \label{prop shadowintnonvide}
Let $x\in \ch$ and $C_x$ be the cone based at $x$. Then $\partial C_x$ is of non-empty interior in $\borg$.  
\end{prop}

\begin{proof}
 By Fact \ref{fact interdial} and Proposition \ref{prop projalong}  we can write \[\partial C_x=\partial D_1\cap\cdots \cap \partial D_k,\]
where  $D_1,\dots, D_k$ is a collection of dials of building bounded by the building-walls $M_1, \dots, M_k$ with   $M_i\perp M_j$ for any $i\neq j$. 

Up to a subfamily, we can assume that for every $i=1,\dots, k$ there exists  $M_i'\in \bw$ such that $M_i\parallel M_i'$. Indeed if $M$ admits no parallel building wall then $\partial M=\borg$ (see beginning of Section \ref{subsec BofLFRAHB}).

By  rotations around $M_1, \dots, M_k$,  all the connected components of $\Sigma\backslash (M_1\cup\dots \cup M_k)$ are isomorphic. Hence, thanks to Lemma \ref{lem existwallparallel},    there exists  $M\in \bw$ such that \[M \parallel M_i \text{ for any } i=1,\dots, k \text{ and } M \subset D_1\cap \dots \cap D_k.\]
 
 In particular, there exists $D\in \dialgp$  bounded by $M$ such that $D\subset D_1\cap \dots \cap D_k$. As $\partial D$ is of non-empty interior, we obtain that $\partial C_x$ is of non-empty interior.
\end{proof}

In the rest of this  article, we   use boundaries of cones as a base for the topology of $\borg$ and to construct approximations. 

\begin{défnot} \label{defshadow}
 Let $x\in \ch$ and  $C_x$ be the corresponding cone of chambers. We call  \index{Shadow} \emph{shadow of $x$} the boundary of $C_x$ in $\borg$ and we write  $v_x= \partial C_x$\index{$v_x$}.
 \end{défnot} 
 
\subsection{Combinatorial metric on $\borg$} 
\label{subsecmet}

Until now we have been considering on $\borg$ the visual metric coming from the geometric action of $\Gamma$ on $\Sigma$. Now we use  minimal galleries to describe a combinatorial metric on $\borg$ that will be more convenient to use in the sequel. We extend  the notion of minimal gallery to infinite galleries.
 \begin{déf} \label{defgalinfini} An infinite gallery $x_0\sim x_1 \sim  \cdots$ (resp. a bi-infinite gallery $\dots \sim x_{-1} \sim x_0\sim x_1 \sim  \cdots$) is  \index{Infinite (bi-infinite) minimal gallery} \emph{minimal} if for any $k\in \N$ (resp. $k\in \Z)$ and $\ell \in \N$  the gallery $x_{k}\sim\dots \sim x_{k+\ell}$ is minimal. 
 \end{déf}
  Let $\dual$ denote the \emph{dual graph} \index{Dual graph}\index{$\dual$} of $\Sigma$. This graph is defined by:
\liste{
\item The set of vertices $\dual^{(0)}$  is given by $\ch$ the set of chambers in $\Sigma$. If  $v\in \dual^{(0)}$ then $c_v$ denotes the corresponding chamber in $\ch$.
\item There exists an edge between two vertices $v_1$ and $v_2$ if and only if $c_{v_1}$ is adjacent to $c_{v_2}$ in $\Sigma$.
\item Each edge is isometric to the segment $[0,1]$.
}  

Naturally, $\dual$ is a proper geodesic and hyperbolic space. It is quasi-isometric to $\Sigma$ and the action of $\Gamma$ on $\dual$ is geometric.  Therefore we identify \[\partial \dual\simeq \borg.\]
 
\begin{exmp}We recall that the group $\Gamma$ is given by the following presentation \[\Gamma = \left\langle s_i\in S \vert  s_i^{q_i}=1,  s_is_j=s_js_i \text{ if } v_i \sim v_j  \right\rangle. \]If  $q_i=2$ or $3$ for any $i=1,\dots,n$  then $\dual$ is identified with  $\cay$ the Cayley graph of $\Gamma$ with respect to the recalled generating set. Otherwise, if we   consider a generator $s\in S$ of $\Gamma$ of order $q\geq4$ then in $\dual$ the full sub-graph generated by the vertices associated with $e,s,\dots, s^{q-1}$ is a complete graph. In $\cay$ the full sub-graph generated by the vertices associated with $e,s,\dots, s^{q-1}$ is a cyclic graph of length $q$.  Nevertheless  $\dual$ and $\cay$ are always quasi-isometric.
\end{exmp}

 With  Definition \ref{defgalinfini}, infinite minimal galleries are identified with geodesic rays in $\dual$ starting from a vertex. Therefore  we can identify  $\borg$ with the set of  equivalence classes of infinite galleries starting at $x_0$ where two such galleries $x_0\sim x_1 \sim  \cdots$ and  $y_0=x_0\sim y_1 \sim  \cdots$ are equivalent if and only if  there exists $K>0$ such that $\dc{x_i,y_i}<K$ for all $i \in \N$.

\begin{exmp} 
Here  we consider only minimal galleries. We write $\mathcal{R}$ the equivalence relation on the infinite galleries starting at $x_0$ defined above.  Let $x\in \ch$ with $\dc{x_0,x}=k\geq 1$. Then we can describe   the shadow $v_x$ as follows 
\[ {v_x}\simeq \{ x_0\sim x_1 \sim  \cdots \sim x_i\sim \dots : x_i\in\ch \text{ and } x_k=x \}/\mathcal{R}. \]
 Likewise, if $\partial P$ is a parabolic limit set associated with the residue $g\Sigma_I$. Let  $x:=\proj{g\Sigma_I}{x_0}$ and assume that  $\dc{x_0,x}=k\geq 1$. Then we can describe  $\partial P$  as follows \[\partial P\simeq \{ x_0\sim x_1 \sim  \cdots   :    x_k=x  
  \text{ and } x_{k+i}\sim_{s_i} x_{k+i+1} \text{ with } s_i \in I \text{ for any } i\geq 0 \}/\mathcal{R}. \]
\end{exmp}

Now we   use the following notation.

\begin{nota}
  If $x_0\sim x_1 \sim \cdots$ is a minimal infinite gallery that goes asymptotically to $\xi \in \borg$, then we write \[\xi=[x_0\sim x_1 \sim \cdots].\] 
\end{nota} 

 \begin{déf}Let  $\xi, \xi'$ be two distinct points in $\borg$, let \index{$\{ \cdot \vert \cdot \}_{x_0}$} $\{\xi \vert \xi' \}_{x_0}$ denote the largest integer $\ell$ such that there exist  two infinite minimal galleries representing  $\xi$ and $\xi'$  \[ \xi =[x_0\sim x_1 \sim  \dots \sim x_i \sim  \cdots] \text{ and }\xi' =[x_0\sim x'_1 \sim  \dots \sim x_i' \sim  \cdots]\] with \[x_i=x_i' \text{ for }i\leq \ell \text{ and }x_{\ell+1} \neq x_{\ell+1}'.\] 
\end{déf}

In terms of shadows,  $\{\xi \vert \xi' \}_{x_0}$ is the largest integer such that there exists   a shadow  ${v_x}$, with $\dc{x_0,x}= \{\xi \vert \xi' \}_{x_0}$, that contains both  $\xi$ and $\xi'$.   The following proposition gives a characterization of this quantity in terms of building-walls. We recall that $D_0(M)$ designates the dial of building bounded by $M$ and containing $x_0$.
 
\begin{prop} \label{prop carcprodgrompremiere}   Let $\xi, \xi'$ be  two distinct points in $\borg$. Then 
\[\{\xi \vert \xi' \}_{x_0} =  \#\{M \in\bw : \text{ there exists }\alpha\neq 0 \text{ s.t. }\{\xi, \xi'\} \subset \partial D_\alpha(M)\}.\]  \end{prop}  

\begin{proof}
Let $M_1,\dots, M_k$ be the set of building-walls such that   there exists $\alpha\neq 0$  with $\{\xi, \xi'\} \subset \partial D_\alpha(M)$.  Let  $\ell = \{\xi \vert \xi' \}_{x_0}$. We prove that $k=\ell$.

For $i=1,\dots k$, let $D_i$ be a dial of building bounding $M_i$ such that   $\{\xi, \xi'\} \subset \partial D_i$ and $x_0\not\subset D_i$. We set $C=D_1\cap \dots \cap D_k$. Since the building-walls are distinct and $\partial C\neq \emptyset$ it follows from Fact \ref{fact interdial} that  $C$ is a cone. Let $x=\proj{C}{x_0}$. As  $\{\xi, \xi'\} \subset \partial C$,  there exists an infinite minimal gallery starting from $x_0$ going asymptotically to $\xi$ (resp. $\xi'$) passing through $x$. Finally, we obtain $\ell \geq \dc{x_0,x}\geq k$.

Now consider $x_0\sim x_1 \sim  \dots \sim x_i \sim  \cdots$ (resp. $x_0\sim x'_1 \sim  \dots \sim x_i' \sim  \cdots$) a minimal infinite gallery representing $\xi$ (resp. $\xi'$) in $\borg$. Assume that \[x_i=x_i' \text{ for }i\leq \ell \text{ and }x_{\ell+1} \neq x_{\ell+1}'.\]

For any $i=1,\dots, \ell$ let $D'_i$ be the dial of building such that $x_{i-1} \not\subset D'_i$ and $x_i\subset D'_i$. By minimality of the galleries, we get that $\{\xi, \xi'\} \subset \partial D'_i$ for any index $i$. Therefore $\ell\leq k$ and the proof is finished.
 
 \end{proof}

In the following, we prove that   $\{\cdot\vert \cdot \}_{x_0}$ coincides with a Gromov product in $\borg$ and thus controls   a visual  metric on $\borg$.
 
\begin{prop} \label{prop prodgromov}

Let $\xi,\xi'$ be two distinct points in $\borg$. Then there exists a bi-infinite minimal gallery between $\xi$ and $\xi'$ that lies at a distance smaller than  $\{\xi \vert \xi' \}_{x_0}+1$ of $x_0$.
\end{prop}
 
\begin{proof}
Let  $\ell = \{\xi \vert \xi' \}_{x_0}$ and assume that  $\xi = [x_0\sim x_1 \sim  \dots \sim x_i \sim  \cdots]$ and  $\xi' =[x_0\sim x'_1 \sim  \dots \sim x_i' \sim  \cdots]$ with \[x_i=x_i' \text{ for }i\leq \ell \text{ and }x_{\ell+1} \neq x_{\ell+1}'.\]We  consider two cases. Either $x_{\ell+1}$ is adjacent to $ x_{\ell+1}'$, or $x_{\ell+1}$ is not adjacent to $ x_{\ell+1}'$. In the first case, Proposition \ref{prop distcombigrap} implies that  the bi-infinite gallery   \[\cdots \sim x_{\ell+2}    \sim x_{\ell+1}   \sim x_{\ell+1}' \sim x_{\ell+2}'\sim  \cdots \] only crosses once the building-walls that separate   $\xi$ and $\xi'$.  Hence it is minimal.
 
In the second case, we apply the same reasoning to the bi-infinite gallery  \[\cdots \sim x_{\ell+2}      \sim x_{\ell+1} \sim x_\ell   \sim x_{\ell+1}' \sim x_{\ell+2}'\sim  \cdots. \]
Finally $\{\xi \vert \xi' \}_{x_0}$ or $\{\xi \vert \xi' \}_{x_0}+1$ is the distance between $x_0$ and a bi-infinite minimal gallery between $\xi$ and $\xi'$.
\end{proof}

 \begin{nota}
Let $d(\cdot,\cdot)$ \index{$d(\cdot,\cdot)$} be the self-similar metric on $\borg$ coming from the geometric action of $\Gamma$ on $\dual$(see Definition \ref{defselfsimbord}).
\end{nota}

As a consequence of  Proposition \ref{prop prodgromov}, a   general result about hyperbolic spaces due to M. Gromov states that there exist two constants $A\geq 1$ and $\alpha >0$ such that for any $\xi,\xi' \in \borg$:\[A^{-1}  e^{-\alpha \{\xi \vert \xi' \}_{x_0}} \leq d(\xi, \xi') \leq A \  e^{-\alpha \{\xi \vert \xi' \}_{x_0}}.\]
In the sequel we   also write \[d(\xi, \xi') \asymp  e^{-\alpha \{\xi \vert \xi' \}_{x_0}}\index{$\asymp$}.\]This means that, $d(\xi, \xi')$ is, up to a multiplicative constant, equal to $e^{-\alpha \{\xi \vert \xi' \}_{x_0}}$. An application  of this description of the visual metric on $\borg$ is the following proposition.

\begin{prop}\label{proptaillepara} For every $\epsilon>0$, there exists only a finite set of parabolic limit sets of diameter larger than $\epsilon$.
\end{prop}

\begin{proof}
Let $\partial P$ be a parabolic limit set. Let  $g'\Sigma_I$ be a residue in $\Sigma$ such that $\partial P \simeq \partial (g'\Sigma_I)$. According to  Proposition \ref{prop deffrojdial}, there exists a unique chamber $x \subset g' \Sigma_I$ such that for every chamber $y \subset g'\Sigma_I$ there exists a minimal gallery from $x_0$ to $y$ passing through $x$. Let $g\in \Gamma$ such that $x=g x_0$.  Then the diameter of $\partial P$ is controlled by $e^{-\alpha \vert g \vert}$ with $\vert g \vert = \dc{x_0,g x_0} $. As there exists  only a finite number of $g \in \Gamma$ such that $\vert g \vert$ is smaller than a fixed  constant, the proposition is proved.

\end{proof}

\subsection{Approximation of $\borg$ with shadows}
\label{subsec approxshadow}

The following proposition says that shadows are almost balls. This will allow us to construct approximations using   shadows.
  
\begin{prop} \label{prop ballsshadows} There exists $\lambda > 1$ such that  for any $x\in \Ch{\Sigma}$ with $\dc{x_0,x}=k$  there exists $z\in v_x$ with 
\[B(z,\lambda^{-1} e^{-\alpha k}) \subset v_x \subset B(z,\lambda e^{-\alpha k}).\]
\end{prop}

\begin{proof}
To prove the right-hand side inclusion it is enough to notice that $\dia{v_x} \leq A e^{-\alpha k}$ where $A$ and $\alpha$ are the visual constants.
Let $C_x$ be the cone based at $x$. Let $C=D_0(M_1) \cap  \dots \cap D_0(M_k)$ and $g\in \Gamma$ such that $g(C)=C_x$ (see Proposition \ref{prop projalong}). Now we recall that  $g^{-1}$ is a bi-Lipschitz homeomorphism. Restricted to  $v_x$, it  rescales  the metric  by a factor $e^{\alpha k}$. According to Proposition \ref{prop shadowintnonvide},  there exist $r >0$ and $z \in \partial C$ such that  $B(z,r) \subset\partial C$. As there is  only a finite number of possible $C$, the proof is achieved. 
\end{proof}

Let $x\in \ch$ and $v_x$ be the associated shadow as in Definition \ref{defshadow}. Thanks to Proposition \ref{prop prodgromov}, if   $\dc{x_0,x}=k$ then  $\dia{v_x}\asymp e^{-\alpha k}$. We use this property to  construct an approximation of $\borg$ consisting of shadows. 

For  an integer $k\geq0$ we set \[S_k= \{ x \in \ch: \dc{x_0,x}=k  \}.\] The set $\{ {v_x} : x\in S_k\}$ is a finite covering of    $\borg$.
 Now let $S'_k$ be a subset of $S_k$  such that   $ \{v_x : x\in S_k'\}$  defines a  minimal covering of $\borg$. This means that for every  $ x \in S_k'$ there exists $z \in v_x$ such that $z \notin v_y$  for any $y \in S_k'\backslash \{x\}$. Finally we set \[G_k=\{v_x : x\in S_k'\}.\] In the following, we prove  that the sequence  $\{G_k\}_{k\geq0}$ defines an approximation of $\borg$.

\begin{prop}\label{propapprox}
For $k\geq0$, let $S_k'$ be the set of chambers previously defined and $G_k$ be the minimal covering of $\borg$ associated with $S'_k$. There   exists $\kappa >1$ such that for any  $x\in S_k'$, there exists  $\xi_x\in v_x$ such that:

\liste{ \item $\forall x \in S_k'$: $B( \xi_x,\kappa^{-1} e^{-\alpha k}) \subset v_x \subset B(\xi_x,\kappa e^{-\alpha k})$,
\item $\forall x, y \in S_k'$ with $x\neq y$: $B(\xi_x,\kappa^{-1} e^{- \alpha k})\cap B(\xi_y,\kappa^{-1} e^{- \alpha k})= \emptyset$.}
\end{prop}

This property is enough to construct an approximation of $\borg$. Indeed the visual constant $\alpha$ can be chosen such that $1/2\leq e^{-\alpha}<1$. In this case we can extract from $\{G_k\}_{k\geq0}$ a subsequence that is an approximation of $\borg$ as defined in  Subsection  \ref{subsecdefapproxcasgen}.

\begin{proof}[Proof of Proposition \ref{propapprox}]
Let  $x  \in S_k'$, and let   $\xi_x \in v_x$.  With $\dia{v_x} \asymp e^{-\alpha k}$, there exists $\kappa >1$ such that for all  $x \in S_k'$: $v_x \subset B(\xi_x,\kappa e^{-\alpha k})$.

  We   recall that the hyperbolicity provides a constant  $N\geq 1$, depending only on the hyperbolicity parameter, such that for $x,x'\in \ch$ with $\dc{x_0,x}=\dc{x_0,x'}$ if $\dc{x,x'} \geq N$ then $v_x\cap v_{x'}=\emptyset$.  

For any $x\in S'_k$, we pick $z_x \in v_x$ such that $z_x \notin v_y$  for any $y \in S_k'\backslash \{x\}$. Let $x,y\in S'_k$, $x\neq y$ and let $c\in \Ch{\Sigma}$ be such that $\dc{x_0,c}= \{z_x\vert z_y\}_{x_0}$ and $\{z_x,z_y\}\subset v_c$. In this setting we can write that $z_x$ and $z_y$ are represented by infinite minimal galleries of the form:
\liste{
\item $z_x =[x_0\sim x_1 \sim  \dots \sim  x_i  \sim \dots \sim x_k \sim x_{k+1} \sim  \cdots]$
\item $z_y =[x_0\sim y_1 \sim \dots   \sim  y_i  \sim \dots   \sim y_k \sim y_{k+1} \sim  \cdots]$} 
with \liste{ \item $x_i= c$ and $y_i=c$ for one $i \in \{1, \dots, k-1\}$,
			\item $x_k=x$ and $y_k=y$.}			

			Now we consider two cases. First, we assume that  $x_{i+1}$ is adjacent to $ y_{i+1}$. In this case, Proposition \ref{prop distcombigrap} implies that  the bi-infinite gallery  \[ \dots  \sim x_{k+1} \sim x_k  \sim \dots \sim x_{i+1}   \sim y_{i+1} \sim \dots     \sim y_k  \sim y_{k+1}  \sim   \cdots,\]    only crosses once the building-walls that separate $\xi$ and $\xi'$. Hence it is minimal.
					
Then we assume that  $x_{i+1}$ is not adjacent to $ y_{i+1}$. In this case  we apply the same reasoning to the bi-infinite gallery  \[ \dots \sim x_{k+1} \sim x_k  \sim \dots \sim x_{i+1} \sim c \sim y_{i+1} \sim \dots     \sim y_k  \sim  y_{k+1}  \sim  \cdots. \]
  
  Summarizing we obtain that  one of the following galleries:
\liste{\item $ \dots  \sim x_{k+1} \sim x_k  \sim \dots \sim x_{i+1}   \sim y_{i+1} \sim \dots     \sim y_k  \sim y_{k+1}  \sim   \cdots$
\item $ \dots \sim x_{k+1} \sim x_k  \sim \dots \sim x_{i+1} \sim c \sim y_{i+1} \sim \dots     \sim y_k  \sim  y_{k+1}  \sim  \cdots$}
is a bi-infinite minimal gallery from $z_x$ to $z_y$.

 In particular $\dc{x_{k+N},y_{k+N}}> N$ and  the corresponding shadows do not intersect: \[v_{x_{k+N}}\cap v_{y_{k+N}} = \emptyset.\]
Now according to Proposition \ref{prop ballsshadows} there exist  $\xi_x\in v_{x_{k+N}}$ and $\xi_y\in v_{y_{k+N}}$ such that
 \[B(\xi_x,\lambda^{-1} e^{-\alpha (k+N)})   \subset v_{x_{k+N}}\text{ and }B(\xi_y,\lambda^{-1} e^{-\alpha (k+N)})   \subset v_{y_{k+N}}.\]
 With $v_{x_{k+N}}\cap v_{y_{k+N}} = \emptyset$, $v_{x_{k+N}} \subset v_{x}$,  and  $v_{y_{k+N}}\subset v_{y}$ we obtain the desired property.

\end{proof}
 
\section{Modulus in the boundary of a building  and in the boundary of an apartment} 
\label{secmoduleapartmoduleimmeuble}
 
The boundary of an apartment is, in a well chosen case, easier  to understand than the boundary of the building. This is why we want to compare the modulus in the boundary of the building with some modulus in the boundary of an apartment.

In this section, we start by defining a convenient approximations on $\borg$ and on the boundaries of the apartments using shadows and retraction maps. Afterwards, we introduce the weighted modulus on the boundary of an  apartment.  Then  we prove Theorem \ref{theocontrolimmappart}. This theorem is, after Theorem \ref{theocourbedanspara}, the second  major step in proving the main theorem (Theorem \ref{theoprincip}). It states that     weighted modulus are  comparable to the modulus in $\borg$.  Finally, using the ideas in Subsection \ref{subsecdimconf}, we reveal a connection between the conformal dimension of $\borg$ and a critical exponent computed in the boundary of an apartment.

We use the notation and assumptions from  Sections \ref{sec LFRAHB}, \ref{seccurves inPLS} and \ref{sec bord topo met}. In particular $p\geq1$ is a fixed constant. We fix $\Gamma$ the graph product associated with the pair $(\mathcal{G}, \{\Z/q_i\Z\}_{i=1,\dots,n})$. The   self-similar metric $d(\cdot,\cdot)$ on $\borg$ is defined as in Subsection \ref{subsecmet}. The visual exponent of $d(\cdot,\cdot)$ is $\alpha$. As in Section \ref{secmoduhyper}, $d_0$ denotes a small constant compared with $\dia{\borg}$ and with the constant of approximate self-similarity. Then $\F_0$ is the set of curves of diameter larger than $d_0$.

\subsection{Notations and conventions in $\partial A$ and in $\borg$}

In the  rest of this  article we fix an apartment $A$ containing the base chamber $x_0$. We shall connect the geometry and   the modulus in $\partial A$ and in $\borg$. Naturally we will use in $\partial A$ and in $\borg$ the same concepts. Here we summarize some of the notation used in the following to avoid confusion. First we  write 
 \[\apo=\{B\in \ap : x_0\subset B\}\index{$\apo$}.\]\index{$\apo$} Let $\pi$ denote the retraction $\pi_{A,x_0} :\Sigma \longrightarrow A$. We also denote by $\pi$ the extension of the retraction to the boundary. The notation $d(\cdot,\cdot)$ and $\alpha$ are also used to describe the metric on $\partial B$ for any $B\in \apo$.

An apartment is a thin building, so we can use in $\partial A$ the tools presented in Subsections   \ref{subsecshadows} and \ref{subsecmet}. First, we  define on $\partial A$ a combinatorial  self-similar metric as in Subsection \ref{subsecmet}. Since $x_0\subset A$,  for $\xi, \xi' \in \partial A$, the quantity $ \{\xi\vert \xi'\}_{x_0}$ is the same whether we compute it in $A$ or in $\Sigma$. Hence, if we choose the same  visual exponents for  the visual metric in $\borg$ and the visual metric in $\partial A$, then the metrics coincide up to a multiplicative constant. In the rest, $d(\cdot,\cdot)$  designate the metric on both $\partial A$ and  $\borg$. Likewise,  $\alpha$ and $A$ designate the visual constants or both $\partial A$ and  $\borg$.    

 Finally,   it makes sense to talk about cones of chambers in $A$ and shadows in $\partial A$. The results of \ref{subsecshadows} also hold in $\partial A$.

\begin{nota} 
\text{  }
\liste{ 
\item for $\xi\in \partial A$ and $r>0$ we designate by $B(\xi, r)\subset \borg$  the open ball of $\borg$ of radius  $r$ and  center  $\xi$, \index{$B^A(\xi, r)$}

\item for $x \in \Ch{A}$ we write $C_x$  \index{$C_x$} for the cone of chambers based on $x$ in $\Sigma$,
\item for $x \in \Ch{A}$ we write $v_x$ (resp. $w_x$) \index{$v_x$} \index{$w_x$}for the shadow of $x$ in $\borg$ (resp. $\partial A$). 
} 
\end{nota} 
 Usually we will use the following conventions. 

\liste{\item  $v$ (resp. $w$) designates an open subset of $\borg$ (resp. of $\partial A$),

\item $\partial P$ \index{$\partial P$} (resp. $\partial Q$ \index{$\partial Q$})  designates a parabolic limit set in  $\borg$  (resp. in $\partial A$).
 }

\subsection{Choice of approximations}
\label{subsec choiceapprox}
 
 The following lemma  says that shadows have a nice behavior under retraction maps. 
 \begin{lem} \label{lem interapartshadow}
 Let $A \in \apo$ and let $x\in \ch$ and $ v_x$ be the associated shadow in $\partial \Gamma$ as defined in Definition \ref{defshadow}. Then \liste{\item  either $x \notin \Ch{A}$ and  $\Int{v_x}\cap \partial A = \emptyset$, \item or $x\in \Ch{A}$ and $v_x\cap \partial A$ is a shadow in $\partial A$.} In the second case $v_x\cap \partial A= \pi(v_x)$.
 \end{lem}\begin{proof}
  Let $C_x$ be the cone based on $x$. If $\Int{v_x}\cap \partial A \neq \emptyset$ then there exists  a chamber $c$ in $A \cap C_x$. By convexity, a minimal  gallery from $x_0$ to $c$ that passes through $x$ is included in $A$ and $x\subset A$. Therefore $v_x\cap \partial A$ is the shadow in $\partial A$ associated with $x$.  
 \end{proof}

We fix $\{G^A_k\}_{k\geq 0}$  an approximation of $\partial A$ based on shadows as constructed in Subsection \ref{subsec approxshadow}.  

\begin{nota} 
  For $k\geq 0$ we set \[G_k:=\{v_y\subset \borg :\pi(v_y) \in G_k^A\}.\]
\end{nota}

We recall that we  chose the same visual exponents for the metrics in $\borg$ and $\partial A$. As a consequence of Lemma \ref{lem interapartshadow}  we get  the following fact.

\begin{fact} \label{fact approximapart}
There exists $\kappa>1$ such that $\{G^A_k\}_{k\geq 0}$  and $\{G_k\}_{k\geq 0}$ are $\kappa$-approximations. Moreover, for any $w \in G^A_k$ there exists a unique $\widetilde{w}\in G_k$ such that  $\Int{\widetilde{w}}\cap \partial A \neq \emptyset$ and  $\pi(\widetilde{w}) =  w$.
\end{fact}

Hereafter, $\{G_k\}_{k\geq 0}$ designates the approximation of $\borg$ obtained from $\{G_k^A\}_{k\geq 0}$ thanks to the preceding fact. This approximation of $\borg$ is canonically associated with  $\{G_k^A \}_{k\geq 0}$ in the following sense: from $\{G_k\}_{k\geq 0}$ we can equip   any $B\in \apo$ with an approximation isometric to $\{G_k^A \}_{k\geq 0}$.  Indeed if   $B \in \apo$, for $k\geq 0$ we set
\[G^B_k:=\{w = \partial  B \cap v : v \in G_k\} .\]
Now let $B \in \apo$ and $f  : B\longrightarrow A$ be the type preserving isometry that fixes $x_0$. The map $f$ is realized by the restriction to $B$ of the retraction $\pi$ and we get the following fact.

\begin{fact} \label{factchangeapartapprox}  $G^A_k= \{f(v)\}_{v\in G^B_k}$. 
\end{fact}

Now that an  approximation $\{G_k \}_{k\geq0}$ is fixed the results we will obtain on the combinatorial modulus in $\borg$ will be valid, up to multiplicative constants, for any approximation thanks to Proposition \ref{propdouble}.

\subsection{Weighted modulus in $\partial A$}  
 
\label{subsec weighmod}
 
On scale $k\geq 0$, to compare the modulus in the building with the modulus in the apartment we need  to compare the cardinality of $G_k$ with the cardinality of $G_k^A$. If the building is thick these quantities differ by an exponential factor in $k$. This is the reason we attach a weight to the elements of    $G_k^A$.
 
\begin{déf} \label{defpoid}
Let $w\in   G_k^A$, we set $q(w)=\#  \{v\in G_k : \pi(v)=w\}$. 

\end{déf}

 Let $k\geq 0$ and let  $\F^A$ be a set of curves contained in $\partial A$. As in Subsection \ref{subsecdefapproxcasgen}, a positive function $\rho: G_k^A \longrightarrow [0,+\infty)$ is said to be $\F^A$-admissible if for any $\gamma \in \F^A$ 
 
 \[  \sum\limits_{\gamma\cap w \neq \emptyset } \rho (w) \geq 1.\]
 The \emph{weighted $p$-mass} of $\rho$ in $\partial A$ is  
  \[WM_p^A(\rho) = \sum\limits_{w\in G_k^A} q(w)\rho (w)^p.\]

  \begin{déf} \label{defmodulepoid}  Let $k\geq 0$ and let  $\F^A$ be a set of curves contained in $\partial A$, we define the \emph{weighted $G^A_k$-combinatorial $p$-modulus of $\F^A$} by  \[\modcombwg{\F^A}  : =\inf \{ WM^A_p(\rho)\}, \]where the infimum is taken over the set of $\F^A$-admissible functions and with the convention \index{$\modcombwg{\cdot}$} $\modcombwg{\emptyset}=0$. For  simplicity, we  usually use the terminology \emph{weighted modulus}.\end{déf}

We can check that Proposition \ref{propmodbase} holds for  weighted modulus as well and  the proof is identical to the one for the usual combinatorial modulus.

This definition of the weighted modulus   strongly depends on the choice we have made for the approximation. In particular, it does not permit to compute a weighted modulus relatively to a generic approximation of $\partial A$.   As a consequence, an analogue to Proposition \ref{propdouble} would make no sense here. This is a huge restriction on the use  of the weighted modulus. Indeed, this proposition is essential in proving Proposition \ref{propscale} and Theorem \ref{theocourbedanspara} for  the usual combinatorial modulus. 
 
However, in the rest of the paper, the weighted modulus  will be used to compute inequalities for the usual combinatorial  modulus. As the usual combinatorial modulus, up to multiplicative constant,  does not depend on the choice of the approximation, this point will not be a problem for us.
 
The following proposition says that the weights are given by the types of the building-walls crossed by a minimal gallery. We recall that the group $\Gamma$ is given by the following presentation \[\Gamma = \left\langle s_i\in S \vert  s_i^{q_i}=1,  s_is_j=s_js_i \text{ if } v_i \sim v_j  \right\rangle. \]
  
\begin{prop} \label{prop poidshadows}
 Let $w\in G_k^A$ be such that  $w$ is a shadow $w=w_x$ for $x\in \Ch{A}$.  Let $x_0\sim_{s_1}  x_1 \sim_{s_2} \dots\sim_{s_{k-1}}  x_{k-1} \sim_{s_{k}}  x$ be a minimal gallery where $s_i$ is the generator of $\Gamma$ associated with the type of the building-wall between   $x_{i-1}$ and $x_{i}$ for every $1\leq i \leq k$. If $q_i$ is the order of $s_i$ for every $1\leq i \leq k$ then    \[q(w)=\prod_{i=1,\dots, k} q_i - 1.\]

\end{prop}

\begin{proof} 
Let $w\in G_k^A$ and $x\in \Ch{A}$ be such that $w=w_x$ in $\partial A$. Then we observe that $\{  v\subset \borg: \pi(v)=w\}=\{ v_y\subset \borg : \pi(y)=x\}$. As a consequence, we obtain   $q(w) = \#\pi^{-1}(x)$.

Now consider  the gallery $x_0\sim_{s_1}  x_1 \sim_{s_2} \dots\sim_{s_{k-1}}  x_{k-1} \sim_{s_{k}}  x$ given in the statement of proposition. Since $\pi$ preserves the types, $y\in \ch$ is in  $ \pi^{-1}(x)$ if and only if there exists a minimal gallery from $x_0$ to $y$ in $\Sigma$ of the form $x_0\sim_{s_1}  y_1 \sim_{s_2} \dots\sim_{s_{k-1}}  y_{k-1} \sim_{s_{k}}  y$.  Finally, we obtain $q(w)=\prod_{i=1,\dots, k} q_i - 1$.
  
\end{proof}

Thanks to the choices we have made, the weighted modulus is invariant up to a change of apartment in the following sense. For $B\in \apo$ consider the approximation $G_k^B$ given by Fact \ref{factchangeapartapprox}. To any element $w\in G_k^B$ we attach a weight    and define a \emph{weighted $G^B_k$-combinatorial $p$-modulus} as it is done in $\partial A$. Now let     $f: B\longrightarrow A$ be a type preserving isometry that fixes $x_0$  and denote $f: \partial B\longrightarrow \partial A$ the extension of this map to the boundary.
 The map  $f$ is realized by the restriction of the retraction $\pi$ to $B$. Thus $f$ preserves the weights.  Then the following fact is an immediate consequence of Fact \ref{factchangeapartapprox}.
 
\begin{fact} \label{factchangeapartmod}
Let  $B\in \apo$. Then for any $k\geq0$ and any set of curves $\F^B$ contained in $\partial B$ one has 
\[\mathrm{Mod}^B_{p}(\F^B,G^B_k)= \modcombwg{f(\F^B)}.\]
\end{fact}

Note that, for any $k\geq 0$ and any $w\in G_k^A$ one has \[ 1\leq q(w)\leq (q-1)^k \text{ with }  q:=\max\{q_1,\dots, q_n \}. \] Therefore for any  set of curves  $\F^A$   contained in $\partial A$, the next inequalities follow  directly from the definition 
\[\modcombapg{\F^A} \leq \modcombwg{\F^A}  \leq (q-1)^k \modcombapg{\F^A},\]where the   modulus in small letters  designates the usual \index{Modulus in the apartment}\index{$\modcombapg{\cdot}$} \emph{modulus    computed in $\partial A$}. In particular if  $\Gamma$ is of constant thickness $q\geq3$  then 
\[\modcombwg{\F^A}  = (q-1)^k \modcombapg{\F^A}. \] As a consequence,  at fixed scale $k\geq 0$, the weighted modulus depends only on the boundary of an apartment. We will discuss this  particular case in Sections \ref{sec applictethickeness} and \ref{sec result}.

  The following proposition is a major motivation of the definition of the weighted modulus.

\begin{prop} \label{propmajopoid}
Let $\F$ be a set of curves in $\borg$ and let $\F^A$ be a set of curves in $\partial A$ such that $\pi(\F)\subset \F^A$. Then
\[ \modcomb{p}{\F,G_k}\leq \modcombwg{\F^A}. \]
\end{prop}

\begin{proof}  Let $\rho_A$ be a $\F^A$-admissible function. We set $\rho : G_k \longrightarrow  [0,+\infty)$ defined by \[\rho(v) = \rho_A \circ \pi(v). \]
If $\gamma \in \F$, let $\gamma_A:=\pi \circ \gamma$. Then, as $\gamma_A\in \F^A$  \[L_{\rho}(\gamma) =  \sum\limits_{v\cap \gamma \neq \emptyset} \rho_A \circ \pi (v) \geq \sum\limits_{w\cap \gamma_A  \neq \emptyset} \rho_A  (w) \geq 1, \]
thus $\rho$ is $\F$-admissible.
Furthermore, one has: \[M_p(\rho) = \sum\limits_{v\in G_k} \rho_A \circ \pi(v)^p = \sum\limits_{w\in G_k^A} q(w)\cdot\rho_A (w)^p=WM_p^A(\rho_A). \]
With the first point it follows that $\modcomb{p}{\F,G_k}\leq \modcombwg{\F^A}$.
\end{proof}

\subsection{Modulus in $\borg$ compared with weighted modulus in $\partial A$}
\label{subsec comparaisonimappart}
  
As before $d_0>0$  is a small constant compared with $\dia{\borg}$ and with the constant of approximate self-similarity. 

We recall that the apartment $A\in \apo$ is fixed. Thanks to Fact \ref{factchangeapartmod} the following result hold for any apartment containing $x_0$.

In this subsection we continue to use the approximations $G_k$ and $G_k^A$ defined at the beginning of Subsection \ref{subsec choiceapprox}.
We recall that if  $\eta$ is  a non-constant curve of $\borg$,  the notation $\U_\epsilon(\eta)$ designates   the $\epsilon$-neighborhood of $\eta$ relative to the ${C}^0$ topology. If $\eta$ is a non-constant curve contained in $\partial A$, we use the notation\index{$\U_\epsilon^A(\eta)$} \[\U_\epsilon^A(\eta):= \{ \gamma \in \U_\epsilon(\eta) : \gamma \subset \partial A\}.\] The next theorem says that in this case, the modulus of $\U_\epsilon(\eta)$ in the boundary of the building is controlled by the weighted modulus of $\U_\epsilon^A(\eta)$ in the boundary of the apartment.   It is a key step in the proof of Theorem \ref{theoprincip}. 

\begin{theo}\label{theocontrolimmappart} Let $p\geq 1$, let $\eta \in \F_0$  and assume  $ \eta \subset \partial A$. For $\epsilon>0$ small enough so that the hypothesis of    Theorem \ref{theocourbedanspara} hold  in $\borg$, there exists a positive  constant  $C = C(p,\eta, \epsilon)$ independent of $k$ such that for every $k\geq 1$ large enough
\[ \modcomb{p}{\U_\epsilon(\eta),G_k }    \leq    \modcombwg{\mathcal{U}_\epsilon^A(\eta)}  \leq C \cdot  \modcomb{p}{\U_\epsilon(\eta),G_k }.\]
Furthermore, when $p$ belongs to a compact subset of $[ 1, + \infty)$ the constant  $C$ may be chosen independent of $p$.
\end{theo}

For the rest of the subsection $\eta\in \F_0$ and $\epsilon>0$ are as in  the hypothesis of the preceding  theorem. For each $\eta\in \F_0$ we fix a constant $r>0$ such that the hypothesis of  Theorem \ref{theocourbedanspara} are satisfied.  To prove the theorem we   need to introduce the following notation:
\liste{
\item $\mathrm{Aut}_\Sigma$ is the full group of type preserving isometries of $\Sigma$.
\item For $n\geq 0$, $B_n \subset \ch$ is the ball of center $x_0$ and of radius $n$ for the distance over the chambers $\dc{\cdot,\cdot}$.
\item For $n\geq 0$, $K_n < \mathrm{Aut}_\Sigma$ is the pointwise stabilizer of $B_n$ under the action of $\mathrm{Aut}_\Sigma$.
\item  $\F_n:= \{g \gamma \subset \borg : g \in K_n \text{ and } \gamma \in \U_\epsilon^A(\eta)\}$.

}

The main step to prove the theorem is to show that  $\F_n$ is an intermediate set of curves between $\U_\epsilon^A(\eta)$ and $\U_\epsilon(\eta)$. This will be done in Lemma \ref{lemcontrolmodulpoid}. Before proving this, we need to discuss the action of $K_n$ on the chambers. The next lemma uses ideas of \cite[Lemma 3.5 and Proposition 8.1]{CapraceAutomRightAngled}.
 
 \begin{lem} \label{lemcardiorbites}
There exists an integer $N>0$ depending only on $n$ and satisfying the following property. Let $x\in \ch$, set $\dc{x_0,x}=k$ and assume $k>n$. Let  \[x_0\sim_{s_1}  x_1 \sim_{s_2} \dots\sim_{s_{n}}x_n\sim_{s_{n+1}}  \dots\sim_{s_{k-1}}  x_{k-1} \sim_{s_{k}}  x\] be a minimal gallery where  $s_1, \dots, s_k$ is the family of types of the building-walls crossed by this gallery.  Then   \[\frac{1}{(q-1)^N}\cdot\prod_{i=n+1}^k q_i -1\leq \# K_n.x \leq \prod_{i=n+1}^k q_i -1,\]
where $ q:=\max\{q_1,\dots, q_n \}$.
 \end{lem}
 
 \begin{proof}
 Since $K_n$ preserves the types and fixes $x_0,\dots, x_n$ it follows that 
 \[\# K_n.x \leq \prod_{i=n+1}^k q_i -1.\]
 
 Now for $D\in \dialgp$ we write $U(D)$ the  pointwise  stabilizer of $D$ under the action of $\mathrm{Aut}_\Sigma$ and we set 
 \[U(n)= \left\langle U(D) \vert B_n \subset \Ch{D} \right\rangle.\]
 Clearly $U(n)<K_n$ and 
 \[\#K_n.x \geq\#U(n).x .\]
 
Now if we write $M_{i}$ the building-wall between $x_i$ and $x_{i+1}$, we observe that the orbit of $x_{i+1}$ under $U(D_0(M_{i}))$ has $q_i-1$ elements. Indeed, $U(D_0(M_{i}))$ acts as the full group of permutations on the set $\{D_1(M_{i}),\dots,D_{q_i-1}(M_{i})  \}$.

 Note that  $U(D_0(M_{i}))<U(n)$ if and only if $B_n\subset \Ch{D_0(M_{i})}$. Otherwise $M_i$ crosses $B_n$, because $x_0 \in \Ch{D_0(M_{i})}$. As a consequence, we set $N$  the number of building-walls that cross $B_n$  and we obtain \[\#U(n).x\geq\frac{1}{(q-1)^N}\cdot\prod_{i=n+1}^k q_i -1.\]
This achieves the proof.
 
 \end{proof}
 
 Now we can prove the main lemma.
 
\begin{lem} \label{lemcontrolmodulpoid}
Let $p\geq 1$. For $n\geq 0$ large enough, there exist two positive  constants $C_1$, $C_2$ depending only on $p$, $\eta$, $\epsilon$, such that for every $k>n$: \[\modcombwg{\mathcal{U}_\epsilon^A(\eta)}\leq C_1\cdot \modcomb{p}{\F_n,G_k }\leq C_2\cdot \modcomb{p}{\U_\epsilon(\eta),G_k}. \]
Furthermore, when $p$ belongs to a compact subset of $[1, + \infty)$ the constants   may be chosen independent of $p$.
\end{lem}

\begin{proof}
i) First we prove  the right-hand side inequality. According to Proposition \ref{prop prodgromov}, for any $g \in K_n$ and any $\xi \in \borg$,  $d(\xi,g\xi) \leq A. e^{-\alpha n}$ where $A$ is the visual multiplicative constant. Hence for $n\geq 0$ large enough, by triangular inequality, $\F_n  \subset \mathcal{U}_{2\epsilon}(\eta)$. Then, as a consequence of   Theorem \ref{theocourbedanspara}, for a fixed $r>0$ the combinatorial modulus of $\U_\epsilon(\eta)$ is controlled  by the combinatorial modulus of $\U_{2\epsilon}(\eta)$ with multiplicative constants depending only on $p$, $\eta$, $\epsilon$. Thus, by  Proposition \ref{propmodbase} (1), there exists $C=C(p,\eta,\epsilon)$ such that \[ \modcomb{p}{\F_n,G_k }\leq C\cdot \modcomb{p}{\U_\epsilon(\eta),G_k}.\]  

ii)  Now we fix an integer $n\geq 0$ large enough so that the first part of the proof holds. We use the notation $K:=K_n$ for   simplicity. Moreover we assume that $k>n$.
Let $\rho : G_k \longrightarrow [0,+\infty)$ be a minimal $\F_n$-admissible function and set $ \rho_A : G_k^A \longrightarrow [0,+\infty)$ the function defined by:\[ \rho_A (w) = \int_K \rho (g\widetilde{w}) d\mu(g),\]
where $\mu$ denotes the Haar  probability measure over $K$ and where the function $:w\in G_k^A \longrightarrow \widetilde{w}\in G_k$ is given by Fact \ref{fact approximapart}. Let $w\in G_k^A$ and let $x\in \ch$ be such that $v_x=\widetilde{w}$. Then $\dc{x_0,x}=k$. As in  Proposition \ref{prop poidshadows}, let \[x_0\sim_{s_1}  x_1 \sim_{s_2} \dots\sim_{s_{n}}x_n\sim_{s_{n+1}}  \dots\sim_{s_{k-1}}  x_{k-1} \sim_{s_{k}}  x\] be a minimal gallery where  $s_1, \dots, s_k$ is the family of types of the building-walls crossed by this gallery. We set \[q(w,n)=\prod_{i=n+1,\dots, k} q_i -1 .\]

We notice that for any $g\in K$  the  translated $g  \widetilde{w}=gv_x$ is the shadow $v_{gx}$. In particular, this means that $\# K.\widetilde{w}=\# K.x$. Then   according to Lemma \ref{lemcardiorbites} 
\[(*) \ \ \ \ \  \frac{q(w,n)}{(q-1)^N}\leq \#K.\widetilde{w}\leq q(w,n),\]
where $q:=\max\{q_1,\dots, q_n \}$ and $N$ is the number of building-walls crossing $B_n$.

As a consequence   we can write \[\rho_A (w) = \frac{1}{\#K.\widetilde{w}} \cdot \sum\limits_{v\in K.\widetilde{w}} \rho(v), \] 
 and we prove the second inequality of the proposition.
 
On the one hand, let $\gamma \in \U_\epsilon^A(\eta)$:
\[L_{\rho_A} (\gamma ) = \sum\limits_{w\cap\gamma \neq \emptyset} \int_K \rho (g  \widetilde{w}) d\mu(g) =  \int_K \sum\limits_{w\cap\gamma\neq \emptyset} \rho (g  \widetilde{w}) d\mu(g)  =  \int_K \sum\limits_{v\cap g (\gamma)\neq \emptyset} \rho (v) d\mu(g). \]
Since  $ g (\gamma ) \in \F_n$, we get $\sum\limits_{v\cap g (\gamma)\neq \emptyset} \rho ( v)\geq 1$ and $\rho_A$ is $\F_A$-admissible.

On the other hand, thanks to  Jensen's inequality,  for $p\geq 1$ one has:  \[WM^A_p(\rho_A) \leq \sum\limits_{w \in G_k^A} q(w) \int_K \rho (g  \widetilde{w})^p d\mu(g) = \sum\limits_{w \in G_k^A}  \frac{q(w)}{\#K.\widetilde{w}} \cdot \sum\limits_{v \in K. \widetilde{w}} \rho (v)^p.\]
Hence with $(*)$ we obtain \[WM^A_p(\rho_A)\leq \sum\limits_{w \in G_k^A}  (q-1)^N\cdot\frac{q(w)}{q(w,n)} \cdot \sum\limits_{v \in K. \widetilde{w}} \rho (v)^p\leq (q-1)^{n+N} M_p(\rho).\]
 Finally we get: \[ \modcombwg{\mathcal{U}_\epsilon^A(\eta)} \leq (q-1)^{n+N} \modcomb{p}{\F_n,G_k}.\]
 
This last multiplicative constant depends only on $n$ and on the geometry of the building. Since $n$ depends only on  $\eta$ and $\epsilon$  the second inequality is proved.

iii) The last statement of the lemma is an immediate consequence of the two first parts and of Theorem \ref{theocourbedanspara}.
 
\end{proof}
 
\begin{proof}[Proof of Theorem \ref{theocontrolimmappart}]
Since $\pi(\mathcal{U}_\epsilon(\eta)) \subset \mathcal{U}_\epsilon^A(\eta)$,   Proposition \ref{propmajopoid} and  Lemma \ref{lemcontrolmodulpoid}   imply the theorem.

\end{proof}
    
\subsection{Consequences}
\label{subsec conseq}

Here we   continue to use the approximations $G_k$ and $G_k^A$ defined in   Subsection \ref{subsec choiceapprox}. For   $\eta$ a non-constant curve in $\partial A$, $\partial Q$ a parabolic limit set in $\partial A$, and $\delta,r, \epsilon >0$, we use the following notation:
\liste{
\item $\F_0^A=\{\gamma \in \F_0 : \gamma \subset \partial A\}$,
\item $\F_{\delta,r}^A(\partial Q)$ is the subset of $\F_0^A$ consisting of all curves $\gamma$ satisfying:\liste{\item $\gamma \subset N_\delta(\partial Q)$, \item  $\gamma \not\subset N_r(\partial Q')$  for any connected parabolic limit set $\partial Q' \varsubsetneq \partial Q $,}

\item $\delta_0(\cdot)$   refers to the increasing function  in   Theorem \ref{theocourbedanspara}.}

We recall that the apartment $A\in \apo$ is fixed. Thanks to Fact \ref{factchangeapartmod}, the following result  holds for any apartment containing $x_0$.

The main consequence of the previous subsection is    Theorem \ref{theomajomodul}, where we control from above the combinatorial modulus of $\F_0$ by the weighted-combinatorial modulus of  $\F_0^A$.

\begin{lem} \label{lemmodulemajo}
Let $p\geq 1$ and $A\in \apo$.  Let $\partial P$ be a parabolic limit set in $\borg$ and assume that $x_0\subset \cv{\partial P}$. Let $\gamma$ be a non-constant curve in $\partial Q= \partial P \cap \partial A$ such that $\partial Q$ is the smallest parabolic limit set of $\partial A$ containing $\gamma$. Let $r>0$ be  small enough so that $\gamma \not\subset \overline{N_r}(\partial Q') $ for any connected parabolic limit set $\partial Q'\varsubsetneq \partial Q$. Let $\delta <\delta_0(r)$. Then for $\epsilon>0$   small enough,   there exists a constant $C=C(p,\gamma,r, \epsilon)$ such that for every $k\geq 1$ \[\modcomb{p}{\F_{\delta,r}(\partial P),G_k}\leq C \cdot \modcomb{p}{\mathcal{U}_\epsilon (\gamma),G_k} \leq C \cdot \modcombwg{\F^A_{\delta,r}(\partial Q) }.\]
 In particular 
 \[\modcomb{p}{\F_{\delta,r}(\partial P),G_k}  \leq C \cdot\modcombwg{\F^A_0 }.\]
Furthermore, when $p$ belongs to a compact subset of $[1, + \infty)$ the constant  $C$ may be chosen independent of $p$.  
\end{lem} 

\begin{proof} As in the beginning  of the proof of Theorem \ref{theocourbedanspara} we can assume,  without loss of generality,   that for $\epsilon>0$ small enough 

\liste{\item  $x_0 \subset \cv{\gamma}$ for every $\gamma \in \F_0$,  
\item $\mathcal{U}_\epsilon(\gamma)\subset \F_0$. } 
As before, the multiplicative constants resulting from these assumptions only depends on $d_0$.
 
 Now for $\epsilon>0$ small enough, we obtain by Proposition \ref{propmodbase}(1) and the preceding assumption
\[ \modcombwg{\mathcal{U}^A_\epsilon (\gamma)} \leq   \modcombwg{\F^A_{\delta,r}(\partial Q) } \leq  \modcombwg{\F^A_0 }.\]
 Since $\pi(\ \mathcal{U}_\epsilon (\gamma)) \subset \mathcal{U}^A_\epsilon (\gamma)$, with   Proposition \ref{propmajopoid} one has\[\modcomb{p}{\mathcal{U}_\epsilon (\gamma),G_k } \leq  \modcombwg{\mathcal{U}^A_\epsilon (\gamma)}.\]
Finally thanks to     Theorem \ref{theocourbedanspara} there exists $C=C(p,\gamma,r, \epsilon)$ such that for every $k\geq 1$\[\modcomb{p}{\F_{\delta,r}(\partial P),G_k}\leq C \cdot \modcomb{p}{\mathcal{U}_\epsilon (\gamma),G_k}.\]
\end{proof}

Now we are ready to prove the  following theorem which is used in  the proof of Theorem \ref{theoprincip}. 

\begin{theo} \label{theomajomodul}
For any $p\geq 1$, there exists a constant $D=D(p)$ such that for every $k\geq 1$
\[ \modcomb{p}{\F_0,G_k}  \leq D\cdot  \modcombwg{\F^A_0}.\]
\end{theo}
\begin{proof}
 
i) First, we recall the following notation. For $\delta,r >0$ and for  $\partial P$   a connected parabolic limit set, $\F_{\delta,r}(\partial P)$ is the set of all the curves in $\F_0$ satisfying:  \liste{ 
\item $\gamma \subset N_\delta(\partial P)$,
\item $\gamma \not\subset N_r(\partial P')$ for any connected parabolic limit set $\partial P' \varsubsetneq \partial P$.}

Now, as it is done in \cite[as a remark of Corollary 6.2.]{BourdonKleinerCLP} in boundaries of Coxeter groups, we  observe that $\F_0$ splits into a finite   union   \[\F_0 =\F_{\delta_1, r_1}(\partial P_1) \cup \dots \cup\F_{\delta_N, r_N}(\partial P_N)\]
 with $\delta_i< \delta_0(r_i)$. 
 
Indeed, let  $\mathcal{P}=\{\partial P_1, \dots, \partial P_N\}$ be the finite set (see Proposition \ref{proptaillepara}) of all the parabolic limit sets of diameter larger than $d_0$. For $\partial P \in \mathcal{P}$ we call      \emph{height} of  $\partial P$ the maximal length of a sequence in $\mathcal{P}$  of the form
\[ \partial P'_0 \varsubsetneq  \partial P'_1 \varsubsetneq \cdots \varsubsetneq \partial P'_i = \partial P.\] 
 Now, we  index $\mathcal{P}$ thanks to the height
 \[\mathcal{P}=\{\partial P_{0,1}, \dots, \partial P_{0,N_0}, \dots, \partial P_{i,j}, \dots, \partial P_{M,1},  \dots, \partial P_{M,N_M} \}\]
 where $i$ is the height of $\partial P_{i,j}$. We fix a small $r_0>0$ and $\delta_0< \delta_0(r_0)$. Then by induction on the height we set 
\[r_{i+1} =\delta_i \text{ and } \delta_{i+1} < \min \{\delta_0(r_{i+1}),  \delta_{i}\}. \]

Now let  $\partial P \in \mathcal{P}$ be of height $i>0$. By  construction, for any $i'<i$ we have  $r_i<\delta_{i'}$. Hence for any $ \partial P' \varsubsetneq \partial P$ of height $i'$ we have  $N_{r_i}(\partial P') \subset N_{\delta_{i'}}(\partial P')$ and \[\F_0=\bigcup^M_{i=0} \bigcup^{N_i}_{j=1} \F_{\delta_i, r_i}(\partial P_{i,j}). \]

ii) Let $\partial P$ be one of the  parabolic limit sets appearing in the preceding decomposition of $\F_0$ and $\delta,r>0$ be the corresponding constants. As in the beginning of the proof of Theorem \ref{theocourbedanspara},  we can assume that $x_0 \subset\cv{\partial P}$. Again the multiplicative constant resulting from this assumption only depends on $d_0$. Now pick $B\in \apo$ such that $\partial B \cap \partial P\neq \emptyset$ and fix a curve $\gamma$ and $\epsilon >0$ so that the hypothesis of Lemma \ref{lemmodulemajo} are satisfied. Then there exists   $C= C(p,\gamma,r,\epsilon)$  such that for every $k\geq 1$ \[\modcomb{p}{\F_{\delta,r}(\partial P),G_k}  \leq C \cdot\mathrm{Mod}^B_p(\F_0^B,G^B_0).\]

 iii) With    Fact \ref{factchangeapartmod}, we observe that  the weighted modulus on the right-hand side of the preceding inequality is independent of the choice of $B\in \apo$.   Finally, by   Proposition \ref{propmodbase} (2) and the two first parts of this proof, there exists a constant $D=D(p)$ such that  such that for every $k\geq 1$  \[\modcomb{p}{\F_0,G_k}  \leq   D \cdot\modcombwg{\F^A_0 }.\]
\end{proof} 
    
Note that for the moment we cannot prove a converse inequality between the modulus. Indeed, in the proof of  Lemma \ref{lemmodulemajo}   the use of   Theorem \ref{theocourbedanspara}   is a key point. As we said before, we cannot  prove  an analogue of Theorem \ref{theocourbedanspara} for the weighted modulus.

Nevertheless, we can define a critical exponent in connection with the weighted modulus as it is done in   Subsection \ref{subsecdimconf}. Then Theorem \ref{theomajomodul} can be used to  helps us   understand this new critical exponent.

\begin{prop}
There exists $p_0\geq 1$ such that  for $p\geq p_0$ the weighted modulus $\modcombw{p}{\F_0^A,G^A_{k}} $ goes to zero as $k$ goes to infinity.

\end{prop}

\begin{proof}
 This proof is the same as the proof of Proposition \ref{propdefexpocrit}.
We recall that  $\kappa$ is the constant of the approximations $\{G_k\}_{k\geq 0}$ and  $\{G^A_k\}_{k\geq 0}$.
 
 According to the doubling condition and the definition of an approximation, there exists an integer $N'$ such that each element $w\in G^A_k$ is covered by at most $N'$ elements of $G^A_{k+1}$. As a consequence, if   $K>0$ is the cardinality $G_0$, then \[\# G^A_k \leq K \cdot N'^k \text{ for any } k \geq 1.\]
 
Moreover, as we saw in the proof of Proposition \ref{prop minofonctionmini}, there exists a constant $K'>0$ such that the constant function $\rho : v\in G^A_k \longrightarrow \rho(v) = K'\cdot 2^{-k}\in [0,+\infty)$ is $\F^A_0$-admissible.

As a consequence \[\modcombapfo\leq C \cdot \Big(\frac{N'}{2^p}\Big)^k,\]
where $C$ is a positive constant. Hence we obtain 
\[\modcombwfo\leq (q-1)^k \cdot \modcombapfo\leq C \cdot \Big(\frac{(q-1)N'}{2^p}\Big)^k,\]
 Thus, for $p$ large enough, $\modcombwfo$ goes to zero.
\end{proof}

It is now natural to define a critical exponent for the weighted modulus in the apartment.  
 
\begin{déf} \label{defexpocritpoid}
The \emph{critical exponent $Q_W$} \index{Critical exponent} of the weighted modulus in $\partial A$ is defined as follows  \[Q_W = \inf \{p\in [1, +\infty) : \lim_{k\rightarrow  + \infty} \modcombwg{\F^A_0 }= 0\}.\] \index{$Q_W$}
\end{déf}  
To avoid confusion, we   use the following notation \liste{
\item $Q$ for the critical exponent associated with the usual modulus $\modcombg{\cdot}$ in $\borg$\index{$Q$},
\item $Q_A$ for the critical exponent associated with the usual modulus \index{$Q_A$} $\modcombapg{\cdot} $ in $\partial A$,
\item $Q_W$ for the critical exponent associated with the weighted modulus $\modcombwg{\cdot}$ in $\partial A$. }

 We recall that $Q$ and $Q_A$ are respectively the conformal dimension of $\borg$ and of $\partial A\simeq \partial W$ (see Theorem \ref{theocarrasco}).  The inequalities between the different modulus imply the following corollary.

    \begin{coro} 
    \label{coroexposentcritique} The following inequalities hold
  \[Q_A \leq Q \leq Q_W.\] 
   \end{coro}
    
 \begin{proof}
 With  Proposition \ref{propmodbase} (1) and Theorem \ref{theomajomodul}, one has \[\modcombapfo\leq \modcombfo \leq  D \cdot\modcombwg{\F^A_0 }.\]
The inequalities between the critical exponents follow. \end{proof}

\section{Application to buildings of constant thickness}
\label{sec applictethickeness}
 
In this section we use the notation and the assumptions from the previous section.  In particular,  the   self-similar metric on $\borg$ is $d(\cdot,\cdot)$. We fix $d_0$   a small constant compared with $\dia{\borg}$ and with the constant of approximate self-similarity.  Then $\F_0$ is the set of curves of diameter larger than $d_0$.   The notation $\delta_0(\cdot)$ still refers to the increasing function   in   Theorem \ref{theocourbedanspara}.

  As before we fix an apartment $A\in \apo$  and $\F^A_0$ is the set of curves in $\partial A$ of diameter larger than $d_0$. 
  
  We assume that $\Sigma$ is of constant thickness $q\geq3$. This means that $\Gamma$ is the graph product given by the pair $(\mathcal{G}, \{\Z/q\Z\}_{i=1,\dots,n})$. 
As before $\{G_k^A\}_{k\geq 0}$ and $\{G_k\}_{k\geq 0}$ are the approximations of $\partial A$ and $\borg$ provided by Fact \ref{fact approximapart}.  We already noticed that, with the constant thickness assumption, we obtain for  $k\geq 0$ and    $\F^A$   a set of curves contained in $\partial A$  \[\modcombwg{\F^A} =q^k \modcombapg{\F^A}, \]where the modulus in small letters designates the usual modulus computed in $\partial A$. In particular, this means that  from  Theorem \ref{theocourbedanspara} applied  to $\modcombapg{\cdot}$ we can obtain analogous inequalities for   $\modcombwg{\cdot}$.

 The constant thickness allows us to  control by bellow the combinatorial modulus of $\F_0$ by the weighted-combinatorial modulus of  $\F_0^A$. Combining with  the results of Subsection \ref{subsec conseq}, we obtain a full control of the two modulus.

  \begin{theo} \label{theomajomodulreciproque}
For any $p\geq 1$, there exists a constant $D=D(p)$ such that for every $k \geq 1$
\[D^{-1}\cdot \modcombwfo  \leq   \modcomb{p}{\F_0,G_k}  \leq D\cdot  \modcombwg{\F^A_0}.\] 
In particular $Q_W= Q$.
\end{theo}
\begin{proof}
 The right-hand side inequality is given by Theorem \ref{theomajomodul}. The proof is almost the same  for the left-hand side inequality. Indeed, $\F^A_0$ admits a decomposition analogous to the decomposition used at the beginning of the proof of   Theorem \ref{theomajomodul}. With a fixed such decomposition and with  Proposition \ref{propmodbase} (2), it is sufficient to prove that for any parabolic limit set $\partial Q \subset \partial A$ and any $\delta,r>0$ with $\delta<\delta_0(r)$, there exists a constant $C=C(p,\partial Q,r)$   such that for every $k\geq 1$
\[\modcombwg{\F_{\delta,r}^A(\partial Q)}  \leq C \cdot  \modcombfo.  \] 
To this end, fix $\eta$ a non-constant curve in $\partial Q$ and  $\epsilon>0$ such that  the hypotheses of   Theorem \ref{theocourbedanspara} in $\partial A$ and of   Theorem \ref{theocontrolimmappart} are satisfied.  
  Then  there exist two constants $K $ and $K'$ depending only on $p$, $\eta$, $r$ and $\epsilon$ such that for every $k\geq 1$
  \[\modcombwg{\F_{\delta,r}^A(\partial Q)}  \leq K \cdot  \modcombwg{\mathcal{U}^A_\epsilon (\eta)}\leq K' \cdot  \modcombg{\mathcal{U}_\epsilon (\eta)}.\]
  Finally, as in the beginning the proof of Theorem \ref{theocourbedanspara} we can assume without loss of generality that for $\epsilon >0$ small enough   $\mathcal{U}_\epsilon(\eta)\subset \F_0$.   Again the multiplicative constant resulting from this assumption only depends on $d_0$. This assumption and Proposition \ref{propmodbase} (1) provide the desired inequality.
 
The equality of the critical exponents is an immediate consequence of the inequalities between the modulus. 
   \end{proof}
 
\begin{rem} In the case where $\Sigma$ is a right-angled Fuchsian building of constant thickness, M. Bourdon gave the explicit value of the conformal dimension of $\borg$.
\begin{theo}[\cite{BourdonImHyperDimConfRigi}] Let $\Gamma$ be the  graph product associated  with a pair $(C_n,\{\Z/q\Z\}_{i=1,\dots, n})$ where $C_n$ is a cyclic graph of length $n\geq 5$ and $q\geq 2$, then \[\confdim{ \borg}= 1+ \frac{\log(q-1)}{\mathrm{Arg} \cosh \frac{n-2}{2}} .\]\end{theo} 

\end{rem}

\section{Dimension 3 and 4 right-angled buildings with boundary satisfying the CLP}
\label{sec result}

In a well chosen case, the symmetries of the Davis chamber, that extend to the boundary of an apartment, provide a strong control of the weighted modulus. This lead to the proof of the main theorem of this  article. 
  
 Here we still assume that $\Gamma$ is of constant thickness $q\geq3$. As usual,  $W$ is the Coxeter group, associated with $\Gamma$. As before $\{G_k^A\}_{k\geq 0}$ and $\{G_k\}_{k\geq 0}$  are the approximations of $\partial A$ and $\borg$ provided by Fact \ref{fact approximapart}.

  In this subsection, we assume that $W$ is the group generated by the reflections about the faces of a compact right-angled polytope $D\subset \h^d$. 
  
 Now we write $\mathrm{Ref}(\h^d)$ the group generated by all the hyperbolic reflections in $\h^d$. In the following we designate by $\mathrm{Ref}(D)$ the stabilizer of $D$ under the action of $\mathrm{Ref}(\h^d)$.

  Now, with additional assumptions on the regularity of $D$, we prove that $\borg$ satisfies the CLP.
 
  \begin{theo} \label{theoprincip}
 Let $\Gamma$ be a graph product of constant thickness $q\geq3$. Assume that $W$ is the group generated by the reflections about the faces of a compact right-angled polytope $D\subset \h^d$. Moreover, assume   that the quotient of $D$ by $R_{ef}(D)$ is a simplex in $\h^d$. Then $\partial \Gamma$ satisfies the CLP. \end{theo}

Now we assume that the hypotheses of the preceding theorem hold and we  use the following notation.

\begin{nota}
\text{  }
\liste{\item $T$ is the hyperbolic simplex in $\h^d$ isometric to $D/R_{ef}(D)$.
     \item $W_T$ is the hyperbolic reflection group generated by the reflections about the codimension 1 faces of $T$.
}
\end{nota} 
  
We notice that $W$ is a finite index subgroup of $W_T$. Indeed, $W$ is a subgroup of $W_T$ and they both act discretely on $\h^d$ with finite co-volume.
Then $W_T$ acts by polyhedral isometries on an apartment of $\Sigma$. Indeed,  a reflection about a face  of $T$ either preserves $D$, or is a reflection about a face of $D$. In particular, it  preserves the tilling of $\h^d$ by $D$.

Thanks to  the constant thickness and the results in the preceding section, we only need to study the usual combinatorial modulus in the apartment to prove the theorem.

\begin{lem} \label{lemcox}
Let $p\geq 1$ and let $A\in \apo$. Let $\eta$ be a non-constant curve in $\partial A$. Then there exists a constant $C = C(p,\eta,\epsilon)$ such that for every $k\geq 1$ \[\modcombapg{\F^A_0}\leq C \cdot \modcombapg{\mathcal{U}^A_\epsilon (\eta)}.\]
Furthermore, when $p$ belongs to a compact subset of $[1, + \infty)$ the constant  $C$ may be chosen independent of $p$.
\end{lem}
\begin{figure}[ht!]\labellist
\small\hair 2pt
 
\pinlabel $\alpha$  at 510 400 
 \pinlabel $\beta$  at 655 530 
  \pinlabel $\gamma$  at 460 730 
  \pinlabel $\delta$  at 655 310 
    \pinlabel $\theta$  at 905 190 
   \pinlabel $\omega$  at 855 360 
\endlabellist
\centering  \includegraphics[height=4cm]{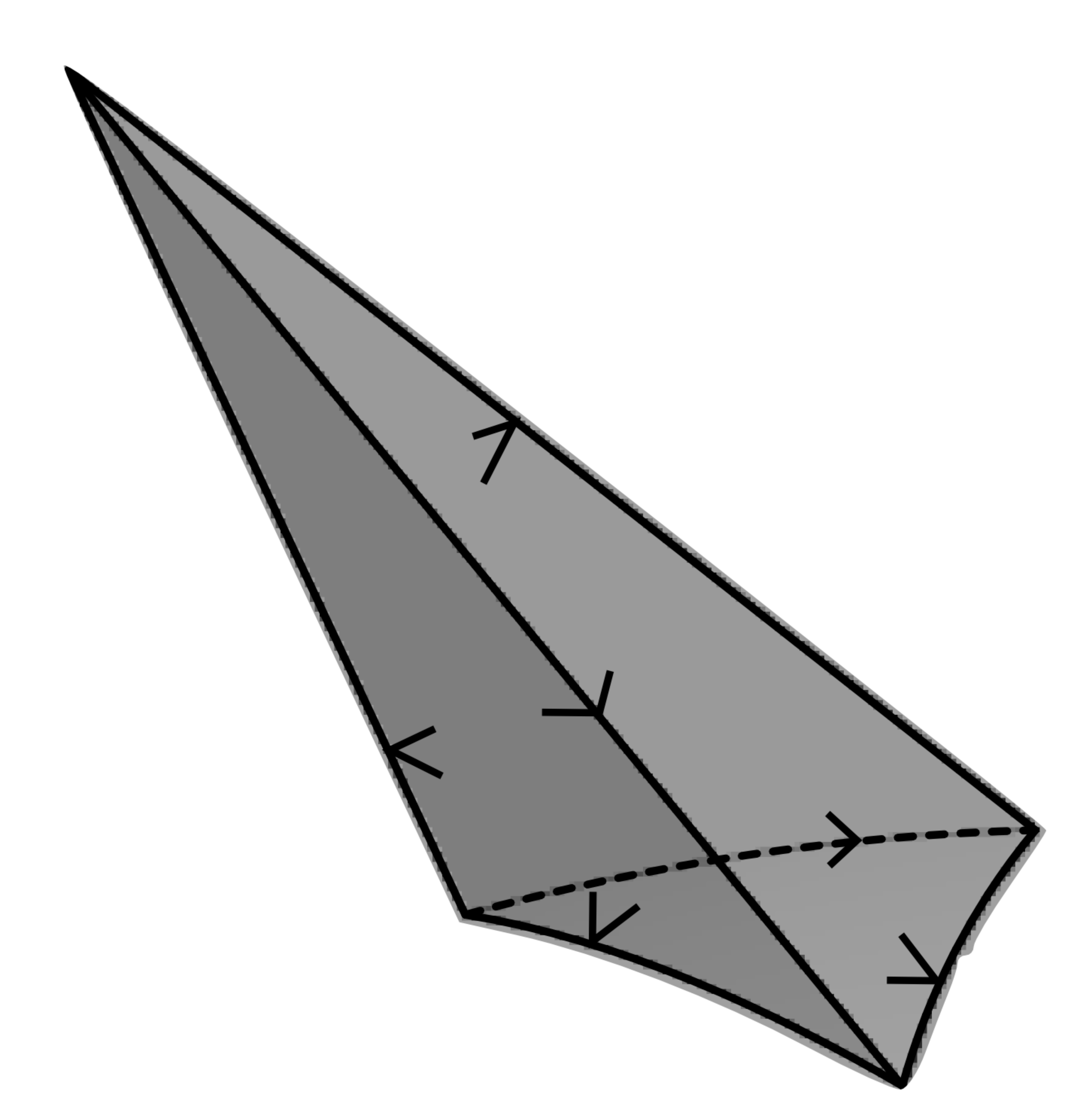}\caption{If $D$ is a dodecahedron, $T$ is the hyperbolic tetrahedron with dihedral angles $\alpha = \pi/5, \beta=\pi/3, \gamma=\delta=\omega= \pi/2$ and $\theta=\pi/4,$ }\label{fig:D24}\end{figure}
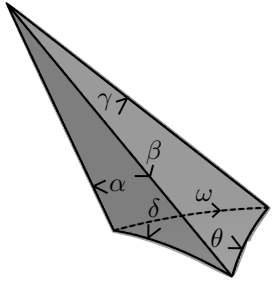

\begin{proof} To prove this lemma, we first use the fact that
 $\partial W_T$ is identified with $\partial A$ as  $W_T$ acts geometrically on $A$. Hence  the combinatorial visual metric on $\partial A$ defines a  self-similar metric $d_{W_T}$ on $\partial W_T$. A $\kappa$-approximation $\{G_k^A\}_{k\geq 0}$ of $\partial A$ induces a $\kappa$-approximation on $\partial W_T$ with same the modulus. 

 On the other hand,  $\partial W_T$ contains no proper parabolic limit set. Indeed, the Coxeter polytope $T$ of $W_T$ is a simplex so all the proper parabolic subgroups of $W_T$ are finite.   In particular, for any non-constant curve $\eta\subset \partial W_T$, the smallest parabolic subset containing $\eta$ is  $\partial W_T$. 

 As a consequence, by \cite[Corollary 6.2.]{BourdonKleinerCLP}    we get that for every $\epsilon>0$, there exists $C = C(p,\eta,\epsilon)$ such that for every $k\geq 1$
 \[\modcombapg{\F_0^A} \leq C \cdot \modcombapg{\mathcal{U}^A_\epsilon (\eta)}.\]
 The  Corollary 6.2 in \cite{BourdonKleinerCLP}  is the equivalent for Coxeter groups of our Theorem \ref{theocourbedanspara} for graph products. 
\end{proof}
 
\begin{proof}[Proof of Theorem \ \ref{theoprincip}.]
 We  check that the hypotheses  of Proposition \ref{propcritCLP} are satisfied. 
 To prove that  $\modcomb{1}{\F_0,G_k}$ is unbounded, it is enough to prove that there exist   $N$ disjoint curves of diameter larger than $d_0$ in $\borg$ for every $N \in \N$ as we did at the beginning of the proof of Theorem \ref{theoapplic au imm de dim2}.
 
Now we let $p\geq 1$, $\eta$ be a non-constant curve in $\borg$, and $\epsilon>0$.  Without loss of generality, we can assume that there exists $A\in \ap$  such that $\eta\subset \partial A$. Indeed, as a consequence of Theorem \ref{theocourbedanspara} if $\partial P$ is  the smallest parabolic limit set containing $\eta$. For $\eta'$ a curve such that $\partial P$ is  the smallest parabolic limit set containing $\eta'$, then, up to a multiplicative constant, the behavior of $\modcombg{\mathcal{U}_\epsilon(\eta)}$ and $\modcombg{\mathcal{U}_{\epsilon}(\eta')}$ when $k$ goes to infinity are the same. The multiplicative constant resulting from this assumption only depends on $\eta$.
 
 Using the same arguments as in the beginning of the proof of Theorem \ref{theocourbedanspara}, we can also assume,  without loss of generality, that for $\epsilon >0$ small enough   
 \liste{\item $\mathcal{U}_\epsilon(\eta)\subset \F_0$,
 \item   $x_0  \in \Ch{A}$. }
Again the multiplicative constant resulting from these assumptions only depends on $d_0$.  
 
Then, thanks to  the constant  thickness, the inequality of   Lemma \ref{lemcox} becomes 
\[\modcombwfo\leq C \cdot \modcombwg{\mathcal{U}^A_\epsilon (\eta)},\]
where $C$ depends only on $p$, $\eta$ and $\epsilon$.

 Finally, it is enough to apply   Theorem \ref{theomajomodul} to the left-hand term  and   Theorem \ref{theocontrolimmappart} to the right-hand term of the previous inequality to complete the proof.
\end{proof}

  \begin{coro} \label{coroprincip}
 Let $\Sigma$ be a building of constant thickness   $q\geq 3$. Assume that the Coxeter group of $\Sigma$ is the reflection group of the right-angled dodecahedron in $\mathbb{H}^3$ or  the reflection group of the right-angled 120-cells in $\mathbb{H}^4$, then $\partial \Sigma$ satisfies the CLP.  
  \end{coro} 
  
  \begin{rem}The hyperbolic $120$-cell was described by H.S.M. Coxeter in \cite{CoxeterPolyt} (see also \cite[Appendix B.2.]{DavisBook}). It has been used by M.W. Davis to  build a compact hyperbolic $4$-manifold in \cite{DavisManif}.\end{rem}
  
\printindex
  \bibliography{Biblio}
  \bibliographystyle{alpha}
               
\end{document}